\documentclass[12pt]{amsart}
\usepackage{amsfonts}
\usepackage{dsfont}
\usepackage{bbm}
\usepackage{mathrsfs}
\usepackage{color}
\usepackage{graphicx,amsmath,amsfonts,amssymb,amscd,amsthm}
\usepackage{hyperref}
\usepackage[numbers,sort&compress]{natbib}
\allowdisplaybreaks[4]
\hypersetup{colorlinks,linkcolor=blue,citecolor=blue}
\numberwithin{equation}{section}

\def\EquationsBySection{\def\theequation
{\thesection.\arabic{equation}}
\@addtoreset{equation}{section}}

\newcommand\old[1]{}

\newtheorem{theorem}{Theorem}[section]
\newtheorem{lemma}{Lemma}[section]

\newtheorem{proposition}{Proposition}[section]
\newtheorem{remark}{Remark}[section]

 \newcommand{\Bp}{\begin{proof}}
 \newcommand{\Ep}{\end{proof}}

\renewcommand{\theequation}{\thesection.\arabic{equation}}
\newcommand{\dif}{\mathrm{d}}

 \newcommand{\R}{\mathbb{R}}

 \def\p{\partial}
 \def\m{m}
 \newcommand{\beq}{\begin{equation}}
\newcommand{\eeq}{\end{equation}}
\newcommand{\I}{\mathbf{I}}

\title{Degenerate bifurcation of two-fold doubly-connected uniformly rotating vortex patches}	

\author{Yuchen Wang${}^{1,2}$}
\address{School of Mathematics, Tianjin Normal Univeristy \\ Tianjin, China \& School of Mathematics and Statistics, Central China Normal University\\ Wuhan 430079, Hubei, China }
\email{wangyuchen@mail.nankai.edu.cn}
\author{Xin Xu${}^3$}
\address{South China Normal University, Guangdong, China}
\email{xuxin1994pkq@163.com}

\author{Maolin Zhou${}^4$}
\address{Chern Institute of Mathematics and LPMC, Nankai University, Tianjin, China}
\email{zhouml123@nankai.edu.cn}

\begin{document}

\maketitle

%\author[a]{Yuchen Wang}
%\affil[a]

%\thanks{${}^\dagger$ Partially supported by the National Science Foundation of China No. 11831009 and the funding of innovating activities in Science and Technology of Hubei Province. }
%\and {\sc }
%\thanks{}}

%\author{Xin Xu${}^2$}
%\address{Chern Institute of Mathematics and LPMC, Nankai University\\ Tianjin, China}
%\email{}
%\thanks{${}^\#$ Partially supported by NSFC (Nos. 12022111, 11771341) andthe Fundamental Research Funds for the Central Universities (No. 2042021kf1059)}
%\author{Maolin Zhou${}^3$}
%\address{Chern Institute of Mathematics and LPMC, Nankai University\\ Tianjin, China}
%\email{}

\begin{abstract}
In this paper, we obtain families of two-fold doubly-connected uniformly rotating vortex patches of the 2-D incompressible Euler equations emanating from some specific annuli. The main difficulty comes from strong degeneracy of the problem, neither the kernel of linearization is one-dimensional nor the transeversallity condition holds. To this end, we make a detailed analysis on the nonlinear functional and the bifurcation curves are obtained by perturbing real algebraic varieties defined by truncated polynomials. In addition, our result partially answers an problem proposed by Hmidi and Mateu in \cite{Hmidi2016a} (\emph{Adv.Math.302 (2016), 799-850}).

{\bf Keywords:} Rotating vortex patch; Incompressible Euler equation; Higher dimensional kernel; Degenerate bifurcation

\end{abstract}

%{\bf Acknowledgement:} The first author would like to thank Prof. Chongchun Zeng for helpful discussions and suggestions. Maolin Zhou is supported by the National Key Research and Development Program of China (2021YFA1002400), Nankai Zhide Foundation and National Science Foundation of China (No. 12271437, 11971498). Yuchen Wang is partially supported by the National Science Foundation of China No. 11831009 and the funding of innovating activities in Science and Technology of Hubei Province.

%{\bf Statements and Declarations:} No conflict of interest exists in the submission of this manuscript. No data was used for the research described in this manuscript.

\section{Introduction} \label{S:Intro}

Consider the two dimensional incompressible Euler equations
\beq \label{E:Euler-E}
\begin{cases}
\p_t u + (u \cdot \nabla) u = - \nabla p,\quad x \in \R^2, \\
\nabla \cdot u = 0,  \\
u \rightarrow 0, \quad x \rightarrow \infty
\end{cases}
\eeq
where $u=(u^1,u^2)^t$ is the velocity of fluids and $p$ is the scalar pressure. Under mild smooth conditions, we shall mainly work with the equivalent vorticity formulation
\beq \label{E:Euler-V}
\begin{cases}
\p_t \omega + (u \cdot \nabla) \omega = 0,  \\
u = -\nabla^{\perp}(-\Delta)^{-1} \omega,  \quad a^\perp := (-a_2,a_1)^t,
\end{cases}
\eeq
with the vorticity field given by $\omega :=\p_1 u^2 - \p_2 u^1$. It is clearly a closed PDEs system concerned $\omega$ and the global well-posedness for smooth solutions and Yudovich-type weak solutions $\omega \in L^1(\R^2) \cap L^\infty(\R^2)$ are well-known, see for example \cite{MP1994,Yud1963}, namely they uniquely determine two dimensional Euler flows globally in time.

Among many other candidates, the vortex patch solutions which are Yudovich-type solutions in the form of
\[
\omega = \sum_{j=1}^N \kappa_j \I_{D_j}, \;\; \text{ $\I_D$ is the characteristic function of domain $D$,}
\]
received particular attention. Here $\kappa_1,\ldots,\kappa_N$ are non-zero constants and $D_1,\ldots,D_N$ are smooth bounded domains. In our opinion, the reason is two-fold. On the one hand, vorticity is transported by the planar Euler flows therefore the evolution of vortex patches is much simpler than others. It indeed involves boundary dynamics of vortical domains only and one could found that the structure is much more nice. For example vortex patches have an invertible symplectic structure, cf. \cite{Berti2023} and \cite{LWZ2019}. On the other hand, it is well-known that measurable functions could be decomposed into the sums of characteristic functions, it also shed a light on understanding the complicated vortex phenomena.

As one describes the unknown boundary of vortical domains by $z(s,t)$ with the arc-length parameter $s$,  vortex patch dynamics is given by
\beq \label{E:Evo-VP}
\p_t z(s,t) \cdot \mathbf{n}(z) = u(z,t) \cdot \mathbf{n}(z) \quad z \in \p D(t),
\eeq
where $\mathbf{n}(z)$ denotes the unit outward normal to $\p D(t)$. Here we only need to consider velocity on the normal directions since the other one does not change the shape
but only a re-parameterization. The global regularity of vortex patch was firstly proved by Chemin \cite{Che1993} via a paradifferential approach. See Bertozzi and Constantin \cite{BC94} and Serfati \cite{Serfati1994} for another elementary proofs. Very recently,  Kiselev and Luo \cite{KisLuo2023} obtained the global regularity of vortex patches for initial data of Sobolev class. The ill-posedness of vortex patches in some critical function spaces was also obtained.

As we mentioned above, the dynamics of vortex patch is more tractable than others, but it is also very delicate due to the highly non-linearity and non-locality.  A first and accessible step is to understand the stationary structures. Note that the 2-dim Euler equation is $S^1$-invariant on the plane. Without loss of generality, we fix the center of vorticity at the origin due to several conservation laws of the vorticity therefore we shall mainly focus on the uniformly rotating vortex patches
\beq
\omega(t,x) = \omega_0(e^{-i \Omega t} x), \quad \omega_0 = \I_D
\eeq
with the angular velocity $\Omega>0$ throughout this paper. Due to \eqref{E:Evo-VP}, the uniformly rotating vortex patch satisfies
\beq \label{E:Steady-V}
\Omega z^\perp \cdot  \mathbf{n}= u \cdot \mathbf{n}, \quad \text{ for } z \in \p D.
\eeq
With the assistance of complex analysis, it can be written equivalently in the formulation
\beq
{\rm Re}\left(\overline{u}  \mathbf{n}\right) = {\rm Re}\left(-i \Omega \overline{z}  \mathbf{n}\right).
\eeq
Here $u$ and $\mathbf{n}$ are considered as complex-valued functions defined on the complex plane
\[
z=x_1+\sqrt{-1}x_2 \in \mathbb{C} \text{ with canonical inner product } a \cdot b := {\rm Re}(a \overline{b}), a,b \in \mathbb{C}.
\]
%The vortical domain rigidly rotates around the origin without deformation on the shapes. Recall the stream-vorticity formula in \eqref{E:Euler-V}.
Furthermore, due to the Stokes' theorem, the complex conjugate of velocity could be represented as
\beq
\overline{u}(z,t) = - \frac{i}{2 \pi} \int_{D} \frac{1}{z - \xi } d \mu_\xi = \frac{1}{4 \pi } \int_{\p D} \frac{\overline{\xi}-\overline{z}}{\xi - z} d \xi.
\eeq
The rotating vortex patch \eqref{E:Steady-V} is expressed in the compact formula
\beq\label{E:Steady-VP-1}
{\rm Im} \left( (2 \Omega \overline{z} + I(z)) \mathbf{n} \right) = 0, \quad z \in \p D
\eeq
with the Cauchy integral operator
\beq
I(z) = \frac{1}{2 \pi i} \int_{\p D} \frac{\overline{\xi-z}}{\xi -z} d \xi.
\eeq
It is a typical nonlinear elliptic free boundary problem where the vortical domain $D$ is the main unknown. %In particular, $\mathbf{I}_D$ is stationary vortex patch if $\p D$ solves \eqref{E:Steady-VP-1} with $\Omega=0$.

Existence and regularity of stationary/uniformly rotating vortex patches of the 2-dim incompressible Euler equation have been extensively studied in last several decades.
It is not hard to found that circular vortex patches (the Rankine vortex) as well as annular patches are uniformly rotating vortex patches for any angular velocity. A remarkable explicit example is the Kirchhoff ellipse vortex patch
\[
\omega=\mathbf{I}_{E_{a,b}},\;\; \Omega = \frac{ab}{(a+b)^2},
\]
where $E_{a,b}$ is an ellipse with the semi-axes $a,b>0$. More uniformly rotating vortex patches were found near the circular vortex patch by implementing the remarkable local bifurcation approach proposed by Crandall and Rabinowitz in \cite{CR1971}.
%\begin{lemma}[Crandall-Rabinowitz local bifurcation theorem\cite{CR1971}] \label{L:CR}
%	Suppose $X$ and $Y$ are two Banach spaces, $V$ is a neighborhood of $0$ in $X$ and $F: \R \times V \rightarrow Y$ is of class $C^k, k \geq 2$ satisfying $F(\lambda,0) =0$ for any $\lambda \in \R$. Suppose also that
%	\begin{itemize}
%		\item $L=\p_x F(\lambda_0,0)$ is a Fredholm operator of index zero.
%		\item The kernel of $L$ is 1-dimensional
%		\item Let ${\rm ker} L = \{ \xi \in X: \xi = s \xi_0 \text{ for some $s \in \R$} \}$ where $\xi_0 \in X \setminus \{0\}$. The transversality condition holds:
%		\[
%		\p_{\lambda,x}^2 F(\lambda_0,0)(1,\xi_0) \notin {\rm Range}(L).
%		\]
%	\end{itemize}
%	Then $(\lambda_0,0)$ is a bifurcation point. More precisely, there exists $\varepsilon >0$ and a branch of solutions
%	\[
%	\big\{ (\lambda,x) = (\Lambda(s),s\chi(s)), s \in \R, |s| < \epsilon \big\} \subset \R \times X
%	\]
%	such that $F(\Lambda(s),s\chi(s)) =0$ for $|s| < \epsilon$ with $\Lambda(0)=0$ and $\chi(0)=\xi_0$.
%\end{lemma}
In \cite{Burbea1982} Burbea pointed out  there are plenty of $m$-fold uniformly rotating vortex patches emanating from the disk, which are referred to the $m$-fold Kelvin waves or $V$-states for $\m \geq 3$ nowaday. A mathematically rigorous proof was given by Hmidi, Mateu and Verda \cite{Hmidi2013}, where they also proved that the boundary of rotating patches is $C^\infty$ smooth. The smoothness was soon improved to real analyticity by Castro, Cordoba and Gomez-Serrano in \cite{Castro2016}. It worths to point out that the Kirchhoff elliptic vortex patch indeed emanates from the circular patch as $2$-fold bifurcations. Note that the linearized contour dynamics at the elliptical vortex patch
is degenerate at the ratio $a:b=3:1$ is degenerate due to Love in 1893. Hmidi and Mateu \cite{Hmidi2016} and Castro, Cordoba and Gomez-Serrano in \cite{Castro2016} proved that a secondary bifurcation would occur on the two-fold branch in the bifurcation diagram of Rankine vortex. The global bifurcation of circular vortex patch was studied by Hassainia, Masmoudi and Wheeler in \cite{HMW2020}.% therefore the bifurcation analysis on the simply-connected vortex patch is rather complete. \\

The situation is much more delicate if the boundary of vortical domain has more than one connected components. For brevity we consider the doubly-connected domains. Suppose $D=D_1 \setminus D_2$ where each $D_j$ is a simply-connected bounded domain and $\Gamma_j = \p D_j$ denotes its boundary respectively, $j=1,2$. The presence of another connected component of boundary implies that \eqref{E:Steady-VP-1} is no longer an equation but an elliptic PDEs system
\beq \label{E:Doub-VP}
\begin{cases}
	{\rm Im} \left( (2\Omega \overline{z} + I_1(z)-I_2(z)) \mathbf{n}(z) \right) = 0, & \quad \forall \; z \in \Gamma_1, \\
	{\rm Im} \left( (2\Omega \overline{z} + I_1(z)-I_2(z)) \mathbf{n}(z) \right) = 0, & \quad \forall \; z \in \Gamma_2,
\end{cases}
\eeq
where the integral operators are defined as
\[
I_j(z) = \frac{1}{2 \pi i} \int_{\Gamma_j} \frac{\overline{\xi-z}}{\xi -z} d \xi, \;  j=1,2.
\]
$I_j(z)$ is a normal integral operator if $z \in \p D_i, i \neq j$ but singular otherwise.
It is clearly that annulus $A_{b} = \{ b \leq |z| \leq 1\},\, 0<b<1$ are doubly-connected uniformly rotating vortex patches hence it is natural to expect that bifurcation argument also works on obtaining doubly-connected rotating vortex patch. In \cite{Hoz2016a}, de la Hoz, Hmidi, Mateu and Verdera studied the linearization of \eqref{E:Doub-VP} on annulus $A_b$ and obtained

{\bf Theorem A} [cf. Theorem 1 in \cite{Hoz2016a}]
{\it
	Suppose $\lambda =1 -2 \Omega$ and $$
	\mathcal{S}\triangleq  \left\{\lambda \in \R: \Delta_m(\lambda,b) = 0, \text{ for some integers $m \geq 0$}\right\}
	$$
	denotes the dispersion relation, where
	$$
	\Delta_m(\lambda,b) := \left((1-\lambda)+b^2 + m(b^2-\lambda) \right) (m(1-\lambda)-\lambda) + b^{2m+2}.
	$$
	Denote by $\mathcal{L}_{\lambda,b}$ the linearized operator of \eqref{E:Doub-VP} at $A_b$ with angular velocity $\Omega$. Then $\mathcal{L}_{\lambda,b}$ has non-trivial kernel if and only if $\lambda \in \mathcal{S}$ for some $m \in \mathbb{N}_0$ and $ 0<b<1$.
	
	When $\lambda \neq \frac{1+b^2}{2}$, the kernel is the 1-dim vector space generated by
	\[
	w \in \mathbb{T} \rightarrow \left((m(1-\lambda)-\lambda)\overline{w}^m,-b^m \overline{w}^m \right)
	\]
	where $m$ is the unique integer such that $\Delta_m(\lambda,b)=0$. The range of $\mathcal{L}_{\lambda,b}$ is closed and its co-kernel is 1-dim when $\lambda \in \mathcal{S} \setminus \{1,b^2,\frac{1+b^2}{2}\}$.
	The transversality holds for $\lambda \in \mathcal{S} \setminus \{\frac{1+b^2}{2}\}$.
	
	On the other hand, when $\lambda = \frac{1+b^2}{2}$,  dimension of kernel and the co-kernel of $\mathcal{L}_{\lambda,b}$ is either $1$ or $2$. It is dimension $2$ if and only if there exists $n \geq 2$ such that $\Delta_n(\frac{1+b^2}{2},b)=0$.
}

The spectral properties of the linearized operator imply the following alternatives occur:
\begin{enumerate}
	\item {\bf Simple eigenvalues:} For $\lambda \in \mathcal{S} \setminus \{1,b^2,\frac{1+b^2}{2}\}$, in \cite{Hoz2016a} de la Hoz, Hmidi, Mateu and Verdera  obtained families of non-radial $m$-fold symmetric doubly-connected vortex patches emanating from the annulus $\{z:b \leq |z| \leq 1\}$ at the angular velocity
	\[%\label{E:Eigen-dou}
	\Omega_m = \frac{1-b^2}{4} \pm \frac{1}{2m} \sqrt{\left(\frac{m(1-b^2)}{2}-1\right)^2-b^{2m}}, \;\; m \geq 3
	\]
	by implementing the standard Crandall-Rabinowitz theorem. In particular, if $\lambda=1$  they obtained a family of translated annulus. Note that in this case all the eigenvalues are simple. In \cite{Hmidi2017} Hmidi and Renault proved that continuation of these local bifurcation would form a loop for $m \geq 3$ if a pair of eigenvalues are close enough to each other. \\

	\item {\bf Single eigenvector without transversality condition:} In \cite{Hmidi2016a}, the author further studied the case $\lambda = \frac{1+b^2}{2}$ where the transversality does not hold. They proved that there exist families of non-radial two-fold doubly-connected vortex patches emanating from the annulus $\{z:b \leq |z| \leq 1\}$ with angular velocity
	\[
	\Omega_2 = \frac{1-b^2}{4},\; \;b\in (0,1) \setminus \{ b_{2p}\}_{p \geq 2}
	\]
	where $b_{2p}$ solves
	\[
	\big(p\left(1-b_{2p}^2\right) - 1\big)^2  -b_{2p}^{4p} = 0, \quad \text{ for } p \geq 2.
	\]
These bifurcations are transcritical. On the other hand, there is no $m$-fold symmetric doubly-connected vortex patches emanating from the annulus $\{z:b \leq |z| \leq 1\}$ at angular velocity $\Omega_m = \frac{1-b_m^2}{4}$ for each $m \geq 3$, where the unique $b_{m}$ solves
	\[
	\big(\frac{m}{2}\left(1-b_{m}^2\right) - 1\big)^2  -b_{m}^{2m} = 0, \quad \text{ for } m \geq 3.
	\]
	It coincides with the presence of loops proved in \cite{Hmidi2017}. The main idea is to study the Taylor expansion  of the reduced functional up to the second-order  where the loss of transversality implies $\lambda$ would occur at least quadratic. However, there are a countable set $b_{2p}$ are left to open since the function space is carefully chosen here to ensure the kernel of linearization is one-dimensional. \\

	\item {\bf Infinitely many co-dimension:} When $\lambda=b^2$, existence of non-trivial uniformly rotating doubly-connected vortex patch is also left to open. Although the kernel is one-dimensional in this case while the transversality holds,  the co-kernel is $\infty$-dimensional hence $\mathcal{L}_{\lambda,b}$ would not be a Fredholm operator.

		\item {\bf Double eigenvectors without transversality condition:} The case $\Omega_2 = \frac{1-b_{2p}^2}{4}, p \geq 2$ is left to open. This problem is much more subtle since  both the kernel and the co-kernel of the linearized operator are two-dimensional, as well as the transversality condition fails. \\
	\end{enumerate}
The last two cases remains open since they may beyond standard bifurcational analysis.
%\red{The last two cases remains open. Among them, the third case with $\lambda =b^2$ poses a particularly interesting challenge due to the infinitely dimensional co-kernel, where the Fredholm index is not $0$ since the kernel is $1$-dimensional. It seems that the nonlinear functional would be vanished on most of directions in the co-kernel except one. On the other hand, in \cite{GPSS2021}, G\'omez-Serrano, Park, Shi and Yao proved a uniformly rotating vortex on $\R^2$ must be radial for $\Omega < 0$ or $\Omega \geq \frac{1}{2}$ no matter the simply-connectedness of the vortical domain. Their proof is in a variational flavor hence it holds for all rotating vortex patches. For given $0<b<1$, we guess that the lower bound would be $\Omega \geq \frac{1-b^2}{2}$ if the uniformly rotating vortex patch is located on the bifurcation curves emanating from the annulus $A_b$, $0<b<1$.}
In this paper, we focus on the case $4$. The main result is given as follow
\begin{theorem} \label{T:main}
 There are non-trivial two-fold uniformly rotating vortex patches emanating from the annulus $A_{b_{2p}}$ with angular velocity $\Omega =\frac{1-b_{2p}^2}{4}$ for $p =2, 3, 4$, where $0<b_{2p}<1$ satisfies the algebraic equation
	\beq \label{E:b-2m}
	b^{2p}=p-1-pb^2.
	\eeq
Moreover, these bifurcation are transcritical.
\end{theorem}
\begin{remark}
The type of bifurcation is consistent with results for $b \in (0,1) \setminus \{b_{2p}\}, p \geq 2$ in \cite{Hmidi2016a}. We believe that our approach would work for any $p \geq 5$ but a more complicated computation seems unavoidable. On the other hand, our consideration might also work for the doubly-connected rotating vortex patches in the active scalar equations \cite{Hmidi2016c} or the quasi-geostrophic shallower water equations \cite{Roulley2023}.
\end{remark}
The main difficulty comes from the degeneracy of the functional: the kernel of linearization is two dimensional and the transversality condition does not hold. We are aware of some papers when they occur alternatively. For example the local bifurcation with simple eigenvalue without transversality condition was considered by Liu, Shi, and Wang in \cite{Liu2013} and Hmidi and Mateu for the vortex patch \cite{Hmidi2016a}. On the other hand, an abstract framework concerned the bifurcation problem with higher-dimensional kernel was also introduced in \cite{Kie2012} but the conditions seem hard to verify.

Heuristically speaking, the transversality condition implies $\lambda$ dominates the linear terms in the Taylor expansion of the bifurcation equation hence a pitch-fork bifurcation may occur. Otherwise, one has to carefully check the Taylor expansion of reduced functional then found the leading terms. In this paper we design a more realizable approach for the real-analytical bifurcation problem with higher-dimensional kernel with less calculations.

Roughly speaking, Our analysis is actually based on calculating the Taylor expansion of the reduced functional \eqref{E:F2} on the annulus $A_{b_{2p}}$ up to the $p+1$-th order for $p \geq 2$. We introduce an auxiliary parameter $a\in \R$ and explore the structure of solutions near $a=0$. We found that the two components of \eqref{E:Doub-VP} after a Lyapunov-Schmidt reduction are not equal. One is $O(1)$ and the other is $O(a)$. In particular, we observe the higher-order derivatives of the 'weaker' component  projected onto the co-kernel space nearly vanish except only one term up to the $p+1$-th order,  $p=2,3,4$. With this fact, we are able to identify the dominated terms in the bifurcation equations and non-trivial solutions follows. See Section \ref{S:Scheme} for more details.
We conjecture that this observation is true for all $p\ge2$. It is the first result, up to the author's knowledge, dealing with the bifurcation problem in this very degenerate setting. % We hope the argument would be also usable for another problems. \\

%\red}

Moreover, based on the computation result for $p=2,3,4$, one would found that the degeneracy increase linearly. Note that $b_{2p} \rightarrow 1$ as $p \rightarrow +\infty$ and $\lambda_{2p} \rightarrow 1$. It somehow coincides with the case $\lambda=1$ where non-trivial solutions are translating annuli. It indicates that the functional $G=(G_1,G_2)$ actually vanish on the most components in Fourier series.

The paper is organized as follows. Section \ref{S:Pre} is devoted to the setting up of rotating vortex patches in the analytic framework and illustrate the scheme of the degenerate bifurcation problem. The main result is presented in Section \ref{S:Scheme} where we introduce a systematic scheme. In Section \ref{S:Deriv-12} and \ref{S:Degenerate}, we compute the Jacobian and the Hessian of the reduced nonlinear function and identify the degenerate direction, which enables us to determine the bifurcation solutions for the cases of $p=2,3,4$ with the computation results of higher-order derivatives up to fifth-order presented in Section \ref{S:Higher-D}. For the case $p=2$, only third-order derivatives are required, see Section \ref{S:Third}; for $p=3$, both third-order and forth-order derivatives are utilized, see Section \ref{S:Third} and \ref{S:Fourth}; and for the case $p=4$, higher-order derivatives up to the fifth are computed, see Section \ref{S:Third}, \ref{S:Fourth} and \ref{S:Fifth}. In summary, with enough computations, all results for any $p\ge2$ can be determined.

\section{Set-up}\label{S:Pre}

\subsection{Conformal mappings and regularity}

For the doubly-connected bounded domain $D= D_1 \setminus D_2$ where $D_1$ and $D_2$ are smooth simply-connected bounded domains, we follow the conformal mapping parameterizations. Due to the Riemannian mapping theorem, there exist unique conformal mappings
\beq \label{E:CM}
\begin{split}
 \phi_j(w) & = b_jw + \sum_{n \geq 0} \frac{a_{j,n}}{w^n},\quad \R^2 \setminus B_1 \rightarrow  \R^2 \setminus D_j, \\
& \phi_j(\infty)=\infty, \; j=1,2, \;  b_1=1,b_2=b,
\end{split}
\eeq
where $a_{j,n} \in \R$ due to the domain is rotation invariant and $0<b<1$ is a prescribed parameter. The conformal mappings are uniquely determined by ${\rm Re} (\phi_j \big|_{\mathbb{T}})$ respectively, which is one-one corresponding to the imaginary part of $\phi_j \big|_{\mathbb{T}}$ by the Hilbert transform.
 \begin{remark}
Here we adopt the conformal mappings for the exterior domains rather than disks since the automorphism group of the unit disk has three degree of freedom.
\end{remark}
A vortex field $\omega = \I_D, D=D_1 \setminus D_2$ is a rotating vortex patch if conformal mappings  $\phi_1$ and $\phi_2$ solve the equations
\beq \label{E:Non}
G_j(\lambda,\phi_1,\phi_2)(w)\triangleq {\rm{Im}} \left\{ \left((1-\lambda) \overline{\phi_j(w)}+I(\phi_j(w))\right)w\phi_j'(w)\right\}=0,\,\forall w\in \mathbb{T}, j=1,2,
\eeq
where $\lambda =1-2\Omega$ and
\beq \label{E:Cauchy-m}
I(z)\triangleq \frac{1}{2\pi i}\int_\mathbb{T} \frac{\overline{z}-\overline{\phi_1(\xi)}}{z-\phi_1(\xi)}\phi_1'(\xi)\dif \xi-\frac{1}{2\pi i}\int_\mathbb{T} \frac{\overline{z}-\overline{\phi_2(\xi)}}{z-\phi_2(\xi)}\phi_2'(\xi)\dif \xi\triangleq I_1(z)-I_2(z)
\eeq
is a nonlocal Cauchy-type integral operator.

Let $\phi_j=b_j{\rm{Id}}+f_j$, $f=(f_1,f_2)$ and $G(\lambda,f)= (G_1(\lambda,\phi_1,\phi_2),G_2(\lambda,\phi_1,\phi_2)): \mathbf{X} \rightarrow \mathbf{Y}$ where
\[
\begin{split}
	\mathbf{X}:=&  \left\{ \sum_{n \geq 0}
	A_n
	\overline{w}^n,\quad w \in \mathbb{T} \right\} \subset  \left(C^{1,\alpha}(\mathbb{T})\right)^2, A_n \in \R^2, \\
	\mathbf{Y}: = & \left\{ \sum_{n \geq 1}
	B_n {\rm Im}(\overline{w})^n, \quad w \in \mathbb{T} \right\} \subset   \left(C^{\alpha}(\mathbb{T})\right)^2, B_n \in \R^2.
\end{split}
\]
In \cite{Hoz2016a}, the author have proved $G \in C^1$.  Inspired by \cite{HMW2020}, the regularity can be improved to real analytical in a small neighbourhood of annular patches.
\begin{lemma} \label{L:Real-analyticity}
$G=(G_1,G_2)$ is real-analytic mappings on $B_\delta(0) \subset C^{1,\alpha}(\mathbb{T}) \times C^{1,\alpha}(\mathbb{T})$ for $0<\alpha<1$ and small $\delta>0$.
\end{lemma}

It is sufficient to prove real-analyticity of the nonlinear operator concerned with $\phi_j,j=1,2$
\[
\begin{split}
	I(\phi_j(z)) & = I_1(\phi_j(z)) - I_2(\phi_j(z)) \\
	& = \frac{1}{2\pi i}\int_\mathbb{T} \frac{\overline{\phi_j(z)}-\overline{\phi_1(\xi)}}{\phi_j(z)-\phi_1(\xi)}\phi_1'(\xi)\dif \xi-\frac{1}{2\pi i}\int_\mathbb{T} \frac{\overline{\phi_j(z)}-\overline{\phi_2(\xi)}}{\phi_j(z)-\phi_2(\xi)}\phi_2'(\xi)\dif \xi,\; z \in \mathbb{T}, j=1,2,
\end{split}
\]
which consists of a singular integral operator  $I_j[\phi_j,\overline{\phi_j(z) - \phi_j(\cdot)}](z),\; j=1,2$,
\[
I_j[\phi_j,f](z):=\frac{1}{2\pi i}\int_\mathbb{T} \frac{f(\xi)}{\phi_j(z)-\phi_j(\xi)}\phi_j'(\xi)\dif \xi,
\]
and a normal integral operator
\[
K_j[\phi_1,\phi_2,f](z) = \frac{1}{2\pi i}\int_\mathbb{T} \frac{f(\xi)}{\phi_j(z)-\phi_k(\xi)}\phi_k'(\xi)\dif \xi, \; z \in \mathbb{T}, \;j,k=1,2,\;j \neq k.
\]
Recall the following result given by de Crisoforis and Preciso in \cite{LP99}.
\begin{lemma}
	Consider the set $U^{k,\alpha}$ defined by
	\beq \label{E:fun-space}
	U^{k,\alpha} = \left\{\phi \in C^{k,\alpha}(\mathbb{T}) \big| \inf_{ s_1 \neq s_2} \left| \frac{\phi(s_1)-\phi(s_2)}{s_1-s_2} \right|>0 \right\}
	\eeq
	The Cauchy integral operator
	$$
	\mathcal{C}(\phi):=\frac{1}{2\pi i} \int_{\mathbb{T}} \frac{\overline{\phi(z)} - \overline{\phi(\xi)}}{\phi(z) - \phi(\xi)} \phi^\prime(\xi) d \xi
	$$
	defines a real-analytic mapping from $U^{k,\alpha}$ to $\mathcal{L}(C^{k,\alpha}(\mathbb{T}))$, where $\mathcal{L}(C^{k,\alpha}(\mathbb{T}))$ is the linear operator from $C^{k,\alpha}(\mathbb{T})$ to itself.
\end{lemma}
Now we could prove Lemma \ref{L:Real-analyticity} by checking out the  conditions.
\begin{proof}[Proof of Lemma \ref{L:Real-analyticity}]
Note that the function
\[
K(f)(z) :=\frac{1}{2 \pi i} \int_{r\mathbb{T}} \frac{f(\xi)}{\xi-z} d \xi, \quad z \in \mathbb{C} \setminus \{ r \mathbb{T}\}
\]
is holomorphic for $f \in C^{m,\alpha}(r\mathbb{T},\mathbb{C})$. The real-analyticity of $K_j[\phi_1,\phi_2,\overline{\phi_jz(z)-\phi_k(\cdot)}](z)$ was obtained in \cite{Cris2001} for $\phi_1,\phi_2 \in U^{k,\alpha}$. Therefore, to obtain the real analyticity of  $G=(G_1,G_2)$ with respect to $\phi_1$ and $\phi_2$,
it is sufficient to show that $\phi_j \in U^{1,\alpha}$ if $\phi_j= b_j z + f_j(z), f_j(z) \in B_\delta \subset  C^{1,\alpha}(\mathbb{T})$ where $ B_\delta$ is  $\delta$-neighbourhood of $0$ for small $\delta >0$. Indeed we have,
\[
\inf_{s_1 \neq s_2}  \left| \frac{\phi_j(s_1) - \phi_j(s_2)}{s_1-s_2} \right| \geq b_j  - \sup_{s_1 \neq s_2}  \left| \frac{f_j(s_1) - f_j(s_2)}{s_1-s_2} \right|  \geq b_j - \| f_j\|_{Lip} \geq b_j - C \|f_j\|_{C^{1,\alpha}}>\frac{b_j}{2},
\]
where $b_1=1,b_2=b$.
%& \geq b_j - \left| \sum_{n \geq 1} A_{j,n} \frac{ (\frac{1}{s_1})^{2n-1} - (\frac{1}{s_2})^{2n-1} }{s_1 - s_2} \right|  \\
%&  \geq b_j - \sum_{n \geq 1}  |A_{j,n}| \left| \sum_{k=0}^{2n-2} s_1^k s_2^{2n-2-k} \right|  \geq b_j -\sum_{n \geq 1} 2n|A_{j,n}|  \\
%& \geq b_j - 2 \| \varphi_j\|_{W^{1,1}(\mathbb{T})}   \\
%& \geq b_j -  4 \pi \left( \| \varphi_j\|_{C^0} + \| \varphi_j^\prime\|_{C^0} \right)  \\
%& \geq b_j- 4 \pi \|\varphi\|_{C^{1,\alpha}} > \frac{b_j}{2} >0
The proof is complete.
\end{proof}

\subsection{Linearization on annulus}

Consider the $m$-fold symmetric function spaces
\begin{eqnarray}
	\mathbf{X}_m=\left\{ f\in (C^{1+\alpha}(\mathbb{T}))^2, f(w)=\sum_{n=1}^{\infty} A_n \overline{w}^{mn-1}, A_n\in\R^2\right\},
\end{eqnarray}
\begin{eqnarray}\label{Y}
	\mathbf{Y}_m=\left\{ G \in (C^{\alpha}(\mathbb{T}))^2, G=\sum_{n=1}^{\infty} B_n e_{mn}, B_n\in\R^2, \; e_{mn} = {\rm Im}(\overline{w}^{mn}) \right\}.
\end{eqnarray}
In \cite{Hoz2016a}, de la Hoz, Hmidi, Mateu and Verda studied the linearization of $G(\lambda,f):\R \times \mathbf{X}_m  \rightarrow \mathbf{Y}_m$ on the annulus patch $A_b$ then obtain the linearized operator with matrix Fourier multipliers
\begin{eqnarray}\label{dG1}
	\mathcal{L}_{\lambda,b}h=\sum_{n\ge 1} M_{nm}(\lambda)
	\left(
	\begin{array}{cc}
		a_{1,n}\\
		a_{2,n}
	\end{array}
	\right)
	e_{nm},
	\quad
	e_n={\rm{Im}}(\overline{w}^n),
\end{eqnarray}
where
%on the $m$-fold symmetric space $h_j(w)=\sum_{n\ge 1} \frac{a_{j,n}}{w^{nm-1}} \in \mathbf{X}_m$
%  proved the block formula of the linearized operator acts on $h$
%%	\mathcal{L}_{\lambda,b}h\triangleq\frac{\dif}{\dif t} G(\lambda,th)_{|t=0}=\partial_f G(\lambda,0)h, \quad h=(h_1,h_2).
%end{eqnarray*}
%is expressed by the matrix Fourier multipliers
$M_n$ given by
\begin{eqnarray*}
	M_n(\lambda)
	=\left(
	\begin{array}{cc}
		n\lambda-1-nb^2 & b^{n+1}\\
		-b^n & b(n\lambda-n+1)
	\end{array}
	\right), \; n \geq 1.
\end{eqnarray*}
and $h=(h_1,h_2), h_j(w)=\sum_{n\ge 1} \frac{a_{j,n}}{w^{nm-1}}$, $j=1,2$. In particular,
let $m=2$ and $\lambda=\lambda_{2p}=\frac{1+b_{2p}^2}{2}$ where $b=b_{2p}$ solves the equation
\beq
b^{2p}=p-1-pb^2.
\eeq
Due to \eqref{dG1}, the Fourier multipliers are given by
\begin{eqnarray} \label{E:M2n}
	M_{2n}(\lambda_{2p})
	=\left(
	\begin{array}{cc}
		n-1-nb_{2p}^2 & b_{2p}^{2n+1}\\
		-b_{2p}^{2n} & b_{2p}(n b_{2p}^2-n+1)
	\end{array}
	\right), \quad n \geq 1
	.\end{eqnarray}
Clearly $M_{2n}(\lambda_{2p})$ is invertible for $n \ne 1,p$. On the other hand, there are two degenerate multipliers
\begin{eqnarray}\label{E:M2p}
	M_{2}(\lambda_{2p})
	=\left(
	\begin{array}{cc}
		-b^2 & b^{3}\\
		-b^{2} & b^3
	\end{array}
	\right),\quad
	M_{2p}(\lambda_{2p})
	=\left(
	\begin{array}{cc}
		b^{2p} & b^{2p+1}\\
		-b^{2p} & -b^{2p+1}
	\end{array}
	\right)
\end{eqnarray}
with rank one for $b=b_{2p}$.

\subsection{Lyapunov-Schmidt reduction}

The block formula \eqref{E:M2p} implies the two dimensional kernel  ${ \rm ker }(\mathcal{L}_{\lambda_{2p},b_{2p}})$ is generated by the eigenvectors
$x^1=\left(
\begin{array}{cc}
	b\\
	1
\end{array}
\right)
\overline{w}$,
$x^2=\left(
\begin{array}{cc}
	b\\
	-1
\end{array}
\right)
\overline{w}^{2p-1}$. The co-kernel, generated by
\begin{eqnarray}\label{y12}
	y_1=
	\left(
	\begin{array}{cc}
		\frac{1}{\sqrt{2}}\\
		-\frac{1}{\sqrt{2}}
	\end{array}
	\right)e_{2}\triangleq \mathbb{W}_2
	\quad {\rm{ and }}\quad
	y_2=-
	\left(
	\begin{array}{cc}
		\frac{1}{\sqrt{2}}\\
		\frac{1}{\sqrt{2}}
	\end{array}
	\right)e_{2p}\triangleq \mathbb{W}_{2p}.
\end{eqnarray}
is also two dimensional. Due to the spectral properties of \eqref{E:M2n}, $\mathcal{L}_{\lambda_{2p},b_{2p}}$ is a Fredholm operator with index $0$. Then it is standard to implement a Lyapunov-Schmidt reduction. Suppose $\mathcal{X}$ is the complementary of ${\rm Ker} \mathcal{L}_{\lambda_{2p},b_{2p}}$ in $\mathbf{X}_2$ given by
\beq \label{E:h}
h(w)=\sum_{n\ge2,n\ne p}A_n \overline{w}^{2n-1}
+\alpha_1
\begin{pmatrix}
	1\\
	0
\end{pmatrix}
\overline{w}
+\alpha_2
\begin{pmatrix}
	1\\
	0
\end{pmatrix}
\overline{w}^{2p-1}, A_n \in \R^2, \alpha_1,\alpha_2 \in \R,
\eeq
and $\mathcal{Y}$ is the complement of co-kernel of $ \mathcal{L}_{\lambda_{2p},b_{2p}}$ in $\mathbf{Y}_2$ given by
\beq \label{E:k}
k(w)=\sum_{n\ge2,n\ne p} B_n e_{2n}
+ \beta_1
\begin{pmatrix}
	1\\
	1
\end{pmatrix}
e_{2}
+\beta_2
\begin{pmatrix}
	1\\
	-1
\end{pmatrix}
e_{2p},  B_n \in \R^2, \beta_1,\beta_2 \in \R.
\eeq
$\mathcal{X}$ and $\mathcal{Y}$ are closed subspaces and projection ${\rm Id}-Q: \mathbf{Y}_2 \rightarrow \mathcal{Y}$ is a bounded operator.

Consider the nonlinear function \eqref{E:Non}
\[
G(\lambda,h+g) = 0, \quad g \in {\rm Ker} \mathcal{L}_{\lambda_{2p},b_{2p}}, \; h \in \mathcal{X}.
\]
Due to \eqref{E:M2n}, the linearized operator $({\rm Id} -Q) \p_f G(\lambda_{2p},0) :\mathcal{X} \rightarrow \mathcal{Y}$ is a bijective mapping and its inverse is also bounded, where $\p_f$ denotes the derivatives with respect to $h$. Applying the implicit function theorem, there exists a small neighbourhood  $U_{\lambda_{2p}}$ of $\lambda_{2p}$ and a unique smooth map $\varphi :\R \times {\rm Ker} \mathcal{L}_{\lambda_{2p},b_{2p}} \rightarrow \mathcal{X}$ such that
\beq \label{E:LS-red}
( {\rm Id} -Q) G(\lambda, g + \varphi(\lambda,g)) = 0, \quad \forall g \in {\rm Ker} \mathcal{L}_{\lambda_{2p},b_{2p}}, \; \lambda \in U_{\lambda_{2p}}.
\eeq
Moreover, suppose that $h \in \mathcal{X},k \in \mathcal{Y}$ solve the linear equation
\[
({\rm Id} -Q) \p_f G(\lambda_{2p},0)h = k.
\]
We have
\beq \label{iG}
\beta_1 = - b^2 \alpha_1, \,\beta_2 = b^{2p} \alpha_2,\, B_n = M_{2n}(\lambda_{2p}) A_n.
\eeq
by straightforward calculations. Seeking for local bifurcations of the system \eqref{E:Non} is equivalent to solve the two dimensional reduced system
\beq \label{E:Reduced-p}
Q G \left(\lambda, g + \varphi(\lambda,g)\right) = 0.
\eeq
A rather general bifurcational approach was proposed by Crandall and Rabinowitz in \cite{CR1971}. The idea is to blow-up the reduced function to eliminate the degeneracy coming from $F(\lambda,0) \equiv 0$ for any $\lambda \in \R $ then consider the expansion up to the second-order terms.

Similar arguments also work for the bifurcation problem involving the higher dimensional kernels in principle, see discussions in \cite{Kie2012} for example. But it very difficult to verify the conditions in general. Inspired by the Crandall-Rabinowitz theorem, we consider the blow-up of the reduced problem
\begin{eqnarray}\label{E:F2}
	F_2(\lambda,t;a)\triangleq \int_0^1 Q\partial_f G(\lambda,stx_a+\varphi(\lambda,stx_a))(x_a+\partial_g\varphi(\lambda,stx_a)x_a) d s=0.
\end{eqnarray}
along a specific direction
\[
x_a = x^1 +a x^2 \in {\rm Ker} \mathcal{L}_{\lambda_{2p},b_{2p}}  := {\rm span}\{ x^1,\; x^2\}, \quad a \in \R
\]
called the bifurcation equation. The idea is that for any given $a$, $F_2(\lambda,t;a)=0$  determines a solution $(\lambda(a),t(a))$ emanating from the trivial solution since it is two dimensional system.
%Then doubly-connected rotating vortex patch is equivalent to the non-trivial solutions of the bifurcation equation \eqref{E:F2}.
In particular, the transversality condition is violated when $\lambda = \lambda_{2p} = \frac{1 + b_{2p}^2}{2}$ implies $\p_\lambda \p_f G(\lambda_{2p},0) g \in \mathcal{Y}$, i.e.,
\beq \label{E:Non-tra}
Q \p_\lambda \p_f G(\lambda_{2p},0) g =0,\quad \forall g \in {\rm Ker} \mathcal{L}_{\lambda_{2p},b_{2p}}.
\eeq
The fact predicts that $\lambda$ would occur as quadratic terms at least in the Taylor expansion of $F_2$. Due to the degeneracy of on the linearization, one has to further explore the nonlinear structure of $F_2$. Firstly, we have
\begin{lemma} \label{L:Real-analy-2}
	$F_2(\lambda,t,a): \; \R^3 \to \R^2$ is a real-analytic mapping.
\end{lemma}
\begin{proof}
	Applying the analytic implicit function theorem (cf. \cite{Ber1977}) on $({\rm Id} -Q) G(\lambda,g+h) = 0$, we have $\varphi: \R \times {\rm Ker} \mathcal{L}_{\lambda_{2p}} \rightarrow \mathcal{X}$ is a $C^\omega$ mapping since $({\rm Id} -Q) G(\lambda,g+h)$ is zero on the constant component of its Taylor expansion and $({\rm Id} -Q) \p_f G(\lambda_{2p},0,0)$ is an isomorphism.
	Therefore we have  $F_2(\lambda,t;a):\;  \R^3\rightarrow \R^2$ is $C^\omega$ mapping for $(\lambda,t,a)$ since $a \rightarrow x_a$ is linear.
\end{proof}

\section{Heuristic discussion and  Proof of the main result} \label{S:Scheme}

We prefer to prove the main theorem \ref{T:main} at first. The proof essentially depends on calculations on the Jacobian, the Hessian and higher-order derivatives of $F_2$ which we left in the following sections.

\subsection{Heuristic discussion on the scheme of bifurcations} \label{SS:Scheme}

Due to the equation \eqref{E:F2}, we are in the position to study the real locus of system
\beq
\begin{cases}
	f_1(\lambda,t,a):=Q_1 F_2(\lambda,t;a)=0, \\
	f_2(\lambda,t,a):= Q_2 F_2(\lambda,t;a)=0.
\end{cases}
\eeq
where $Q_1$ and $Q_{2}$ are the projection operator onto $\mathbb{W}_2$ and  $\mathbb{W}_{2p}$ respectively. Here we let $\lambda:=\lambda-\lambda_{2p}$ for simplicity. Although both $f_1$ and $f_2$ are real-analytical functions and their real locus may possess a complicate structure, it suffices to study local structure of the real locus near $(\lambda,t) = (0,0)$. Here our main idea is to approximate it by considering the real algebraic varieties determined by the leading polynomials. The procedure is summarized as follows:
\begin{itemize}
	\item {\bf Step $1$.} Compute  the Jacobian and Hessian of $f_1$ and $f_2$ with respect to $\lambda,t$ at $(0,0;a)$.  Due to Proposition \ref{JF2c} and \ref{d2F2}, they are expanded as
	\beq \label{E:Sys-Po}
	\begin{cases}
		f_1(\lambda,t;a)= H_2(\lambda,t;a) +  H_{\geq 3}(\lambda,t;a) = 0, \\
		f_2(\lambda,t;a)= J_2(\lambda,t;a) +  J_{\geq 3}(\lambda,t;a) = 0,
	\end{cases}
	\eeq
	where $H_2$, $J_2$ are quadratic homogeneous polynomials concerned with $\lambda$ and $t$ and all the higher-degree terms are left to $H_{\geq 3}, J_{\geq 3}$.
	
	\item {\bf Step $2$.}  Consider the real algebraic variety
	\[
	\mathcal{J}:=\big\{ (\lambda,t) \in \R^2 \big| \;\; H_2(\lambda,t;a)=0, J_2(x,y;a) = 0
	 \big\}
	\]
	for $a \in \R$. If it consists of isolated point $(0,0)$, we claim that manifolds
	\[
	\mathcal{M}_j:=\{(\lambda,t) \big|\; f_j(\lambda,t,a)=0\},\quad j=1,2
	\]
	would intersect with each other at the isolated point $(0,0)$ locally. Indeed, we have
	\[
	\begin{pmatrix}
		f_1(\lambda,t;a) \\
		f_2(\lambda,t;a)
	\end{pmatrix} = \begin{pmatrix} A_{11}(a) & 0 \\
		0 & A_{22}(a)
	\end{pmatrix} \begin{pmatrix} \lambda^2 \\
		t^2
	\end{pmatrix} + \begin{pmatrix} \tilde{H}_{\geq 3}(\lambda,t;a) \\
		\tilde{J}_{\geq 3}(\lambda,t;a)
	\end{pmatrix} = 0
	\]
after a suitable rotation. Note the fact $\mathcal{M}_1$ and $\mathcal{M}_2$ are transversallity intersected implies
$\begin{pmatrix} A_{11}(a) & 0 \\
	0 & A_{22}(a)
\end{pmatrix}$ is inevitable and the inverse operator is bounded. We have the equality
\[
	\begin{pmatrix} \lambda^2 \\
		t^2
	\end{pmatrix} = - \left(\begin{pmatrix} A_{11}(a) & 0 \\
		0 & A_{22}(a)
	\end{pmatrix} \right)^{-1} \begin{pmatrix} \tilde{H}_{\geq 3}(\lambda,t;a) \\
		\tilde{J}_{\geq 3}(\lambda,t;a)
	\end{pmatrix}
	\]
in which the left-hand side is quadratic and the right-hand side is cubic in its leading orders. Then $(\lambda,t)=(0,0)$ follows from a standard fixed-point argument.

\item{\bf Step 3.} Due to Lemma \ref{L:non-deg}, the real algebraic variety $\mathcal{J}$ is not isolated points when $a=0$. We consider system \eqref{E:Sys-Po} separately. For the first component, the Taylor expansion is given as
\beq \label{E:Com-1}
	H_2(\lambda,t;a) +  H_3(\lambda,t;a)+ H_{\geq 4}(\lambda,t;a) = 0, \quad |a| \ll 1.
\eeq
We solve the algebraic equation consists of leading polynomials $H_2(\lambda,t;a) + H_3(\lambda,t;a)$ where $\lambda$ and $a$ are coefficients and $t$ is the unknown. Due to Proposition \ref{d2F2}, one has $H_2(\lambda,t;0)$ is hyperbolic. Solving equation $H_2(\lambda,t;0)=0$ we have
\[
C^2 \lambda^2 - t^2 =0
\]
with some constant $C>0$. We claim that for small $a \neq 0$, the following function
\beq \label{E:t-lambda}
t(\lambda,a) =A(a)\lambda + B \lambda^2+ O(|a| \lambda^2 + \lambda^3)
\eeq
solves \eqref{E:Com-1}, where $B \neq 0$ and $A(0)=C$. Indeed, dividing $\lambda^2$ from both sides of \eqref{E:Com-1}, one has the equation
%\[
%a_1(a) \lambda^2 - a_2(a) t^2 +  H_{\geq 3}(\lambda,t,a) = 0.
%\]
%where $a_1$ is continuous function and let $a_2 \equiv 1$ with lost of generality. It is equivalent to consider
\[
A(a) - \frac{t^2}{\lambda^2} + \frac{1}{\lambda^2}  H_{\geq 3}(\lambda,t,a) = 0, \;\; A(0)=C>0.
\]
Let $\xi = \frac{t}{\lambda}$. Note that the components in $H_{\geq 3}(\lambda,t,a)$ are at least cubic. We have
\[
\mathcal{F}(\lambda,\xi,a):=A(a) - \xi^2 + \lambda \widehat{H_{\geq 3}}(\xi,\lambda,a) = 0.
\]
where $\widehat{H_{\geq 3}}(\xi,\lambda,a) = \frac{1}{\lambda^3} H_{\geq 3} (\lambda,\xi \lambda;a)$ is sums of polynomials. Applying the implicit function theory on $\xi=C, a=0,\lambda=0$, one has
\[
\xi = C + C_1(a + \lambda) + h.o.t,
\]
which is indeed equation \eqref{E:t-lambda}

\item {\bf Step $4$.} Now we in the position to consider the other component of  \eqref{E:Sys-Po}. The Taylor expansion implies that it could be represented as
	\[
	J_2(\lambda,t(\lambda,a);a) + J_{\geq 3}(\lambda,t(\lambda,a);a) = 0.
	\]
We have $J_2(\lambda,t(\lambda,0);0)=0$ since the Hessian is degeneracy when $a=0$.  We claim that there exists integer $n_0 \geq 3$ such that
\[
		 \sum_{2 \leq m \leq n_0-1} \mathbf{Pr}_m^{\lambda,t} J(\lambda,t;0) = 0,\quad  \mathbf{Pr}_{n_0}^{\lambda,t} J(\lambda,t;0)  \neq 0, %+ \sum_{2 \leq m \leq n_0-1}\mathbf{Pr}_m^{\lambda} J(\lambda,t(\lambda);a_0) \neq 0,
	\]
	where $\mathbf{Pr}_{m}^{\lambda,t}$ denotes the projection onto the degree-$m$ homogeneous polynomials concerned with $\lambda,t$. Otherwise, due to the real analyticity of $f_2$, one has the nonlinear function $f_2(\lambda,t;a) \equiv 0$ for small $\lambda \in \R$. One could completely ignore the this component and  it is sufficient to solve the single equation $f_1(\lambda,t;a_0)=0$ only. \\
	
Putting \eqref{E:t-lambda} into the equation $f_2(\lambda,t;a)=0$  and consider its Taylor expansion, we have
\beq\label{E:lambda-a}
	\begin{split}
0= & f_2(\lambda,t(\lambda,a);a) =  J_2(\lambda,t(\lambda,a);a) + \sum_{k=3}^{n_0}J_{k }(\lambda,t(\lambda,a);a) + J_{\geq n_0+1}(\lambda,t(\lambda,a);a) \\
	& = C_{2} a^{k_1} \lambda^2 + C_3a^{k_2} \lambda^3 + \ldots + C_{n_0-1}a^{k_{n_0-2}} \lambda^{n_0-1} + C_{n_0}\lambda^{n_0} + J_{\geq {n_0+1}}(\lambda;a)
	\end{split}
	\eeq
To formulate in a perturbation approach, one needs to found the dominated term.	To this end, dividing $C_2a^{k_1} \lambda^2$ from both sides of \eqref{E:lambda-a}, we have
\[
1 + R(\eta,a) + C_{n_0} \eta^{n_0-2} + a \widetilde{J_{\geq {n_0-1}}}(\eta;a) =0.
\]
where  $\eta = \frac{\lambda}{a^{\frac{k_1}{n_0-2}}}$ and $\widetilde{J_{\geq {n_0-1}}}(\eta,a)$ consists of higher-order terms which are considered as perturbations. For example, suppose $J_{\geq {n_0+1}}(\lambda;a)=a^{k_{n_0}}\lambda^{n_0+1}$=$a^{k_{n_0}+k_1\frac{n_0+1}{n_0-2}}\eta^{n_0+1}$ which consists of the lowest order term. Then
\[
\begin{split}
\widetilde{J_{\geq {n_0-1}}}(\eta,a) & = \frac{J_{\geq {n_0+1}}(\lambda;a)}{C_2 a^{k_1+1}\lambda^2} \\
& = C a^{k_{n_0}-1+k_1\frac{n_0+1}{n_0-2}-k_1-k_1\frac{2}{n_0-2}}\eta^{n_0-1}
=C a^{k_{n_0}+k_1\frac{1}{n_0-2}-1}\eta^{n_0-1}
\end{split}
\]
which the exponents are positive. Therefore one needs to carefully check the lower order terms concerned with $\lambda$ less than $n_0$. Suppose
\[
R(\eta,a) = \sum_{j=2}^{n_0-1} a^{k_j - k_1} \lambda^{j-2} = \sum_{j=2}^{n_0-1} a^{k_j - k_1 + \frac{j-2}{n_0-2}k_1} \eta^{j-2}
=\sum_{j=2}^{n_0-1} a^{k_j -\frac{n_0-j}{n_0-2}k_1} \eta^{j-2}
\]
When $k_j -\frac{n_0-j}{n_0-2}k_1 >0$ for $j=2,\ldots,n_0-1$, the leading term is $C_2 + R(\eta,a) + C_{n_0} \eta^{n_0-2}=0$ then we can apply the implicit function theory to ensure the real solution persists with slightly perturbations. If there exists $k_j -\frac{n-j}{n-2}k_1\leq 0$ for some $j$ otherwise, we choose this term as well as zero-order term $C_2$ as the dominated terms then the implicit function theory could apply. Fortunately, we have $k_1=1$ in our consideration therefore the dominated polynomial consists of zero-order and $n_0$-order terms.

\end{itemize}
In conclusion, there are two key ingredients in our consideration. The critical value $a=0$ and a non-vanishing higher-order terms $J_{n_0}(\lambda,t;a_0), n_0 \geq 3$. In the next section, we prove that the only degenerate direction is $a=0$. It implies that the $\mathbb{W}_{2p}$-component of $F_2$, which is of order $O(a)$, is much more smaller the $\mathbb{W}_{2}$-component with an order $O(1)$. We note that it is sufficient to consider higher-order derivatives when $a=0$ due to the real-analyticity of functional $F_2$. Due to the explicit computations, we found that the higher-order derivatives on the component $\mathbb{W}_{2p}$ vanish except exactly one term for $p \leq 4$ (Lemma \ref{L:Key-L}) up to $p+1$-th order. We further conjecture that

{\it
	For any $p \geq 5$, all derivatives of $Q_2 F_2(\lambda,t;0)$ would be vanished up to the $p+1$-th order at $(\lambda,t)=(0,0)$ except
	%\[
	%\sum_{k_1+k_2 = p+1} \frac{(p+1)!}{k_1!k_2!}  \mathcal{D}_{\lambda^{k_1},t^{k_2}}^{p+1} Q_2 F_2(0,0;0) \lambda^{k_1} t^{k_2}
	%\]	
	%is non-zero. More precisely, there would be exactly only one non-vanishing term
	$$\frac{(p+1)!}{2(p-1)!}  \mathcal{D}_{\lambda^{2},t^{p-1}}^{p+1} Q_2 F_2(0,0;0) \lambda^{2} t^{p-1}.$$}

\begin{remark}
	As pointed out in Step 4 of the scheme, we note that there exists $n_0 \geq 3$ such that the $n_0$-order terms in the Taylor expansion of $Q_2 F_2$ is definitely not vanished, i.e.,
	\[
	\sum_{k_1+k_2 = n_0} \frac{n_0!}{k_1!k_2!}  \mathcal{D}_{\lambda^{k_1},t^{k_2}}^{k_1 +k_2} Q_2 F_2(0,0;0) \lambda^{k_1} t^{k_2} \neq 0.
	\]
	Otherwise, due to the real analyticity of the functional (Lemma \ref{L:Real-analy-2}), the nonlinear functional $Q_2F_2(\lambda,t,0)$ would vanish for all $(\lambda,t)$ near $(\lambda_{2p},0)$. Hence one could deduct the eigenvector $x^2$ from ${\rm Ker} \mathcal{L}_{\frac{1+b_{2p}^2}{2},b_{2p}}$ and
	 $\mathbb{W}_{2p}$ from the target space $Y$. The linearized operator is still a Fredholm operator with index $0$ so it is sufficient to consider the problem with the kernel spanned by $x^1$ only.
\end{remark}

\subsection{Proof of main result Theorem \ref{T:main}}

With the bifurcation scheme considered above, we would like to prove the main result here where some ingredients concerned the computation of derivatives are given in Section \ref{S:Deriv-12} and Section \ref{S:Higher-D}. Denote by $\lambda:= \lambda - \lambda_{2p}$ for simplicity for $p \geq 2$. The proof follows from perturbational analysis on the real locus of truncated Taylor polynomials of $F_2$. Proposition \ref{JF2c} implies first-order derivatives of $F_2(0,0;a)$ vanishes for  $a \in \R$ hence the leading terms are quadratic.

Due to Proposition \ref{d2F2}, we have
\[
\begin{split}
	& \p_{\lambda \lambda} F_2(\lambda_{2p},0;a) = 4 \sqrt{2} \begin{pmatrix}
		b^{-1} \\ap^2 b^{1-2p}
	\end{pmatrix}, \quad\,
	\p_{\lambda t} F_2(\lambda_{2p},0;a) = 0, \\
	& \p_{tt} F_2(\lambda_{2p},0;a) \!=\! \sqrt{2} \begin{pmatrix}
		-b^{-1}(b^2-1)^2 + c_1 \\
		b^{-1}(b^2-1)^2[2p a+(2p-1)p a^3]+c_2
	\end{pmatrix}\!,
\end{split}
\]
where $c_1$ and $c_2$ are given in Proposition \ref{d2F2} with $c_1=O(a^2)$ and $c_2=O(a^3)$. For the quadratic terms, we have
\begin{small}
	\beq
	\begin{split}
		\mathbf{P}^{\lambda,t}_2F_2(\lambda,t;a) = \frac{\sqrt{2}}{2} \times
		\begin{pmatrix}
			4 b^{-1} \lambda^2 + \left[-b^{-1}(b^2-1)^2 +c_1\right]t^2 \\
			a \bigg(4 p^2 b^{1-2p} \lambda^2 \!+\! \left[  b^{-1}(b^2-1)^2(2p +(2p-1)p a^2)+\frac{c_2}{a} \right] t^2 \bigg)
		\end{pmatrix}.
	\end{split}
	\eeq
\end{small}
Lemma \ref{L:non-deg} implies the Hessian of $F_2(0,0 ;a)$ is degenerate if and only if $a=0$.

Recall the projection $Q_1:Y \rightarrow y_1$ where $Y$ and $y_1$ are defined in \eqref{Y} and \eqref{y12}. As \eqref{E:Com-1}, we consider the component $ Q_1 F_2(\lambda,t;a)$ at follows
\[
\begin{split}
	Q_1 F_2(\lambda,t;a) & =   \frac{1}{\sqrt{2}}  \left[-b^{-1}(b^2-1)^2 + c_1\right] t^2  +2 \sqrt{2} b^{-1} \lambda^2+ \sum_{k \geq 3}  \mathbf{P}_{k}(\lambda,t;a) \\
	& =\frac{2\sqrt{2}}{b} \lambda^2 - \frac{(b^2-1)^2}{\sqrt{2}b} t^2 + aC_1(a)t^2 + \sum_{k \geq 3}  \mathbf{P}_{k}(\lambda,t;a).
\end{split}
\]
where $C_j(\cdot)$ denotes continuous bounded function and $\mathbf{P}_k(\lambda,t;a)$ denotes $k$-th degree homogeneous bivariate polynomials concerned $\lambda,t$, which may be different from line to line. By the implicit function theory, one obtains
%The equation $\mathcal{Q}_1 F_2(\lambda,t,a) =0$ is hyperbolic hence it is always solvable therefore $t$ is determined by $\lambda$ and the following asymptotic formula
\beq\label{E:lambda-1}
t^2 = \frac{4}{(b^2-1)^2} \lambda^2 +  a C_2(a) \lambda^2 + O(\lambda^3)
\eeq
holds. It immediately implies
\beq\label{E:lambda-2}
t = \pm \frac{2}{b^2-1} \lambda + aC_3(a) \lambda + O(\lambda^2).
\eeq
Now we are in the position to consider the second component $Q_{2} F_2(\lambda,t;a)=0$ with the asymptotic formula \eqref{E:lambda-2} where $Q_2:Y \rightarrow y_2$ is the projection.

\begin{proof} [Proof of Theorem \ref{T:main}]

We begin with $p=2$. Due to Proposition \ref{P:d3F2a=0}, we have
\beq
\begin{split}
	\p_{\lambda \lambda \lambda} F_2(\lambda_{2p},0,0) & = 12 \sqrt{2} \begin{pmatrix}
		b^{-3} \\
		0
	\end{pmatrix},p\ge2, \\
\p_{\lambda t t} F_2(\lambda_{2p},0,0) & =
	\begin{pmatrix}
		\sqrt{2}(b^{-3}+6b^{-1}-7b) \\
		0
	\end{pmatrix},p\ge2,
\end{split}
\eeq
\beq	
\begin{split}
	\p_{\lambda \lambda t} F_2(\lambda_{2p},0,0) & = \begin{cases}  0, \quad & p \ge 3, \\
		\begin{pmatrix}
			0 \\
			-8 \sqrt{2}
		\end{pmatrix},
		\quad & p =2,
	\end{cases}\\
\p_{ttt} F_2(\lambda_{2p},0,0) & = 0, p \geq 2.
\end{split}
\eeq
	For $p=2$, let us consider the Taylor expansion of $\mathcal{Q}_2F_2(\lambda,t;a)$ up to the third-order. Therefore, due to local real analyticity of $F_2(\lambda,t;a)$, we have
	\[
	\begin{split}
		Q_2 F_2(\lambda,t,a) & =  a\left(8 \sqrt{2}  b^{-3} \lambda^2 + \frac{1}{\sqrt{2}} \big( b^{-1}(b^2-1)^2 (4+6a^2) +\frac{c_2}{a} \big) t^2\right) \\
		& - 4 \sqrt{2} \lambda^2 t + a   \mathbf{P}_{3}(\lambda,t;a)
		+\sum_{k \geq 4}\mathbf{P}_{k}(\lambda,t;a)= 0,
	\end{split}
	\]
	for $a \ll 1$. %To guarantee the above equation exists a real solution, it is sufficient to consider the leading term with suitable non-degeneracy condition.To work out this equation, we consider the
	
	Moreover, jointing with \eqref{E:lambda-1} and \eqref{E:lambda-2}, we have
	\begin{align} \label{E:Q2F2-p}
		Q_2 F_2(\lambda,t,a) & =Q_2 F_2(\lambda,t(\lambda,a);a) \nonumber\\
		&= \lambda^2\bigg[  8 \sqrt{2}a b^{-3} \!+\! \frac{a}{\sqrt{2}} \big( b^{-1}(b^2\!-\!1)^2 (4\!+\!6a^2) +\frac{c_2}{a} \big) \!
		\left( \frac{4}{(b^2-1)^2}\! +\! a C_2(a)\!+\! O(\lambda)\right)  \nonumber\\
		& - 4\sqrt{2}(\pm \frac{2}{b^2-1}) \lambda  \bigg] + aC_5(a) \lambda^3 + O(\lambda^4) = 0.
	\end{align}
	Note that $c_2=O(a^3)$. We first neglect higher-order terms for example $O(\lambda^2 a^2)$ and $O(\lambda^3 a)$ in \eqref{E:Q2F2-p} and consider the real polynomial
	\beq \label{E:leading-2}
	\lambda^2(8\sqrt{2}a b^{-3} + 8 \sqrt{2}a b^{-1} \mp \frac{8\sqrt{2}}{1-b^2} \lambda) = 0.
	\eeq
	Solving the polynomial concludes that there are non-trivial solutions
	\[
	\lambda(a) =\pm a (1-b^2) (\frac{1}{b^3} + \frac{1}{b}).
	\]
	On the next we prove that this solution is robust as one come back to the analytic function. Firstly, the equation $Q_2 F_2(\lambda,t(\lambda,a);a)=0$ implies
	\[
\mathcal{F}_3(\lambda,a)\triangleq\frac{1}{b^3} + b^{-1} \mp \frac{1}{1-b^2} \frac{\lambda}{a} + C_5(a)\lambda + a^2C_6(a)  + O(\frac{\lambda^2}{a}) =0.
	\]
	Let $\eta = \frac{\lambda}{a}$. We have
	\[
	\mathcal{F}_3(\eta,a) = \frac{1}{b^3} + b^{-1} \mp \frac{1}{1-b^2} \eta +a C_5(a) \eta + a^2 C_6(a) + O(a \eta^2) = 0
	\]
	Applying the implicit function theory on $\mathcal{F}_3(\eta,a)=0$ at $\eta = \mp (1-b^2) (\frac{1}{b^3} + \frac{1}{b})$ and $a=0$, we have
	\[
	\lambda(a) =\pm a (1-b^2) (\frac{1}{b^3} + \frac{1}{b}) + o(a), \; \text{ hence } t(a) = \pm 2 a(\frac{1}{b^3} + \frac{1}{b}) + o(a)
	\]
	for $a \ll 1$ which implies a transcritical bifurcation occurs. \\
	
	The problem is more degenerate for $p=3$ since all the third-order terms of $Q_2 F_2(\lambda,t;0)$ vanish due to Proposition \ref{P:d3F2a=0}. Therefore we have to consider the local expansion up to the fourth-order terms. From Proposition \ref{P:d4F2a=0}, we can see that all the fourth-order terms in the Taylor series expansion of $Q_2 F_2(\lambda,t;0)$ is $0$ at $(\lambda_{6},0)$ except $\p_{\lambda \lambda t t} F_2(\lambda_6,0;0) = - 24 \sqrt{2} b \; y_2$. Therefore, due to Proposition \ref{d2F2}, we have
	\[
	\begin{split}
		Q_2 F_2(\lambda,t;a) & = \frac{\sqrt{2}a}{2} \bigg( \frac{36}{b^5} \lambda^2 + \left(\frac{3(b^2-1)^2}{b}(2+5a^2) +\frac{c_2}{a}\right)  t^2 \bigg) \\
		& + a \mathbf{P}_3 (\lambda,t,a) - 6 \sqrt{2}b \lambda^2 t^2 + a \mathbf{P}_4(\lambda,t,a) + \sum_{k \geq 5} \mathbf{P}_k(\lambda,t;a)
	\end{split}
	\]
	%Similar with the scenario for $p=2$, we obtain
	%\beq \label{E:lambda-3}
	%t^2 = \frac{4}{(b^2-1)^2} \lambda^2 + aC_6(a) \lambda^2 + O(\lambda^3), \;\; t =\pm \frac{2}{b^2-1} \lambda + aC_7(a) \lambda + O(\lambda^3).
	%\eeq
	%by solving the first component of the equation $\mathcal{Q}_1 F_2(\lambda,t,a)=0$, ,
	
	Inserting the formula \eqref{E:lambda-1}-\eqref{E:lambda-2}, since the leading terms in $a P_3(\lambda, t(\lambda,a),a)$ is in the form of $a C_{6}(a)\lambda^3$, we have
	\[
	\begin{split}
		Q_2 F_2(\lambda;a) &  =  \frac{\sqrt{2}a}{2} \bigg[ \frac{36}{b^5} \lambda^2 + \left(\frac{3(b^2-1)^2}{b}(2+5a^2)+ \frac{c_2}{a}\right) \frac{4}{(b^2-1)^2}
		\big( \lambda^2+ aC_7(a)\lambda^2 + O(\lambda^3) \big) \bigg] \\
		%	& \times \big( \lambda^2+ aC_7(a)\lambda^2 + O(\lambda^3) \big) \bigg]
		&+ a \mathbf{P}_3(\lambda, t(\lambda,a),a) - \frac{24 \sqrt{2}b}{(b^2-1)^2}  \lambda^4 + O(\lambda^5 + a \lambda^4) \\
		& = \lambda^2 \bigg( \frac{18\sqrt{2}a}{b^5} + \frac{12 \sqrt{2}a}{b} + a^2C_8(a) + a C_9(a) \lambda -\frac{24\sqrt{2}b}{(b^2-1)^2} \lambda^2 \bigg)  + O(\lambda^5 + a \lambda^4).
	\end{split}
	\]
	Clearly
	\[
	\frac{18\sqrt{2}a}{b^5} + \frac{12 \sqrt{2}a}{b}-\frac{24\sqrt{2}b}{(b^2-1)^2} \lambda^2 = 0
	\]
	is always solvable for $a>0$. Then $aC_9(a) \lambda$ would be a term of $a^\frac{3}{2}$  and the leading terms in $Q_2F_2(\lambda,a) = 0$ is represented by
	\[
	\frac{18\sqrt{2}a}{b^5} + \frac{12 \sqrt{2}a}{b}-\frac{24\sqrt{2}b}{(b^2-1)^2} \lambda^2 = 0.
	\]
	Then we obtain a pair of non-zero solutions
	\[
	\lambda = \pm \frac{\sqrt{a}}{2}(b^2-1) \sqrt{\frac{3}{b^6}+\frac{2}{b^2}}, \;\; a>0
	\]
	Note that $C_9(a) \lambda = \sqrt{a}  C_9(a) \frac{\lambda}{\sqrt{a}}$. We have
	\[
	\lambda(a) = \pm \frac{\sqrt{a}}{2}(b^2-1) \sqrt{\frac{3}{b^6}+\frac{2}{b^2}} + o(a^{\frac{1}{2}}),\;\; t(a) = \pm \sqrt{a} \sqrt{\frac{3}{b^6}+\frac{2}{b^2}} + o(a^{\frac{1}{2}}),
	\]
	for small $a>0$ by the implicit function theorem as the case $p=2$. It is also a transcritical bifurcation.
	
	For $p=4$, we have ${\bf P}_{3}(\lambda,t;0)={\bf P}_4(\lambda,t;0)=0$ and ${\bf P}_5(\lambda,t;0) = - 8 \sqrt{2}b^2 \lambda^2 t^3$.
	Then the nonlinear functional $Q_2 F_2$ can be expanded in the form of
	\[
	\mathcal{Q} F_2(\lambda,a) = 8\sqrt{2} \lambda^2 \left( \left(\frac{4}{b^7} + \frac{2}{b} \right)a + a \left( C_{11}(a) + C_{12}(a) \lambda + C_{13}(a) \lambda^2 \right) \pm  \frac{8b^2}{(b^2-1)^3} \lambda^3 \right) + O(a \lambda^5 + \lambda^6).
	\]
	Following a similar argument as $p=2$, we can solve the solution for small $a \in \R$, and hence the conclusion holds.
	
\end{proof}

\section{Jacobian and Hessian of $F_2(\lambda,t;a)$} \label{S:Deriv-12}

The rest of paper is devoted to explicit calculations on the derivatives of $F_2(\lambda,t;a)$ with respect to $\lambda,t$ at $(\lambda_{2p},0)$. As shown in Section \ref{S:Scheme}, they are essential ingredients in our consideration. Although these computations are elementary that they are not involved advanced topics in the discipline except basic complex analysis, for example the residue theorem, the computation is still considerably complicated, especially when $p$ is kept as a parameter. On the other hand, these calculations are indeed very similar when $p$ varies. Therefore, we prefer to focus on the calculation for $p=2$ and give all the details. For a general $p \geq 3$, we shall not give proofs with full details but a sketch.

The main results of this section is Proposition \ref{JF2c}, Proposition \ref{d2F2}, which give the explicit information on the first and second-order derivatives for any parameter $a \in \R$ respectively. The calculation follows from the following formula.
\begin{lemma}\label{acformal}
Recall $x_a = x^1 + a x^2 \in {\rm Ker} \mathcal{L}_{\lambda_{2p},b_{2p}}$ for $a \in \R$ and $F_2$ given in \eqref{E:F2}. We have the following computational formula
\begin{enumerate}
	\item The explicit formula of the first-order derivatives:
		\begin{eqnarray*}
			\begin{split}
				&\partial_{\lambda}F_2(\lambda_{2p},0)=Q\partial_\lambda\partial_fG(\lambda_{2p},0) x_a, \\
				&\partial_t F_2(\lambda_{2p},0)=Q\partial_{ff}G(\lambda_{2p},0)[x_a,x_a]
				= \frac{1}{2} \frac{\dif^2}{\dif t^2}QG(\lambda_{2p}, tx_a)_{|t=0}.
			\end{split}
		\end{eqnarray*}
		
		\item The explicit formula for $\partial_{tt} F_2(\lambda_{2p},0)$.
		\begin{eqnarray*}
			\partial_{tt}F_2(\lambda_{2p},0)=\frac{1}{3}\frac{\dif^3}{\dif t^3}QG(\lambda_{2p},tx_a)_{|t=0}+Q\partial_{ff}G(\lambda_{2p},0)[x_a,\bar{x}]
		\end{eqnarray*}
		in which
		\[
		\begin{split}
			\bar{x}& =\frac{\dif^2}{\dif t^2}\varphi(\lambda_{2p}, tx_a)_{|t=0}=\partial_{gg}\varphi(\lambda_{2p},0)[x_a,x_a] \\
			& =-\left[\partial_{f} G(\lambda_{2p},0)\right]^{-1}\frac{\dif^2}{\dif t^2}({\rm {Id}-Q})G(\lambda_{2p}, tx_a)_{|t=0}
		\end{split}
		\]
		and
		\begin{eqnarray*}
			Q\partial_{ff}G(\lambda_{2p},0)[x_a,\bar{x}]=Q\partial_t\partial_s G(\lambda_{2p},t x_a+s\bar{x})_{|t=0,s=0}.\nonumber
		\end{eqnarray*}
		
		\item  The explicit formula for $\partial_t\partial_\lambda F_2(\lambda_{2p},0)$.
		\begin{eqnarray*}
			\begin{split}
				\partial_t\partial_\lambda F_2(\lambda_{2p},0)=&\frac{1}{2}Q\partial_\lambda\partial_{ff}G(\lambda_{2p},0)[x_a,x_a]
				+Q\partial_{ff}G(\lambda_{2p},0)\left[\tilde{x},x_a\right]\\
				&+\frac{1}{2}Q\partial_\lambda\partial_f G(\lambda_{2p},0)\bar{x}
			\end{split}
		\end{eqnarray*}
		in which
		\begin{eqnarray*}
			\tilde{x}=\partial_\lambda\partial_g \varphi(\lambda_{2p},0)x_a =-\left[\partial_f G(\lambda_{2p},0)\right]^{-1}({\rm {Id}}-Q)\partial_{\lambda}\partial_f G(\lambda_{2p},0)x_a.
		\end{eqnarray*}
		
		\item  The explicit formula for $\partial_{\lambda\lambda} F_2(\lambda_{2p},0)$.
		\begin{eqnarray*}
			\partial_{\lambda\lambda} F_2(\lambda_{2p},0)=-2Q\partial_{\lambda}\partial_f G(\lambda_{2p},0) \left[\partial_f G(\lambda_{2p},0)\right]^{-1}({\rm {Id}}-Q)\partial_{\lambda}\partial_f G(\lambda_{2p},0)x_a.
		\end{eqnarray*}
	\end{enumerate}
\end{lemma}
It follows from formal calculations on the first and second-order derivatives of composition of functions concerned with the implicit function $\varphi$ satisfying the following lemma, which is crucial. The proof is similar with Lemma 1 in \cite{Hmidi2016a} hence we omit here for brevity.
\begin{lemma}\label{L:phi-1}
	%\begin{eqnarray*}
	%\partial_{\lambda}\varphi(\lambda_{4},0)=\partial_g\varphi(\lambda_{4},0)=0,\text{ and }\partial_{\lambda\lambda}\varphi(\lambda_{4},0)=0,
	%\end{eqnarray*}
	%these formulas are true for all $p\ge2$, that is
	Recall the implict function $\varphi$ defined in \eqref{E:LS-red}. One has
	\begin{eqnarray*}
		\partial_{\lambda}\varphi(\lambda_{2p},0)=\partial_g\varphi(\lambda_{2p},0)=0,\text{ and }\partial_{\lambda\lambda}\varphi(\lambda_{2p},0)=0,
	\end{eqnarray*}
	where $b_{2p}$ is determined by \eqref{E:b-2m} and $\lambda_{2p}= \frac{1+b_{2p}^2}{2}$.
\end{lemma}

\subsection{First-order derivatives} \label{SS:Jacobian}
We begin with explicitly computating on the Jabobian of $F_2(\lambda,t;a)$.
\begin{proposition}\label{JF2c}
	For any $p \geq 2$, the first-order derivatives $F_2$ vanish at $(\lambda_{2p},0)$ for any $a \in \R$, namely
	\[
	\partial_{\lambda}F_2(\lambda_{2p},0;a)=\partial_{t}F_2(\lambda_{2p},0;a)=0.
	\]
\end{proposition}
Now let $p=2$. One has $\lambda=\lambda_4=\frac{1+b^2_4}{2}$ where $b=b_4$ solves $b^4=1-2b^2$ and the kernel of $\mathcal{L}_{\lambda_{4},b_4}$ is generated by the eigenvectors
$$x^1=\left(
\begin{array}{cc}
	b\\
	1
\end{array}
\right)
\overline{w}, \quad x^2=\left(
\begin{array}{cc}
	b\\
	-1
\end{array}
\right)
\overline{w}^{3},$$ while the co-kernel is generated by
$$
y_1=
\left(
\begin{array}{cc}
	\frac{1}{\sqrt{2}}\\
	-\frac{1}{\sqrt{2}}
\end{array}
\right)e_{2}, \quad
y_2=-
\left(
\begin{array}{cc}
	\frac{1}{\sqrt{2}}\\
	\frac{1}{\sqrt{2}}
\end{array}
\right)e_{4}.
$$
Note that $x_a=\left(
\begin{array}{cc}
	b\\
	1
\end{array}
\right)
\overline{w}
+a
\left(
\begin{array}{cc}
	b\\
	-1
\end{array}
\right)
\overline{w}^{3}.$
For convenience, we introduce symbols
\begin{eqnarray}\label{notation1}
	b_1=1,\, b_2=b,\, \alpha_1=b,\, \alpha_2=1,\, \theta_1=ab,\, \theta_2=-a.
\end{eqnarray}
Consider the deformation of $\p D$ with respect to $x_a$, that is,
\begin{eqnarray}\label{phi}
	\phi_j(t,w)=b_jw+t\alpha_j\overline{w}+t\theta_j\overline{w}^{3}.
\end{eqnarray}
The function $G=(G_1,G_2)$ is written as
\begin{align} \label{Gc}
	G_j(\lambda,tx_a)=\mathrm{Im}\left\{\left[(1-\lambda)\overline{\phi_j(t,w)}+I(\phi_j(t,w))\right]w (b_j-t\alpha_j \overline{w}^2-3t\theta_j\overline{w}^{4})\right\},
\end{align}
where
\begin{eqnarray*}
	I(\phi_j(t,w))=I_1(\phi_j(t,w))-I_2(\phi_j(t,w)), \\
	I_i(\phi_j(t,w))=\frac{1}{2\pi i} \int_{\mathbb{T}}\frac{\overline{\phi_j(t,w)}-\overline{\phi_i(t,\tau)}}{\phi_j(t,w)-\phi_i(t,\tau)}\phi_i'(t,\tau)\dif \tau.
\end{eqnarray*}

\begin{proof}[Proof of Proposition \ref{JF2c} for $p=2$]
	Due to Lemma \ref{acformal}, we have
	\begin{eqnarray*}
		\begin{split}
			\partial_{\lambda}F_2(\lambda_{4},0;a)&=Q\partial_\lambda\partial_fG(\lambda_{4},0)x_a,\\
			2\partial_t F_2(\lambda_{4},0;a)
			&=\frac{\dif^2}{\dif t^2}QG(\lambda_{4}, tx_a)_{|t=0}.
		\end{split}
	\end{eqnarray*}
and	$\partial_{\lambda}F_2(\lambda_{4},0)\!\!=\!0$ follows from the condition \eqref{E:Non-tra} immediately. It suffices to consider $\frac{\dif^2}{\dif t^2}QG(\lambda_{4}, tx_a)_{|t\!=\!0}$.
	
	Due to the formulation of $G$ given in \eqref{Gc}, a straightforward computation indicates
	\begin{small}
		\begin{align}\label{dG2c}
			\frac{\dif^2}{\dif t^2}G_j(\lambda, tx_a)_{|t=0}
			=&\mathrm{Im}\left\{ \!b_j w \frac{\dif^2}{\dif t^2}I(\phi_j(t,w))_{|t=0}\!\right\}
			\!-\!2\mathrm{Im}\!\left\{  (\alpha_j \overline{w}+3\theta_j \overline{w}^{3})
			\frac{\dif}{\dif t}I(\phi_j(t,w))_{|t=0}\!\right\}\nonumber\\
			&-4(1-\lambda)\alpha_j\theta_j e_{2},
		\end{align}
	\end{small}where $e_{2}=\mathrm{Im}\{ \overline{w}^{2}\}$. It suffices for us to calculate
	terms $\frac{\dif^2}{\dif t^2}I(\phi_j(t,w))_{|t=0}$ and $\frac{\dif}{\dif t}I(\phi_j(t,w))_{|t=0}$.
	
	To this end, for bervity, we introduce some notations at first
	\begin{eqnarray*}
		A=b_jw-b_i\tau,\, B=\alpha_j w-\alpha_i \tau,\, C=\theta_jw^{3}-\theta_i\tau^{3},\, D=\alpha_i\overline{\tau}^2,\,E=3\theta_i \overline{\tau}^{4},
	\end{eqnarray*}
	hence
	\begin{eqnarray}\label{E:Iij}
		I_i(\phi_j(t,w))=\frac{1}{2\pi i}\int_{\mathbb {T}}\frac{\overline{A}+t(B+C)}{A+t(\overline{B}+\overline{C})}[b_i-t(D+E)]\dif \tau,
	\end{eqnarray}
	where $\phi_j(t,w)$ is defined in \eqref{phi}. For simplicity, we use $\int$ to instead the full notation $\frac{1}{2\pi i}\int_{\mathbb {T}}$.
	
	\subsubsection{Compute $\frac{\dif}{\dif t}I_i(\phi_j(t,w))_{|t=0}$}
	
	According to \eqref{E:Iij}, it is easy to check that
	\[
	\begin{split}
		\frac{\dif}{\dif t}I_i(\phi_j(t,w))_{|t=0}=&\int\frac{b_iB+b_iC-\overline{A}D-\overline{A}E}{A}\dif \tau
		-\int\frac{b_i\overline{A}(\overline{B}+\overline{C})}{A^2}\dif \tau\\
		=&\int\frac{b_i(\alpha_j w-\alpha_i \tau)}{b_jw-b_i\tau}\dif \tau
		+\int\frac{b_i(\theta_jw^{3}-\theta_i\tau^{3})}{b_jw-b_i\tau}\dif\tau
		\\
		-&\int\frac{3(b_j \overline{w}-b_i\overline{\tau})\theta_i\overline{\tau}^{4}}{b_jw-b_i\tau}\dif \tau
		-\int\frac{b_i(b_j\overline{w}-b_i\overline{\tau})(\alpha_j \overline{w}-\alpha_i\overline{\tau})}{(b_jw-b_i\tau)^2}\dif \tau\\
		-&\int\frac{b_i(b_j\overline{w}-b_i\overline{\tau})(\theta_j \overline{w}^{3}-\theta_i\overline{\tau}^{3})}{(b_jw-b_i\tau)^2}\dif \tau
		-\int\frac{(b_j\overline{w}-b_i\overline{\tau})\alpha_i\overline{\tau}^2}{b_jw-b_i\tau}\dif \tau.
	\end{split}
	\]
	Changing variables to get the homogeneity of $w$, we could find
	\[
	\begin{split}
		\frac{\dif}{\dif t}I_i(\phi_j(t,w))_{|t=0}=&w\int\frac{b_i(\alpha_j -\alpha_i \tau)}{b_j-b_i\tau}\dif \tau
		+w^{3}\int\frac{b_i(\theta_j-\theta_i\tau^{3})}{b_j-b_i\tau}\dif\tau
		\\
		&- \overline{w}^{5}\int\frac{3(b_j-b_i\overline{\tau})\theta_i\overline{\tau}^{4}}{b_j-b_i\tau}\dif \tau
		-\overline{w}^3\int\frac{b_i(b_j-b_i\overline{\tau})(\alpha_j-\alpha_i\overline{\tau})}{(b_j-b_i\tau)^2}\dif \tau\\
		&-\overline{w}^{5}\int\frac{b_i(b_j-b_i\overline{\tau})(\theta_j -\theta_i\overline{\tau}^{3})}{(b_j-b_i\tau)^2}\dif \tau
		-\overline{w}^3\int\frac{(b_j-b_i\overline{\tau})\alpha_i\overline{\tau}^2}{b_j-b_i\tau}\dif \tau\\
		=&\mu_{ij} w+\eta_{ij}w^{3}+\gamma_{ij}\overline{w}^3+\kappa_{ij}\overline{w}^{5}
		,
	\end{split}
	\]
	where $\mu_{ij}$, $\eta_{ij}$, $\gamma_{ij}$, $\kappa_{ij}$, $i,j=1,2$ are real constants such that
	\begin{eqnarray*}
		\begin{split}
			&\mu_{ij}=\int\frac{b_i(\alpha_j -\alpha_i \tau)}{b_j-b_i\tau}\dif \tau,\qquad \eta_{ij}=\int\frac{b_i(\theta_j-\theta_i\tau^{3})}{b_j-b_i\tau}\dif\tau,\\
			&\gamma_{ij}=-\int\frac{(b_j-b_i\overline{\tau})\alpha_i\overline{\tau}^2}{b_j-b_i\tau}\dif\tau
			-\int\frac{b_i(b_j-b_i\overline{\tau})(\alpha_j-\alpha_i\overline{\tau})}{(b_j-b_i\tau)^2}\dif \tau,\\
			& \kappa_{ij}=-\int\frac{3(b_j-b_i\overline{\tau})\theta_i\overline{\tau}^{4}}{b_j-b_i\tau}\dif \tau
			-\int\frac{b_i(b_j-b_i\overline{\tau})(\theta_j -\theta_i\overline{\tau}^{3})}{(b_j-b_i\tau)^2}\dif \tau.
		\end{split}
	\end{eqnarray*}
	
	Now we compute $\mu_{ij}$ first. With notations in \eqref{notation1}, when $i=j$ we may write
	\[
	\mu_{ii}=\int\alpha_i\dif\tau=0,
	\]
	where the last equality comes from the holomorphism of constants. Again with notations in \eqref{notation1}, we write
	\[
	\mu_{12}
	=\int\frac{(1 -b \tau)}{b-\tau}\dif \tau
	=\int\frac{(\tau -b )}{(b\tau-1)\tau^2}\dif \tau=b^2-1
	\]
	where the last two equalities comes from the residue theorem. Similarly, using \eqref{notation1} and the residue theorem at $\infty$ and the holomorphism of integrals, we have
	\[
	\mu_{21}
	=\int\frac{b(b - \tau)}{1-b\tau}\dif \tau
	=\int\frac{b(b\tau -1 )}{(\tau-b)\tau^2}\dif \tau=0.
	\]
	
	To compute $\eta_{ij}$, we just need to repeat the above computation process of $\mu_{ij}$, that is, using the holomorphism, when $i=j$ we write
	\[
	\eta_{ii}=\int\theta_i(\tau^2+\tau+1)\dif\tau=0,
	\]
	and with \eqref{notation1} and the residue theorem, we have
	\[
	\eta_{12}
	=-\int\frac{(a +ab \tau^3)}{b-\tau}\dif \tau
	=-\int\frac{(a\tau^3 +ab )}{(b\tau-1)\tau^4}\dif \tau=a(b^4+1)
	\]
	and
	\[
	\eta_{21}
	=\int\frac{b(ab +a \tau^3)}{1-b\tau}\dif \tau
	=\int\frac{b(ab\tau^3 +a )}{(\tau-b)\tau^4}\dif \tau=0.
	\]
	
	For $\gamma_{ij}$, when $i=j$, we could write
	\[
	\gamma_{ii}=-\int\frac{(1-\overline{\tau})\alpha_i\overline{\tau}^2}{1-\tau}\dif\tau
	-\int\frac{\alpha_i(1-\overline{\tau})^2}{(1-\tau)^2}\dif \tau=0
	\]
	because the integrands are extended holomorphically to $\mathbb{C}^*$. By \eqref{notation1}, we write
	\[
	\gamma_{12}=-\int\frac{(b-\overline{\tau})b\overline{\tau}^2}{b-\tau}\dif\tau
	-\int\frac{(b-\overline{\tau})(1-b\overline{\tau})}{(b-\tau)^2}\dif \tau
	\]
	where the integrands decay quickly more than $\frac{1}{\tau^2}$. So using the residue theorem at $\infty$, we have $\gamma_{12}=0$. Again, with \eqref{notation1}, we write
	\[
	\gamma_{21}=-\int\frac{(1-b\overline{\tau})\overline{\tau}^2}{1-b\tau}\dif\tau
	-\int\frac{b(1-b\overline{\tau})(b-\overline{\tau})}{(1-b\tau)^2}\dif \tau.
	\]
	Applying the residue theorem, we get
	\[
	\gamma_{21}=-\int\frac{(1-b\tau)\tau}{\tau-b}\dif\tau
	-\int\frac{b(1-b\tau)(b-\tau)}{(\tau-b)^2}\dif \tau=(b^3-b)+(b-b^3)=0.
	\]
	
	Finally, we compute $\kappa_{ij}$. We just need to repeat the computation process of $\gamma_{ij}$. When $i=j$, the integral terms of $\kappa_{ii}$ are  extended holomorphically, so
	\[
	\kappa_{ii}=-\int\frac{3(1-\overline{\tau})\theta_i\overline{\tau}^{4}}{1-\tau}\dif \tau
	-\int\frac{(1-\overline{\tau})(\theta_i -\theta_i\overline{\tau}^{3})}{(1-\tau)^2}\dif \tau=0.
	\]
	With \eqref{notation1} and using the residue theorem on the fast decay integrands, we get
	\[
	\kappa_{12}=-\int\frac{3(b-\overline{\tau})ab\overline{\tau}^{4}}{b-\tau}\dif \tau
	-\int\frac{(b-\overline{\tau})(-a -ab\overline{\tau}^{3})}{(b-\tau)^2}\dif \tau=0.
	\]
	Again with \eqref{notation1} and applying the residue theorem, we get
	\begin{align*}
		\kappa_{21}&=-\int\frac{3(1-b\overline{\tau})(-a)\overline{\tau}^{4}}{1-b\tau}\dif \tau
		-\int\frac{b(1-b\overline{\tau})(ab +a\overline{\tau}^{3})}{(1-b\tau)^2}\dif \tau\\
		&=
		\int\frac{3a(1-b\tau)\tau^{3}}{\tau-b}\dif \tau
		-\int\frac{b(1-b\tau)(ab +a\tau^{3})}{(\tau-b)^2}\dif \tau\\
		&=3ab^3(1-b^2)+a(4b^5-2b^3)=a(b^3+b^5).
	\end{align*}
	%By directly computation
	%\begin{eqnarray*}
	%&&\mu_{11}=\mu_{22}=\mu_{21}=0,\,\mu_{12}=b^2-1,\\
	%&& \eta_{11}=\eta_{22}=\eta_{21}=0,\,\eta_{12}=a(1+b^{4}),\\
	%&&\gamma_{11}=\gamma_{22}=\gamma_{12}=\gamma_{21}=0,\\
	%&& \kappa_{11}=\kappa_{22}=\kappa_{12}=0,\,\kappa_{21}=a(b^3+b^{4+1}),
	%\end{eqnarray*}
	
	Recall that $\frac{\dif}{\dif t}I_i(\phi_j(t,w))_{|t=0}=\mu_{ij} w+\eta_{ij}w^{3}+\gamma_{ij}\overline{w}^3+\kappa_{ij}\overline{w}^{5}
	$, with the computation results of these constants above and the definition of $I$ in \eqref{E:Cauchy-m}, we can easily see
	\begin{align}\label{dI1c}
		&\frac{\dif}{\dif t}I(\phi_1(t,w))_{|t=0}
		\!=\!\frac{\dif}{\dif t}I_1(\phi_1(t,w))_{|t=0}-\frac{\dif}{\dif t}I_2(\phi_1(t,w))_{|t=0}
		\!=\!-a(b^3+b^{5})\overline{w}^{5},\\
		& \frac{\dif}{\dif t}I(\phi_2(t,w))_{|t=0}
		\!=\!\frac{\dif}{\dif t}I_1(\phi_2(t,w))_{|t=0}-\frac{\dif}{\dif t}I_2(\phi_2(t,w))_{|t=0}
		\!=(b^2-1)w+a(1+b^{4})w^{3}.\nonumber
	\end{align}
	
	\subsubsection{Compute $\frac{\dif^2}{\dif t^2}I_i(\phi_j(t,w))_{|t=0}$}
	
	For $\frac{\dif^2}{\dif t^2}I_i(\phi_j(t,w))_{|t=0}$, we differentiate \eqref{E:Iij} with respect to $t$ twice directly and get
	\begin{small}
		\begin{eqnarray*}
			\begin{split}
				\frac{\dif^2}{\dif t^2}I_i(\phi_j(t,w))_{|t=0}\!=\!&-2\!\int\frac{BD+BE+CD+CE}{A}\dif \tau
				\!-\!2\!\int\frac{b_iB+b_iC-\overline{A}D-\overline{A}E}{A^2}(\overline{B}+\overline{C})\dif \tau\\
				& +2\!\int\frac{b_i\overline{A}(\overline{B}+\overline{C})^2}{A^3}\dif \tau.\\
			\end{split}
		\end{eqnarray*}
	\end{small}For clearly, we expand the above equality in the following way
	\begin{small}
		\begin{align*}
			&\frac{1}{2}\frac{\dif^2}{\dif t^2}I_i(\!\phi_j(t,w))_{|t=0}\\
			=\!&
			-\!\!\int\!\!\frac{(\alpha_j w-\alpha_i \tau)\alpha_i\overline{\tau}^2}{b_jw-b_i\tau}\dif \tau
			\!-\!\!\int\!\!\frac{3(\alpha_j w-\alpha_i \tau)\theta_i \overline{\tau}^{4}}{b_jw-b_i\tau}\dif \tau\\
			&-\!\!\int\!\!\frac{(\theta_jw^{3}-\theta_i\tau^{3})\alpha_i\overline{\tau}^2}{b_jw-b_i\tau}\dif\tau
			\!-\!\!\int\!\!\frac{3(\theta_jw^{3}-\theta_i\tau^{3})\theta_i\overline{\tau}^{4}}{b_jw-b_i\tau}\dif\tau\\
			&-\!\!\int\!\!\frac{b_i(\alpha_j w-\alpha_i \tau)(\alpha_j \overline{w}-\alpha_i\overline{\tau})}{(b_jw-b_i\tau)^2}\dif \tau
			\!-\!\!\int\!\!\frac{b_i(\theta_jw^{3}-\theta_i\tau^{3})(\alpha_j \overline{w}-\alpha_i\overline{\tau})}{(b_jw-b_i\tau)^2}\dif \tau
			\\
			&+\!\!\int\!\!\frac{(b_j\overline{w}-b_i\overline{\tau})(\theta_j\overline{w}^{4-1}-\theta_i\overline{\tau}^{4-1})\alpha_i\overline{\tau}^2}{(b_jw-b_i\tau)^2}\dif \tau
			+\!\!\int\!\!\frac{(b_j \overline{w}-b_i\overline{\tau})(\alpha_j \overline{w}-\alpha_i\overline{\tau})\alpha_i\overline{\tau}^{2}}{(b_jw-b_i\tau)^2}\dif \tau
			\\
			&
			-\!\!\int\!\!\frac{b_i(\alpha_j w-\alpha_i \tau)(\theta_j \overline{w}^{3}-\theta_i\overline{\tau}^{3})}{(b_jw-b_i\tau)^2}\dif \tau
			+\!\!\int\frac{3(b_j\overline{w}-b_i\overline{\tau})(\alpha_j \overline{w}-\alpha_i\overline{\tau})\theta_i\overline{\tau}^{4}}{(b_jw-b_i\tau)^2}\dif \tau
			\\
			&-\!\!\int\!\!\frac{b_i(\theta_jw^{3}-\theta_i\tau^{3})(\theta_j \overline{w}^{3}-\theta_i\overline{\tau}^{3})}{(b_jw-b_i\tau)^2}\dif \tau
			+\!\!\int\!\!\frac{(3)(b_j\overline{w}-b_i\overline{\tau})(\theta_j \overline{w}^{3}-\theta_i\overline{\tau}^{3})\theta_i\overline{\tau}^{4}}{(b_jw-b_i\tau)^2}\dif \tau\\
			&+\!\!\int\!\!\frac{b_i(b_j\overline{w}-b_i\overline{\tau})(\alpha_j \overline{w}-\alpha_i\overline{\tau})^2}{(b_jw-b_i\tau)^3}\dif \tau
			+\!\!\int\!\!\frac{b_i(b_j\overline{w}-b_i\overline{\tau})(\theta_j \overline{w}^{4-1}-\theta_i\overline{\tau}^{4-1})^2}{(b_jw-b_i\tau)^3}\dif \tau\\
			&+\!\!\int\!\!\frac{2b_i(b_j\overline{w}-b_i\overline{\tau})(\alpha_j \overline{w}-\alpha_i\overline{\tau})(\theta_j \overline{w}^{4-1}-\theta_i\overline{\tau}^{4-1})}{(b_jw-b_i\tau)^3}\dif \tau
		\end{align*}
		Similar as the method to compute $\frac{\dif}{\dif t}I_i(\!\phi_j(t,w))_{|t=0}$, we write the homogeneity of $w$ by the change of variables
		\begin{align*}
			&\frac{1}{2}\frac{\dif^2}{\dif t^2}I_i(\phi_j(t,w))_{|t=0}\\
			=&
			-\overline{w}\!\!\int\!\!\frac{(\alpha_j -\alpha_i \tau)\alpha_i\overline{\tau}^2}{b_j-b_i\tau}\dif \tau
			-\overline{w}^{3}\!\!\int\!\!\frac{3(\alpha_j -\alpha_i \tau)\theta_i \overline{\tau}^{4}}{b_j-b_i\tau}\dif \tau
			-w\!\!\int\!\!\frac{(\theta_j-\theta_i\tau^{3})\alpha_i\overline{\tau}^2}{b_j-b_i\tau}\dif\tau
			\\&
			-\overline{w}\!\!\int\!\!\frac{3(\theta_j-\theta_i\tau^{3})\theta_i\overline{\tau}^{4}}
			{b_j-b_i\tau}\dif\tau
			-\overline{w}\!\!\int\!\!\frac{b_i(\alpha_j -\alpha_i \tau)(\alpha_j-\alpha_i\overline{\tau})}
			{(b_j-b_i\tau)^2}\dif \tau
			\\&
			-\overline{w}\!\!\int\!\!\frac{b_i(\theta_j-\theta_i\tau^{3})(\theta_j -\theta_i\overline{\tau}^{3})}
			{(b_j-b_i\tau)^2}\dif \tau
			-\!w\!\!\int\!\!\frac{b_i(\theta_j-\theta_i\tau^{3})(\alpha_j-\alpha_i\overline{\tau})}
			{(b_j-b_i\tau)^2}\dif \tau
			\\&
			-\overline{w}^{3}\!\!\int\!\!\frac{b_i(\alpha_j-\alpha_i \tau)(\theta_j-\theta_i\overline{\tau}^{3})}{(b_j-b_i\tau)^2}\dif \tau
			+\overline{w}^5\!\!\int\!\!\frac{(b_j-b_i\overline{\tau})(\alpha_j-\alpha_i\overline{\tau})\alpha_i\overline{\tau}^{2}}
			{(b_j-b_i\tau)^2}\dif \tau
			\\&
			+\overline{w}^{7}\!\!\int\!\!\frac{(b_j-b_i\overline{\tau})(\theta_j-\theta_i\overline{\tau}^{3})\alpha_i\overline{\tau}^2}
			{(b_j-b_i\tau)^2}\dif \tau\
			+\overline{w}^{7}\!\!\int\!\!\frac{3(b_j-b_i\overline{\tau})(\alpha_j-\alpha_i\overline{\tau})
				\theta_i\overline{\tau}^{4}}
			{(b_j-b_i\tau)^2}\dif \tau
			\\&+\overline{w}^{9}\!\!\int\!\!\frac{3(b_j-b_i\overline{\tau})(\theta_j-\theta_i\overline{\tau}^{3})
				\theta_i\overline{\tau}^{4}}{(b_j-b_i\tau)^2}\dif \tau
			+\overline{w}^5\!\!\int\!\!\frac{b_i(b_j-b_i\overline{\tau})(\alpha_j-\alpha_i\overline{\tau})^2}
			{(b_j-b_i\tau)^3}\dif \tau
			\\&
			+\overline{w}^{9}\!\!\int\!\!\frac{b_i(b_j-b_i\overline{\tau})(\theta_j-\theta_i\overline{\tau}^{4-1})^2}
			{(b_j-b_i\tau)^3}\dif \tau
			+\overline{w}^{7}\!\!\int\!\!\frac{2b_i(b_j-b_i\overline{\tau})(\alpha_j-\alpha_i\overline{\tau})(\theta_j-\theta_i\overline{\tau}^{3})}
			{(b_j-b_i\tau)^3}\dif \tau\\
			=&
			\overline{\mu}_{ij}w +\overline{\eta}_{ij}\overline{w}+\overline{\kappa}_{ij}\overline{w}^{3}+\overline{\gamma}_{ij}\overline{w}^5
			+\overline{\nu}_{ij}\overline{w}^{7}+\overline{\rho}_{ij}\overline{w}^{9},
		\end{align*}
	\end{small}
	where $\overline{\mu}_{ij}$, $\overline{\eta}_{ij}$, $\overline{\kappa}_{ij}$, $\overline{\gamma}_{ij}$, $\overline{\nu}_{ij}$, $\overline{\rho}_{ij}$, $i,j=1,2$ are real constants and we write these constants in the following form
	\begin{eqnarray*}
		\overline{\mu}_{ij}&=&
		-\!\int\!\frac{(\theta_j-\theta_i\tau^{3})\alpha_i\overline{\tau}^2}{b_j-b_i\tau}\dif\tau
		-\int\frac{b_i(\theta_j-\theta_i\tau^{3})(\alpha_j-\alpha_i\overline{\tau})}{(b_j-b_i\tau)^2}\dif \tau,\\
		\overline{\eta}_{ij}&=&
		-\int\frac{(\alpha_j -\alpha_i \tau)\alpha_i\overline{\tau}^2}{b_j-b_i\tau}\dif \tau
		-\int\frac{3(\theta_j-\theta_i\tau^{3})\theta_i\overline{\tau}^{4}}{b_j-b_i\tau}\dif\tau\\
		&&-\int\frac{b_i(\alpha_j -\alpha_i \tau)(\alpha_j -\alpha_i\overline{\tau})}{(b_j-b_i\tau)^2}\dif \tau
		-\int\frac{b_i(\theta_j-\theta_i\tau^{3})(\theta_j -\theta_i\overline{\tau}^{3})}{(b_j-b_i\tau)^2}\dif \tau ,\\
		\overline{\kappa}_{ij}&=&
		-\int\frac{3(\alpha_j -\alpha_i \tau)\theta_i \overline{\tau}^{4}}{b_j-b_i\tau}\dif \tau
		-\int\frac{b_i(\alpha_j-\alpha_i \tau)(\theta_j-\theta_i\overline{\tau}^{3})}{(b_j-b_i\tau)^2}\dif \tau\\
		\overline{\gamma}_{ij}&=&
		\int\frac{(b_j-b_i\overline{\tau})(\alpha_j-\alpha_i\overline{\tau})\alpha_i\overline{\tau}^{2}}{(b_j-b_i\tau)^2}\dif \tau
		+\int\frac{b_i(b_j-b_i\overline{\tau})(\alpha_j-\alpha_i\overline{\tau})^2}{(b_j-b_i\tau)^3}\dif \tau\\
		\overline{\nu}_{ij}&=&
		\int\frac{(b_j-b_i\overline{\tau})(\theta_j-\theta_i\overline{\tau}^{3})
			\alpha_i\overline{\tau}^2}{(b_j-b_i\tau)^2}\dif \tau
		+\int\frac{3(b_j-b_i\overline{\tau})(\alpha_j-\alpha_i\overline{\tau})\theta_i\overline{\tau}^{4}}{(b_j-b_i\tau)^2}\dif \tau\\
		&&+\int\frac{2b_i(b_j-b_i\overline{\tau})(\alpha_j-\alpha_i\overline{\tau})(\theta_j-\theta_i\overline{\tau}^{3})}{(b_j-b_i\tau)^3}\dif \tau\\
		\overline{\rho}_{ij}&=&
		\int\frac{3(b_j-b_i\overline{\tau})(\theta_j-\theta_i\overline{\tau}^{3})
			\theta_i\overline{\tau}^{4}}{(b_j-b_i\tau)^2}\dif \tau
		+\int\frac{b_i(b_j-b_i\overline{\tau})(\theta_j-\theta_i\overline{\tau}^{3})^2}{(b_j-b_i\tau)^3}\dif \tau.
	\end{eqnarray*}
	
	\textbf{Some useful identities.} There is the position for us to give some useful identities. Applying the residue theorem at $\infty$, we have
	\begin{align}\label{E:0}
		\int\frac{\tau^{k_1}\overline{\tau}^{k_2}}{(1-\tau)^{k_3}} \dif \tau
		=&\int\frac{\tau^{k_2+k_3-k_1-2}}{(\tau-1)^{k_3}} \dif \tau,\nonumber\\
		\int\frac{\tau^{k_1}\overline{\tau}^{k_2}}{(b-\tau)^{k_3}} \dif \tau
		=&\int\frac{\tau^{k_2+k_3-k_1-2}}{(b\tau-1)^{k_3}} \dif \tau,\\
		\int\frac{\tau^{k_1}\overline{\tau}^{k_2}}{(1-b\tau)^{k_3}} \dif \tau
		=&\int\frac{\tau^{k_2+k_3-k_1-2}}{(\tau-b)^{k_3}} \dif \tau,\nonumber
	\end{align}
	where $k_1$, $k_2$ are nonnegative integers and $k_3$ is a positive integer. Denote $k_2+k_3-k_1-2=-k$. Then using the residue theorem, when $k_2+k_3-k_1-2\ge 0$, due to the holomorphism of the integrals, it is easy to check that $\int\frac{\tau^{k_2+k_3-k_1-2}}{(\tau-1)^{k_3}} \dif \tau=0$, which is,
	\begin{align}\label{E:01}
		\int\frac{\tau^{k_1}\overline{\tau}^{k_2}}{(1-\tau)^{k_3}} \dif \tau=0;
	\end{align}
	when $k_2+k_3-k_1-2< 0$, again using the residue theorem, we get
	\begin{align}\label{E:1}
		\int\frac{\tau^{k_1}\overline{\tau}^{k_2}}{(1-\tau)^{k_3}} \dif \tau
		=&\int\frac{1}{(\tau-1)^{k_3}\tau^k} \dif \tau
		=\frac{(-1)^{k_3}(k_3+k-2)!}{(k-1)!(k_3-1)!}\nonumber\\
		=&\frac{(-1)^{k_3}(k_1-k_2)!}{(1+k_1-k_2-k_3)!(k_3-1)!}.
	\end{align}
	Similarly, we obtain that when $k_2+k_3-k_1-2\ge 0$,
	\begin{align}\label{E:02}
		\int\frac{\tau^{k_1}\overline{\tau}^{k_2}}{(b-\tau)^{k_3}} \dif \tau=0
	\end{align}
	and when $k_2+k_3-k_1-2< 0$
	\begin{align}\label{E:2}
		\int\frac{\tau^{k_1}\overline{\tau}^{k_2}}{(b-\tau)^{k_3}} \dif \tau
		=\frac{(-1)^{k_3}(k_1-k_2)!}{(1+k_1-k_2-k_3)!(k_3-1)!}b^{1+k_1-k_2-k_3}.
	\end{align}
	Now we compute the value of the last identity, $\int\frac{\tau^{k_1}\overline{\tau}^{k_2}}{(1-b\tau)^{k_3}} \dif \tau
	=\int\frac{\tau^{-k}}{(\tau-b)^{k_3}} \dif \tau$. When $-k=k_2+k_3-k_1-2\ge 0$, applying the residue theorem, we obtain that
	\begin{align}\label{E:03}
		\int\frac{\tau^{k_1}\overline{\tau}^{k_2}}{(1-b\tau)^{k_3}} \dif \tau
		=&\int\frac{\tau^{-k}}{(\tau-b)^{k_3}} \dif \tau
		=\frac{(-k)!}{(-k-k_3+1)!(k_3-1)!}b^{-k-k_3+1}\nonumber\\
		=&\left\{
		\begin{array}{ccc}
			\frac{(k_2+k_3-k_1-2)!}{(k_2-k_1-1)!(k_3-1)!}b^{k_2-k_1-1}&\text{ for } k_2-k_1-1\ge 0, \\
			0 &\text{ for } k_2-k_1-1< 0,
		\end{array}
		\right.
	\end{align}
	or we can set (negative integer)$!=\infty$;
	when $-k=k_2+k_3-k_1-2< 0$, again applying the residue theorem, we get
	\begin{align}\label{E:3}
		&\int\frac{\tau^{k_1}\overline{\tau}^{k_2}}{(1-b\tau)^{k_3}} \dif \tau
		=\int\frac{1}{(\tau-b)^{k_3}\tau^{k}} \dif \tau\nonumber\\
		=&\frac{(-1)^{-k_3}(k+k_3-2)!}{(k_3-1)!(k-1)!}b^{-k-k_3+1}
		+\frac{(-1)^{k_3-1}(k+k_3-2)!}{(k_3-1)!(k-1)!}b^{-k-k_3+1}
		\nonumber\\
		=&0.
	\end{align}

	\subsubsection{Compute $\overline{\mu}_{ij}$}
	
	Now we use the above equalities to compute $\overline{\mu}_{ij}=-\!\int\!\frac{(\theta_j-\theta_i\tau^{3})\alpha_i\overline{\tau}^2}{b_j-b_i\tau}\dif\tau
	-\int\frac{b_i(\theta_j-\theta_i\tau^{3})(\alpha_j-\alpha_i\overline{\tau})}{(b_j-b_i\tau)^2}\dif \tau$. First, we rewrite
	\begin{align*}%\label{omuij}
		\overline{\mu}_{ij}
		=&-\int\frac{\theta_j\alpha_i\overline{\tau}^2}{b_j-b_i\tau}\dif\tau
		+\int\frac{\theta_i\alpha_i\tau^{3}\overline{\tau}^2}{b_j-b_i\tau}\dif\tau
		-\int\frac{b_i\theta_j\alpha_j}{(b_j-b_i\tau)^2}\dif \tau
		\nonumber\\&
		+\int\frac{b_i\theta_i\alpha_j\tau^{3}}{(b_j-b_i\tau)^2}\dif \tau
		+\int\frac{b_i\theta_j\alpha_i\overline{\tau}}{(b_j-b_i\tau)^2}\dif \tau
		-\int\frac{b_i\theta_i\alpha_i\tau^{3}\overline{\tau}}{(b_j-b_i\tau)^2}\dif \tau
	\end{align*}
	For $i=j$, we can easily obtain
	\begin{align*}%\label{omuij}
		\overline{\mu}_{ii}
		=&-\int\frac{\theta_i\alpha_i\overline{\tau}^2}{b_i(1-\tau)}\dif\tau
		+\int\frac{\theta_i\alpha_i\tau^{3}\overline{\tau}^2}{b_i(1-\tau)}\dif\tau
		-\int\frac{\theta_i\alpha_i}{b_i(1-\tau)^2}\dif \tau
		\nonumber\\&
		+\int\frac{\theta_i\alpha_i\tau^{3}}{b_i(1-\tau)^2}\dif \tau
		+\int\frac{\theta_i\alpha_i\overline{\tau}}{b_i(1-\tau)^2}\dif \tau
		-\int\frac{\theta_i\alpha_i\tau^{3}\overline{\tau}}{b_i(1-\tau)^2}\dif \tau.
	\end{align*}
	For the first term $-\int\frac{\theta_i\alpha_i\overline{\tau}^2}{b_i(1-\tau)}\dif\tau$, we observe that $k_1=0$, $k_2=2$ and $k_3=1$ with the notations in the \eqref{E:0}, so $k_2+k_3-k_1-2=1>0$. Applying \eqref{E:01}, we get $-\int\frac{\theta_i\alpha_i\overline{\tau}^2}{b_i(1-\tau)}\dif\tau=0$. Similarly, the third and the fifth term are equal to zero, that is,
	$-\int\frac{\theta_i\alpha_i}{b_i(1-\tau)^2}\dif \tau
	=\int\frac{\theta_i\alpha_i\overline{\tau}}{b_i(1-\tau)^2}\dif \tau=0$. For the second term, $\int\frac{\theta_i\alpha_i\tau^{3}\overline{\tau}^2}{b_i(1-\tau)}\dif\tau$, again, we find $k_1=3$, $k_2=2$ and $k_3=1$ with the notations in the \eqref{E:0}, so $k_2+k_3-k_1-2=-2<0$. From \eqref{E:1}, we have
	\[
	\int\frac{\theta_i\alpha_i\tau^{3}\overline{\tau}^2}{b_i(1-\tau)}\dif\tau
	=\frac{(-1)(1!)}{(1)!(0)!}\theta_i\alpha_i
	=-\frac{\theta_i\alpha_i}{b_i}.
	\]
	By the same idea, we could find the value of the forth and the sixth term,
	\begin{align*}
		\int\frac{\theta_i\alpha_i\tau^{3}}{b_i(1-\tau)^2}\dif \tau=3\frac{\theta_i\alpha_i}{b_i},\,
		-\int\frac{\theta_i\alpha_i\tau^{3}\overline{\tau}}{b_i(1-\tau)^2}\dif \tau=-2\frac{\theta_i\alpha_i}{b_i}
	\end{align*}
	Inserting these values of every term into the formula of $\overline{\mu}_{ii}$, we know $\overline{\mu}_{ii}=0$.
	Next, we consider $\overline{\mu}_{12}$ with \eqref{notation1},
	\begin{align*}%\label{omuij}
		\overline{\mu}_{12}
		=&-\int\frac{\theta_2\alpha_1\overline{\tau}^2}{b-\tau}\dif\tau
		+\int\frac{\theta_1\alpha_1\tau^{3}\overline{\tau}^2}{b-\tau}\dif\tau
		-\int\frac{\theta_2\alpha_2}{(b-\tau)^2}\dif \tau
		\nonumber\\&
		+\int\frac{\theta_1\alpha_2\tau^{3}}{(b-\tau)^2}\dif \tau
		+\int\frac{\theta_2\alpha_1\overline{\tau}}{(b-\tau)^2}\dif \tau
		-\int\frac{\theta_1\alpha_1\tau^{3}\overline{\tau}}{(b-\tau)^2}\dif \tau.
	\end{align*}
	As the computation of $\overline{\mu}_{ii}$, the first, the third and the fifth terms are zero from \eqref{E:02} and we only need to compute the rest terms using \eqref{E:2}. With the notations in \eqref{E:0}, to compute the second term, we know $k_1=3$, $k_2=2$, $k_3=1$ and hence by \eqref{notation1} and \eqref{E:2},
	\[
	\int\frac{\theta_1\alpha_1\tau^{3}\overline{\tau}^2}{b-\tau}\dif\tau=-\theta_1\alpha_1 b=-a b^3.
	\]
	Using the same way, we have
	\[
	\int\frac{\theta_1\alpha_2\tau^{3}}{(b-\tau)^2}\dif \tau=3\theta_1\alpha_2 b^2=3 a b^3,\,
	-\int\frac{\theta_1\alpha_1\tau^{3}\overline{\tau}}{(b-\tau)^2}\dif \tau=-2 \theta_1\alpha_1 b= -2 a b^3.
	\]
	Inserting these results into the formula of $\overline{\mu}_{12}$, we have $\overline{\mu}_{12}=0$. Finally, we compute $\overline{\mu}_{21}$ with \eqref{notation1},
	\begin{align*}%\label{omuij}
		\overline{\mu}_{21}
		=&-\int\frac{\theta_1\alpha_2\overline{\tau}^2}{(1-b\tau)}\dif\tau
		+\int\frac{\theta_2\alpha_1\tau^{3}\overline{\tau}^2}{b(1-b\tau)}\dif\tau
		-\int\frac{b\theta_1\alpha_1}{(1-b\tau)^2}\dif \tau
		\nonumber\\&
		+\int\frac{b\theta_2\alpha_1\tau^{3}}{(1-b\tau)^2}\dif \tau
		+\int\frac{b\theta_1\alpha_2\overline{\tau}}{(1-b\tau)^2}\dif \tau
		-\int\frac{b\theta_2\alpha_2\tau^{3}\overline{\tau}}{(1-b\tau)^2}\dif \tau.
	\end{align*}
	From \eqref{E:3}, we know the second, the forth and the sixth terms are zero and we need to compute the rest terms by \eqref{E:03}. For the first term $-\int\frac{\theta_1\alpha_2\overline{\tau}^2}{b(1-b\tau)}\dif\tau$, observe that $k_1=0$, $k_2=2$ and $k_3=1$ with notations in \eqref{E:0}. So it is easy to check
	\[
	-\int\frac{\theta_1\alpha_2\overline{\tau}^2}{(1-b\tau)}\dif\tau=-\theta_1\alpha_2b= -a b^2.
	\]
	For the third term and the fifth term, we can compute
	\[
	-\int\frac{b\theta_1\alpha_1}{(1-b\tau)^2}\dif \tau=0, \,
	\int\frac{b\theta_1\alpha_2\overline{\tau}}{(1-b\tau)^2}\dif \tau=b\theta_1\alpha_2=a b^2  .
	\]
	Also, inserting these terms into the formula of $\overline{\mu}_{21}$, we get $\overline{\mu}_{21}=0$. Actually, using the method of computing $\frac{\dif}{\dif t}I_i(\phi_j(t,w))_{|t=0}$, which is directly applying the residue theorem at $\infty$ and using the holomorphism of some integers, we can also obtain the results that $\overline{\mu}_{ij}=0$ for $i,j=1,2$.
	
	Using the same method of calculating $\overline{\mu}_{ij}$, applying \eqref{E:01}, \eqref{E:1}, \eqref{E:02}, \eqref{E:2}, \eqref{E:03} and \eqref{E:3}
	and after a series of calculations with some cancellations, we have
	\begin{eqnarray*}
		%&&\overline{\mu}_{11}=\overline{\mu}_{22}=\overline{\mu}_{12}=\overline{\mu}_{21}=0,
		%\\
		&&\overline{\eta}_{11}=b^2+3a^2b^2,\,\overline{\eta}_{22}=b^{-1}+3a^2b^{-1},\,
		\overline{\eta}_{12}=b-3a^2b^{3},\, \overline{\eta}_{21}=1+3a^2,
		\\
		&&\overline{\gamma}_{11}=\overline{\gamma}_{22}=\overline{\gamma}_{12}=\overline{\gamma}_{21}=0,
		\\
		&&\overline{\kappa}_{11}=ab^2,\,\overline{\kappa}_{22}=-ab^{-1},\,\overline{\kappa}_{12}=-ab,\,\overline{\kappa}_{21}=-ab^{2},\\
		&&\overline{\nu}_{11}=\overline{\nu}_{22}=\overline{\nu}_{12}=0,\,\overline{\nu}_{21}=a(b^4+b^{6}),\\
		&&\overline{\rho}_{11}=\overline{\rho}_{22}=\overline{\rho}_{12}=0,\,\overline{\rho}_{21}=-a^2(b^4+5b^{6}+4b^{8}).
	\end{eqnarray*}
	We would like to omit details in these calculation since the arguments is exactly the same as the previous analysis. \\

	%\red{Since the methods using in the computation process are totally the same as above and it is too long to write, we omit it here and in the following of this article.}
	
	Recall the expression that $\frac{1}{2}\frac{\dif^2}{\dif t^2}I_i(\phi_j(t,w))_{|t=0}=\overline{\mu}_{ij}w +\overline{\eta}_{ij}\overline{w}+\overline{\kappa}_{ij}\overline{w}^{3}+\overline{\gamma}_{ij}\overline{w}^5
	+\overline{\nu}_{ij}\overline{w}^{7}+\overline{\rho}_{ij}\overline{w}^{9}$. Then with the value of these constants, we have
	\begin{align}\label{dI2c1}
		\frac{\dif^2}{\dif t^2}I(\phi_1(t,w))_{|t=0}
		&=\frac{\dif^2}{\dif t^2}I_1(\phi_1(t,w))_{|t=0}-\frac{\dif^2}{\dif t^2}I_2(\phi_1(t,w))_{|t=0}\nonumber\\
		&=2(b^2-1)(1+3a^2)\overline{w}
		+4a b^2\overline{w}^{3}\nonumber\\
		&-2a(b^4+b^{6})\overline{w}^{7}
		+2a^2(b^4+5b^{6}+4b^{8})\overline{w}^{9},
	\end{align}
	and
	\begin{align}\label{dI2c2}
		\frac{\dif^2}{\dif t^2}I(\phi_2(t,w))_{|t=0}
		&=\frac{\dif^2}{\dif t^2}I_1(\phi_2(t,w))_{|t=0}-\frac{\dif^2}{\dif t^2}I_2(\phi_2(t,w))_{|t=0}\nonumber\\
		&=2[b-b^{-1}-3a^2(b^{-1}+b^{3})]\overline{w}
		+2a(b^{-1}-b)\overline{w}^{3}.
	\end{align}
	Inserting \eqref{dI1c}, \eqref{dI2c1} and \eqref{dI2c2} into \eqref{dG2c}, we have
	\begin{small}
		\begin{align*}
			\!\frac{\dif^2}{\dif t^2}G_1(\lambda_{4},tx_a)_{|t=0}
			\!=&4a b^2e_{2}\!-\!2a(b^4+b^{6})e_{6}+2a^2(b^4+5b^{6}+4b^{8})e_{8}\\
			&+2ab(b^3+b^{5})e_{6}\!+\!6a^2(b^4+b^{6})e_{8}\!-\!4a(1-\lambda_{4})b^2e_{2}\\
			=&2a(b^{2}+b^4)e_{2}+8a^2(b^4+2b^{6}+b^{8}) e_{8}, \\
			\frac{\dif^2}{\dif t^2}G_2(\lambda_{4},tx_a)_{|t=0}
			\!=&2a(1-b^2)e_{2}+2a(1+b^{4})e_{2}+6a(b^2-1)e_{2}\\
			&+4(1-\lambda_{4})ae_{2}\\
			=&2a(b^2+b^{4})e_{2},
		\end{align*}
	\end{small}
	where $\lambda_4=\frac{1+b^2}{2}$ and $b=b_4$. With the definition of projection operator $Q$ in Section \ref{S:Pre} and $G=(G_1,G_2)$, we have
	\begin{eqnarray*}
		Q\frac{\dif^2}{\dif t^2}G(\lambda_2,tx_a)_{|t=0}
		=
		\left\langle
		\left(
		\begin{array}{cc}
			2a(b^{2}+b^4)\\
			2a(b^2+b^{4})
		\end{array}
		\right),
		\left(
		\begin{array}{cc}
			\frac{1}{\sqrt{2}}\\
			-\frac{1}{\sqrt{2}}
		\end{array}
		\right)
		\right\rangle
		y_1=0,
	\end{eqnarray*}
	which completes our proof.
	
\end{proof}

\subsection{Second-order derivatives} \label{SS:Hessian}

Moreover, we continue to compute second-order derivatives of the reduced functional $F_2$.
\begin{proposition}\label{d2F2}
Let  $p \geq 2$ and $b_{2p}^{2p} = p -1-pb_{2p}^2$. Let
\[
	y_1=
	\left(
	\begin{array}{cc}
		\frac{1}{\sqrt{2}}\\
		-\frac{1}{\sqrt{2}}
	\end{array}
	\right)e_{2}, \quad
	y_2=-
	\left(
	\begin{array}{cc}
		\frac{1}{\sqrt{2}}\\
		\frac{1}{\sqrt{2}}
	\end{array}
	\right)e_{2p}
\]
be the basis of the co-kernel of $\mathcal{L}_{\frac{1+b_{2p}^2}{2},b_{2p}}$. The second-order derivatives of $F_2$ are calculated as follows:
\begin{enumerate}
		\item $\partial_{\lambda\lambda} F_2(\lambda_{2p},0;a)=4\sqrt{2}\left(b_{2p}^{-1}y_1+ap^2b_{2p}^{1-2p}y_2\right).$
		\item $\partial_t\partial_\lambda F_2(\lambda_{2p},0;a)=0.$
		\item
		\[\partial_{tt} F_2(\lambda_{4},0;a)= \sqrt{2}\left[-b_4^{-1}(b_4^2-1)^2+c_1\right]y_1 +\sqrt{2}\left[b_4^{-1}(b_4^2-1)^2(4a+6a^3)+c_2\right]y_2,
		\]
		where
		\begin{eqnarray*}
			c_1=-2a^2b_4(b_4^2+1)^2 ,\quad
			c_2= \frac{2 a^3 (b_4^2-1)^2(4b_4^2-3)}{  b(2b_4^4-1)}.
		\end{eqnarray*}
		for $p=2$. On the other hand, for $p \geq 3$, we have
		\[
		\begin{split}
			\partial_{tt} F_2(\lambda_{2p},0;a)& = \sqrt{2}\left[-b_{2p}^{-1}(b_{2p}^2-1)^2+c_1\right]y_1 \\
			& +\sqrt{2}\left\{b_{2p}^{-1}(b_{2p}^2-1)^2[2pa+(2p-1)pa^3]+c_2\right\}y_2,
		\end{split}
		\]
		where
		\[
		c_1=-2a^2(b_{2p}^2-1)^2b_{2p}^{-1}[p+(p-1)b_{2p}^2],
		\]
		\[
		c_2= \frac{4a^3(p-1)^4  (b_{2p}^2-1)^2 [2pb_{2p}^3+(1-2p)b_{2p}]}{ b_{2p}^2[2+4p(b_{2p}^2-1)+p^2(b_{2p}^2-1)^2]}.
		\]
	\end{enumerate}	
\end{proposition}

\begin{proof}[Proof of Proposition \ref{d2F2}]
(1)For $p=2$,
$x_a=\left(
\begin{array}{cc}
b\\
1
\end{array}
\right)
\overline{w}
+a
\left(
\begin{array}{cc}
b\\
-1
\end{array}
\right)
\overline{w}^{3}$.
Since $\partial_{\lambda\lambda}G(\lambda_{4},0)=0$ is obviously by definition of $G$, from Proposition \ref{acformal}, we know
\begin{eqnarray*}
\partial_{\lambda\lambda} F_2(\lambda_{4},0;a)=-2Q\partial_{\lambda}\partial_f G(\lambda_{4},0) \left[\partial_f G(\lambda_{4},0)\right]^{-1}({\rm {Id}}-Q)\partial_{\lambda}\partial_f G(\lambda_{4},0)x_a.
\end{eqnarray*}

By the representation of $\partial_fG$ in \eqref{dG1} and the definition of the projection region of $Q$ in \eqref{E:k}, we have
\begin{eqnarray*}
\begin{split}
\partial_{\lambda}\partial_f G(\lambda_{4},0)x_a&=
\left(
\begin{array}{cc}
2&0\\
0&2b
\end{array}
\right)
\left(
\begin{array}{cc}
b\\
1
\end{array}
\right)e_2
+a
\left(
\begin{array}{cc}
4&0\\
0&4b
\end{array}
\right)
\left(
\begin{array}{cc}
b\\
-1
\end{array}
\right)e_{4}\\
&=
2b
\left(
\begin{array}{cc}
1\\
1
\end{array}
\right)e_2
+
4ab
\left(
\begin{array}{cc}
1\\
-1
\end{array}
\right)e_{4}\\
&=(\mathrm{Id}-Q)\partial_{\lambda}\partial_f G(\lambda_{4},0)x_a.
\end{split}
\end{eqnarray*}
With the relationships in \eqref{iG}, we could easily compute the inverse of $\partial_fG$
\begin{eqnarray*}
-\left[\partial_f G(\lambda_{4},0)\right]^{-1}({\rm {Id}}-Q)\partial_{\lambda}\partial_f G(\lambda_{4},0)x_a
=2b^{-1}
\left(
\begin{array}{cc}
1\\
0
\end{array}
\right)\overline{w}
-
4ab^{-3}
\left(
\begin{array}{cc}
1\\
0
\end{array}
\right)\overline{w}^{3},
\end{eqnarray*}
that is,
\begin{eqnarray}\label{E:philga2}
\tilde{x}=\partial_\lambda\partial_g \varphi(\lambda_{4},0)x_a
=2b^{-1}
\left(
\begin{array}{cc}
1\\
0
\end{array}
\right)\overline{w}
-
4ab^{-3}
\left(
\begin{array}{cc}
1\\
0
\end{array}
\right)\overline{w}^{3}.
\end{eqnarray}
Again, by \eqref{dG1} and the definition of the projection operator $Q$ in Section \ref{S:Pre}, we have
\begin{eqnarray*}
\begin{split}
\partial_{\lambda\lambda} F_2(\lambda_{4},0;a)&=Q\left[
4b^{-1}
\left(
\begin{array}{cc}
2&0\\
0&2b
\end{array}
\right)
\left(
\begin{array}{cc}
1\\
0
\end{array}
\right)e_2
-8ab^{-3}
\left(
\begin{array}{cc}
4&0\\
0&4b
\end{array}
\right)
\left(
\begin{array}{cc}
1\\
0
\end{array}
\right)e_{4}
\right]\\
&=
4\sqrt{2}(b^{-1}y_1+4ab^{-3}y_2).
\end{split}
\end{eqnarray*}
The same approach can be used to obtain the result for $p\ge3$, which requires complex computations that we omit here.

(2) For $p=2$, from Proposition \ref{acformal}, we have
\begin{eqnarray*}
\begin{split}
\partial_t\partial_\lambda F_2(\lambda_{4},0;a)=&
\frac{1}{2}Q\partial_\lambda\partial_{ff}G(\lambda_{4},0)[x_a,x_a]
+Q\partial_{ff}G(\lambda_{4},0)\left[\partial_\lambda\partial_g\varphi(\lambda_{4},0)x_a,x_a\right]\\
&+\frac{1}{2}Q\partial_\lambda\partial_f G(\lambda_{4},0)\bar{x}.
\end{split}
\end{eqnarray*}

The direct computations of \eqref{Gc} yield that
\begin{eqnarray*}
\partial_\lambda\partial_{ff}G(\lambda_{4},0)[x_a,x_a]=\partial_\lambda\partial_t^2G(\lambda_4,t x_a)|_{t=0}=
4a
\left(
\begin{array}{cc}
b^2\\
-1
\end{array}
\right)
e_{2},
\end{eqnarray*}
or we may apply Lemma \ref{lem1} to get the same result.
%So, if $p\ne 2$,
%$\frac{1}{2}Q\partial_\lambda\partial_{ff}G(\lambda_{4},0)[x_a,x_a]=0$, and if $p=2$
Then by the definition of $Q$ in \ref{S:Pre}, we get the first term in the formation of $\partial_t\partial_\lambda F_2(\lambda_{4},0;a)$, which is,
\begin{eqnarray*}
\frac{1}{2}Q\partial_\lambda\partial_{ff}G(\lambda_{4},0)
=Q\left[2
a
\left(
\begin{array}{cc}
b^2\\
-1
\end{array}
\right)e_2\right]
=\sqrt{2}a(b^2+1)y_1.
\end{eqnarray*}

For the second term, notice that
\begin{eqnarray*}
\partial_{ff}G(\lambda_{4},0)\left[\partial_\lambda\partial_g\varphi(\lambda_{4},0)x_a,x_a\right]=
\frac{\partial^2}{\partial t\partial s}G(\lambda_{4},t x_a+s\partial_\lambda\partial_g\varphi(\lambda_{4},0)x_a)_{|t,s=0}
\end{eqnarray*}
and
$\tilde{x}=
\partial_\lambda\partial_g \varphi(\lambda_{4},0)x_a
=2b^{-1}
\left(
\begin{array}{cc}
1\\
0
\end{array}
\right)\overline{w}
-
4ab^{-3}
\left(
\begin{array}{cc}
1\\
0
\end{array}
\right)\overline{w}^{3}
$
from \eqref{E:philga2} in the proof of (1).
Then we still use the notation of $\phi_j$ without changes for convenience and denote
\begin{eqnarray*}
\phi_j(t,s,w)=b_jw+(t\alpha_j+s\tilde{\alpha}_j)\overline{w}+(t\theta_j+s\tilde{\theta}_j)\overline{w}^{3},
\end{eqnarray*}
where $\tilde{\alpha}_1=2b^{-1}$, $\tilde{\alpha}_2=0$, $\tilde{\theta}_1=-4ab^{-3}$, $\tilde{\theta}_2=0$ %comes from \eqref{E:philga},
and $\alpha_j$, $b_j$, $\theta_j$ are in \eqref{notation1}.
Now, we write
\begin{eqnarray*}
G_j(\lambda_{4},t x_a+s\partial_\lambda\partial_g\varphi(\lambda_{4},0)x_a)
\!=\!\mathrm{Im}\!\left\{\! \left[(1-\lambda_{4})\overline{\phi_j(t,s,w)}+I(\phi_j(t,s,w))\right]w\phi_j'(t,s,w)\!\right\}\!.
\end{eqnarray*}
The direct computations yield that
\begin{align}
&\frac{\partial^2}{\partial t\partial s}
G_j(\lambda_{4},tx_a+s\partial_\lambda\partial_g\varphi(\lambda_{4},0)x_a)_{|t,s=0}\nonumber\\ \nonumber
=&\mathrm{Im}\left\{b_j w \frac{\partial^2}{\partial t\partial s}I(\phi_j)_{|t,s=0} \right\}
-\mathrm{Im}\left\{(\tilde{\alpha}_j \overline{w}+3\tilde{\theta}_j\overline{w}^{3}) \frac{\partial}{\partial t}I(\phi_j)_{|t,s=0} \right\}\\
&-\mathrm{Im}\left\{(\alpha_j \overline{w}+3\theta_j\overline{w}^{3}) \frac{\partial}{\partial s}I(\phi_j)_{|t,s=0} \right\}+2(\lambda_{4}-1)(\tilde{\alpha}_j\theta_j+\alpha_j\tilde{\theta}_j)e_{2}.\label{dGst}
\end{align}

From the definition of $I(z)$ in \eqref{E:Cauchy-m}, we have
\begin{eqnarray*}
I_i(\phi_j(t,s,w))=\int \frac{\overline{A}+t(B+C)+s(H+J)}{A+t(\overline{B}+\overline{C})+s(\overline{H}+\overline{J})}
[b_i-t(D+E)-s(K+L)]\dif \tau,
\end{eqnarray*}
where
\begin{eqnarray*}
&&A=b_jw-b_i\tau,\, B=\alpha_j w-\alpha_i \tau,\, C=\theta_jw^{3}-\theta_i\tau^{3},\, D=\alpha_i\overline{\tau}^2,\,E=3\theta_i \overline{\tau}^{4},\\
&& H=\tilde{\alpha}_j w-\tilde{\alpha}_i \tau,\,   J=\tilde{\theta}_jw^{3}-\tilde{\theta}_i\tau^{3},\,
K=\tilde{\alpha}_i\overline{\tau}^2,\, L=3\tilde{\theta}_i \overline{\tau}^{4}.
\end{eqnarray*}
It is easy to check that
\begin{eqnarray*}
\frac{\partial}{\partial t}I_i(\phi_j(t,s,w))_{|s,t=0}=
\int \frac{b_i(B+C)-\overline{A}(D+E)}{A}\dif \tau
-\int \frac{b_i\overline{A}(\overline{B}+\overline{C})}{A^2}\dif \tau,
\end{eqnarray*}
\begin{eqnarray*}
\frac{\partial}{\partial s}I_i(\phi_j(t,s,w))_{|s,t=0}=
\int \frac{b_i(H+J)-\overline{A}(K+L)}{A}\dif \tau
-\int \frac{b_i\overline{A}(\overline{H}+\overline{J})}{A^2}\dif \tau,
\end{eqnarray*}
and
\begin{small}
\begin{align*}
\frac{\partial^2}{\partial t\partial s}\!I_i(\!\phi_j(t,s,w)\!)_{|s,t=0}
\!=\!&
-\!\!\int\!\! \frac{(B+C)(K+L)+(H+J)(D+E)}{A}\dif \tau
\!+\!2\!\!\int\!\!\frac{b_i\overline{A}(\overline{B}+\overline{C})(\overline{H}+\overline{J})}{A^3}\dif \tau\\
&-\!\!\int\!\! \frac{b_i(B+C)-\overline{A}(D+E)}{A^2}(\overline{H}+\overline{J})\dif \tau\\
&-\!\!\int\!\! \frac{b_i(H+J)-\overline{A}(K+L)}{A^2}(\overline{B}+\overline{C})\dif \tau.
\end{align*}
\end{small}
Recall that in the proof of Proposition \ref{acformal},
\begin{eqnarray*}
\frac{\partial}{\partial t}I_i(\phi_j(t,s,w))_{|s,t=0}
=\mu_{ij}w+\eta_{ij}w^{3}+\gamma_{ij}\overline{w}^3+\kappa_{ij}\overline{w}^{5}.
\end{eqnarray*}
Since $\tilde{x}$ has the same form of $x_a$, we have
\begin{eqnarray*}
\frac{\partial}{\partial s}I_i(\phi_j(t,s,w))_{|s,t=0}
=\tilde{\mu}_{ij}w+\tilde{\eta}_{ij}w^{3}+\tilde{\gamma}_{ij}\overline{w}^3+\tilde{\kappa}_{ij}\overline{w}^{5}.
\end{eqnarray*}
Using the change of variables and expand the equation of $\frac{\partial^2}{\partial t\partial s}I_i(\phi_j(t,s,w))_{|s,t=0}$ in the above, we may write
\begin{eqnarray*}
\frac{\partial^2}{\partial t\partial s}I_i(\phi_j(t,s,w))_{|s,t=0}
=\hat{\eta}_{ij}w+\hat{\mu}_{ij}\overline{w}+\hat{\kappa}_{ij}\overline{w}^{3}+\hat{\gamma}_{ij}\overline{w}^5
+\hat{\nu}_{ij}\overline{w}^{7}+\hat{\rho}_{ij}\overline{w}^{9}.
\end{eqnarray*}
Of course, $\tilde{\mu}_{ij},\tilde{\eta}_{ij},\tilde{\gamma}_{ij},\tilde{\kappa}_{ij},
\hat{\eta}_{ij}, \hat{\mu}_{ij}, \hat{\kappa}_{ij},\hat{\gamma}_{ij},\hat{\nu}_{ij},\hat{\rho}_{ij}$ are real constants, and can be written as follows
\begin{eqnarray}\label{1}
&&\tilde{\mu}_{ij}w=\int \frac{b_iH}{A}\dif \tau,\quad \tilde{\eta}_{ij}w^{3}=\int \frac{b_iJ}{A}\dif \tau,\quad
\tilde{\gamma}_{ij}\overline{w}^3=-\int \frac{\overline{A}K}{A}\dif \tau-\int \frac{b_i\overline{A}\overline{H}}{A^2}\dif \tau,\nonumber\\
&&\hat{\eta}_{ij}w=-\int \frac{CK+JD}{A}\dif \tau-\int \frac{b_iC\overline{H}}{A^2}\dif \tau
-\int\frac{b_i\overline{B}J}{A^2}\dif \tau,\nonumber\\
&&\hat{\kappa}_{ij}\overline{w}^{3}=-\int \frac{BL+HE}{A}\dif \tau-\int \frac{b_iB\overline{J}}{A^2}\dif \tau
-\int \frac{b_iH\overline{C}}{A^2}\dif \tau,\nonumber\\
&&\hat{\gamma}_{ij}\overline{w}^5=\int \frac{\overline{A}D\overline{H}}{A^2}\dif \tau
+\int \frac{\overline{A}K\overline{B}}{A^2}\dif \tau+\int \frac{2b_i\overline{A}\overline{B}\overline{H}}{A^3}\dif \tau.
\end{eqnarray}
%Inserting the above results into \eqref{dGst}, we obtain
%\begin{eqnarray*}
%\frac{\partial^2}{\partial t\partial s}
%G_j(\lambda_{4},tx_a+s\partial_\lambda\partial_g\varphi(\lambda_{4},0)x_a)_{|t,s=0}
%=C_{j,1}e_{4-2}+C_{j,2}e_4+C_{j,3}e_{4+2}+C_{j,4}e_{4p}
%\end{eqnarray*}
%for some real constants $C_{j,1},C_{j,2},C_{j,3},C_{j,4}$. So if $p\ne2$, it is easy to see that
%\begin{eqnarray*}
%Q\partial_{ff}G(\lambda_{4},0)\left[\partial_\lambda\partial_g\varphi(\lambda_{4},0)x_a,x_a\right]
%=Q\frac{\partial^2}{\partial t\partial s}
%G(\lambda_{4},t x_a+s\partial_\lambda\partial_g\varphi(\lambda_{4},0)x_a)_{|t,s=0}=0.
%\end{eqnarray*}
Because the last step to compute the second term is to apply the projection operator onto equation \eqref{dGst}, we only write the constants that would work in this step, that is, we only concern the constants causing $e_2$ and $e_4$ terms (when $p\ne2$, $e_2$ and $e_{2p}$ terms). So in \eqref{1}, we do not give the forms of all constants.
Applying the methods in the proof of Proposition \ref{JF2c}, directly using the holomorphism of integrals and the residue theorem at $\infty$ or using the useful identities \eqref{E:0}, \eqref{E:01}, \eqref{E:1}, \eqref{E:02}, \eqref{E:2}, \eqref{E:03}, \eqref{E:3}, we obtain the following results
\begin{eqnarray*}
&&\tilde{\mu}_{11}=\tilde{\mu}_{22}=\tilde{\mu}_{21}=0,\,\tilde{\mu}_{12}=2,\\
&&\tilde{\eta}_{11}=\tilde{\eta}_{22}=\tilde{\eta}_{21}=0,\, \tilde{\eta}_{12}=-4a,\\
&&\tilde{\gamma}_{11}=\tilde{\gamma}_{22}=\tilde{\gamma}_{12}=0,\,\tilde{\gamma}_{21}=2b,\\
&&\hat{\eta}_{11}=\hat{\eta}_{22}=\hat{\eta}_{21}=0,\,\hat{\eta}_{12}=-6ab,\\
&&\hat{\kappa}_{22}=\hat{\kappa}_{21}=0,\,\hat{\kappa}_{11}=2a(1-2b^{-2}),\,\hat{\kappa}_{12}=-2ab^{-1},\\
&&\hat{\gamma}_{11}=\hat{\gamma}_{22}=\hat{\gamma}_{12}=0,\,\hat{\gamma}_{21}=2b^2,
\end{eqnarray*}
where we omit the details of computation process because it is cumbersome and all the methods are the same as we used in the proof of Proposition \ref{JF2c}. Then by the definition of $I(z)$ defined in \eqref{E:Cauchy-m} and \eqref{dI1c}, using the value of these real constants,  we have
\begin{align*}
\left(
\begin{array}{cc}
\frac{\partial}{\partial t}I(\phi_1(t,s,w))_{|s,t=0}\\
\frac{\partial}{\partial t}I(\phi_2(t,s,w))_{|s,t=0}
\end{array}
\right)&
=
\left(
\begin{array}{cc}
0\\
a(1+b^4)
\end{array}
\right)w^3+
\left(
\begin{array}{cc}
0\\
b^2-1
\end{array}
\right)w-
\left(
\begin{array}{cc}
a(b^3+b^5)\\
0
\end{array}
\right)\overline{w}^5,
\\
\left(
\begin{array}{cc}
\frac{\partial}{\partial s}I(\phi_1(t,s,w))_{|s,t=0}\\
\frac{\partial}{\partial s}I(\phi_2(t,s,w))_{|s,t=0}
\end{array}
\right)&
=
\left(
\begin{array}{cc}
0\\
2
\end{array}
\right)w
-
\left(
\begin{array}{cc}
0\\
4a
\end{array}
\right)w^3-
\left(
\begin{array}{cc}
2b\\
0
\end{array}
\right)\overline{w}^3+
\left(
\begin{array}{cc}
\tilde{\kappa}_{11}-\tilde{\kappa}_{21}\\
\tilde{\kappa}_{12}-\tilde{\kappa}_{22}
\end{array}
\right)\overline{w}^5,
\\
\left(
\begin{array}{cc}
\frac{\partial^2}{\partial t\partial s}I(\phi_1(t,s,w))_{|s,t=0}\\
\frac{\partial^2}{\partial t\partial s}I(\phi_2(t,s,w))_{|s,t=0}
\end{array}
\right)
&=
-\left(
\begin{array}{cc}
0\\
6ab
\end{array}
\right)w+
\left(
\begin{array}{cc}
2a(1-2b^{-2})\\
-2ab^{-1}
\end{array}
\right)\overline{w}^3-
\left(
\begin{array}{cc}
2b^2\\
0
\end{array}
\right)\overline{w}^5
\\&+
\left(
\begin{array}{cc}
\hat{\mu}_{11}-\hat{\mu}_{21}\\
\hat{\mu}_{12}-\hat{\mu}_{22}
\end{array}
\right)\overline{w}+
\left(
\begin{array}{cc}
\hat{\nu}_{11}-\hat{\nu}_{21}\\
\hat{\nu}_{12}-\hat{\nu}_{22}
\end{array}
\right)\overline{w}^7+
\left(
\begin{array}{cc}
\hat{\rho}_{11}-\hat{\rho}_{21}\\
\hat{\rho}_{12}-\hat{\rho}_{22}
\end{array}
\right)\overline{w}^9.
\end{align*}
Plugging the above results into \eqref{dGst}, and we have
\begin{eqnarray*}
Q\frac{\partial^2}{\partial t\partial s}
G(\lambda_{4},tx_a+s\tilde{x})_{|t,s=0}
=
Q
\left[
\left(
\begin{array}{cc}
2a(b^2-2)\\
6ab^2
\end{array}
\right)e_2
\right]
=-2\sqrt{2}a(1+b^{2})y_1.
\end{eqnarray*}

To compute the third term, we need to compute $\bar{x}$ first, which is given by
\begin{eqnarray*}
\bar{x}=-\left[\partial_{f} G(\lambda_{4},0)\right]^{-1}\frac{\dif^2}{\dif t^2}({\rm {Id}-Q}) G(\lambda_{4}, tx_a)_{|t=0}.
\end{eqnarray*}
Using the expression of $\frac{\dif^2}{\dif t^2}G(\lambda_{4}, tx_a)_{|t=0}$ in the proof of Proposition \ref{JF2c} and \eqref{iG}, we have
\begin{align*}
\bar{x}&=
2a(1+b^2)
\left(
\begin{array}{cc}
1\\
0
\end{array}
\right)\overline{w}
-M_{8}^{-1}(\lambda_2)
\left(
\begin{array}{cc}
8a^2(b^2+b^{4})^2\\
0
\end{array}
\right)\overline{w}^{7}\\
&=
2a(1+b^2)
\left(
\begin{array}{cc}
1\\
0
\end{array}
\right)\overline{w}
-\frac{8a^2(b^2+b^4)^2}{b^{17}-b(4b^2-3)^2}
\left(
\begin{array}{cc}
4b^3-3b\\
b^8
\end{array}
\right)
\overline{w}^{7}\\
&=
2a(1+b^2)
\left(
\begin{array}{cc}
1\\
0
\end{array}
\right)\overline{w}
-\frac{a^2(b^2+b^4)^2}{(b^2-1)^2[1+4(b^2-1)+2(b^2-1)^2]}
\left(
\begin{array}{cc}
4b^2-3\\
b^7
\end{array}
\right)
\overline{w}^{7},
\end{align*}
where the last equation comes by using the identity $(b^4)^4=(1-2b^2)^4$, that is $b^4=1-2b^2$. For the convenience of future use, we denote
\begin{align}\label{notation2}
\nu_1&=2a(1+b^2), \quad\quad\quad\nu_2=0,\\
\eta_1&= -\frac{a^2(b^2+b^4)^2(4b^2-3)}{(b^2-1)^2[1+4(b^2-1)+2(b^2-1)^2]}
\text{ and }\eta_2=-\frac{a^2(b^2+b^4)^2b^7}{(b^2-1)^2[1+4(b^2-1)+2(b^2-1)^2]},\nonumber
\end{align}
and hence $\bar{x}=
\left(
\begin{array}{cc}
\nu_1\\
\nu_2
\end{array}
\right)\overline{w}+
\left(
\begin{array}{cc}
\eta_1\\
\eta_2
\end{array}
\right)\overline{w}^7.
$
According to the representation of $\partial_f G$ in \eqref{dG1} and the definition of the projection operator $Q$, we get
\begin{eqnarray*}
\frac{1}{2}Q\partial_\lambda\partial_fG(\lambda_{4},0)\overline{x}
=
Q
\left[
a(1+b^2)
\left(
\begin{array}{cc}
2&0\\
0&2b
\end{array}
\right)
\left(
\begin{array}{cc}
1\\
0
\end{array}
\right)e_2
\right]
=\sqrt{2}a(1+b^{2})y_1.
\end{eqnarray*}
Combining this with other two terms, we have
\begin{eqnarray*}
\partial_t\partial_\lambda F_2(\lambda_{4},0;a)
=2\sqrt{2}a(1+b^2) y_1- 2\sqrt{2}a(1+b^{2})y_1=0.
\end{eqnarray*}
The same approach can be used to obtain the result for $p\ge3$, which requires complex computations that we omit here.

(3) For $p=2$, from Proposition \ref{acformal}, we have
\begin{eqnarray*}
\partial_{tt}F_2(\lambda_{4},0;a)=\frac{1}{3}\frac{\dif^3}{\dif t^3}QG(\lambda_{4},tx_a)_{|t=0}+Q\partial_{ff}G(\lambda_{4},0)[x_a,\bar{x}],
\end{eqnarray*}
where
\begin{eqnarray*}
Q\partial_{ff}G(\lambda_{4},0)[x_a,\bar{x}]=Q\partial_t\partial_s G(\lambda_{4},t x_a+s\bar{x})_{|t=0,s=0}.
\end{eqnarray*}
Since the projection operator $Q$ project space $Y$ into space $\langle y_1\rangle\oplus \langle y_2\rangle$, in the following computations of $\frac{\dif^3}{\dif t^3}G(\lambda_{4},tx_a)_{|t=0}$ and $\partial_t\partial_s G(\lambda_{4},t x_a+s\bar{x})_{|t=0,s=0}$, we only need to focus on the $e_2$ and $e_{4}$ terms.

With $G$ defined in \eqref{Gc}, it is easy to check that
%\begin{small}
\begin{eqnarray*}
\frac{\dif^3}{\dif t^3}G_j(\lambda_{4},tx_a)_{|t=0}
\!=\!\mathrm{Im}\!\left\{\!b_j w\frac{\dif^3}{\dif t^3}I(\phi_j(t,w))_{|t=0} \!\right\}\!
-\!3\mathrm{Im}\!\left\{\!(\alpha_j \overline{w}+3\theta_j\overline{w}^{3})\frac{\dif^2}{\dif t^2}I(\phi_j(t,w))_{|t=0} \!\right\}\!,
\end{eqnarray*}
%\end{small}
Inserting \eqref{notation1}, \eqref{dI2c1} and \eqref{dI2c2} into the second term of the above equation, we know the $e_2$ term and $e_{4}$ term in the second term of the formula of $\frac{\dif^3}{\dif t^3}G(\lambda_{4},tx_a)|_{t=0}$ are,
\begin{eqnarray}\label{secondterm}
6\left(
\begin{array}{cc}
(b-b^3)(1+3a^2)\\
b^{-1}-b+3(b^{-1}+b^{3})a^2
\end{array}
\right)e_2\,\,%, \nonumber\\
\text{ and }\,\,
6\left(
\begin{array}{cc}
3(b-b^3)a+9(b-b^3)a^3\\
4(b-b^{-1})a-9(b^{-1}+b^{3})a^3
\end{array}
\right)e_{4},
\end{eqnarray}
respectively. Now we compute the first term of the formula of $\frac{\dif^3}{\dif t^3}G(\lambda_{4},tx_a)_{|t=0}$. Recall that
\begin{eqnarray*}
I_i(\phi_j(t,w))=\int\frac{\overline{A}+t(B+C)}{A+t(\overline{B}+\overline{C})}[b_i-t(D+E)]\dif \tau,
\end{eqnarray*}
with
\begin{eqnarray*}
A=b_jw-b_i\tau,\, B=\alpha_j w-\alpha_i \tau,\, C=\theta_jw^{3}-\theta_i\tau^{3},\, D=\alpha_i\overline{\tau}^2,\,E=3\theta_i \overline{\tau}^{4}.
\end{eqnarray*}
By directly computation and the change of variables, we have
\begin{eqnarray*}
\begin{split}
\frac{1}{6}\frac{\dif^3}{\dif t^3}I_i(\phi_j(t,w))_{|t=0}
=&\!\!\int\!\!\frac{(B\!+\!C)(D\!+\!E)(\overline{B}\!+\!\overline{C})}{A^2}\dif \tau
\!+\!\!\int\!\!\frac{b_i(B\!+\!C)\!-\!\overline{A}(D\!+\!E)}{A^3}(\overline{B}+\overline{C})^2\dif \tau\\
&-\!\!\int\!\!\frac{b_i\overline{A}(\overline{B}+\overline{C})^3}{A^4}\dif \tau\\
=& \breve{\mu}_{ij}\overline{w}^3\! +\!\breve{\eta}_{ij}\overline{w}^{5}\!+\!\breve{\kappa}_{ij}\overline{w}
\!+\!\breve{\gamma}_{ij}\overline{w}^{7}\!+\!\breve{\rho}_{ij}\overline{w}^{9}\!+\!\breve{\sigma}_{ij}\overline{w}^{11}
\!+\!\breve{\xi}_{ij}\overline{w}^{13}
\end{split}
\end{eqnarray*}
where $\breve{\mu}_{ij}$, $\breve{\eta}_{ij}$, $\breve{\kappa}_{ij}$, $\breve{\gamma}_{ij}$, $\breve{\rho}_{ij}$, $\breve{\sigma}_{ij}$, $\breve{\xi}_{ij}$ are real constants. Since we only concern the $e_2$ and $e_4$ terms in $\frac{\dif^3}{\dif t^3}G_j(\lambda_{4},tx_a)_{|t=0}$, then we just write the formulas of parts of them,
\begin{eqnarray*}
\begin{split}
& \breve{\mu}_{ij}\overline{w}^3
=\int\frac{BD\overline{B}}{A^2}\dif \tau
+\int\frac{CD\overline{C}}{A^2}\dif \tau
+\int\frac{CE\overline{B}}{A^2}\dif \tau
+\int\frac{b_iB\overline{B}^2}{A^3}\dif \tau
+\int\frac{2b_iC\overline{B}\overline{C}}{A^3}\dif \tau,\\
&\breve{\eta}_{ij}\overline{w}^{5}
=\int\frac{BD\overline{C}}{A^2}\dif \tau
+\int\frac{BE\overline{B}}{A^2}\dif \tau
+\int\frac{CE\overline{C}}{A^2}\dif \tau
+\int\frac{2b_iB\overline{B}\overline{C}}{A^3}\dif \tau
+\int\frac{b_iC\overline{C}^2}{A^3}\dif \tau.
%\\
%&\breve{\kappa}_{ij}w^{4-5}
%=\int\frac{CD\overline{B}}{A^2}\dif \tau
%+\int\frac{b_iC\overline{B}^2}{A^3}\dif \tau,
%\,\,\breve{\rho}_{ij}\overline{w}^{7}
%=-\int\frac{\overline{A}(\overline{B})^2D}{A^3}\dif \tau
%-\int\frac{b_i\overline{A}(\overline{B})^3}{A^4}\dif \tau.
\end{split}
\end{eqnarray*}
By a series computations with the same methods in the proof of Proposition \ref{JF2c}, using the residue theorem at $\infty$ or some useful identities defined in \eqref{E:0} and computed in \eqref{E:01}-\eqref{E:3}, we get
\begin{eqnarray*}
&& \breve{\mu}_{11}=3a^2b^3,\,\breve{\mu}_{21}=3a^2b,\,\breve{\mu}_{12}=-3a^2b^{2},\,\breve{\mu}_{22}=3a^2b^{-2},\\
&&\breve{\eta}_{11}=3a^3b^3,\,\breve{\eta}_{21}=-a(b+b^{3})-3a^3 (b+2b^{3}),\,
\breve{\eta}_{12}=3a^3b^{2},\,\breve{\eta}_{22}=-3a^3b^{-2},
%\\
%&&\breve{\kappa}_{11}=\breve{\kappa}_{12}=\breve{\kappa}_{21}=\breve{\kappa}_{22}=0,\,\,
%\breve{\rho}_{11}=\breve{\rho}_{12}=\breve{\rho}_{21}=\breve{\rho}_{22}=0,
\end{eqnarray*}
From \eqref{E:Cauchy-m} and the representations of $\frac{\dif^3}{\dif t^3}I_i(\phi_1(t,w))_{|t=0}$ with these real constants, we obtain
\begin{align*}
&\frac{1}{6}
\left(
\begin{array}{cc}
\frac{\dif^3}{\dif t^3}I(\phi_1(t,w))_{|t=0}\\
\frac{\dif^3}{\dif t^3}I(\phi_2(t,w))_{|t=0}
\end{array}
\right)\\
=&-3a^2
\left(
\begin{array}{cc}
b-b^3\\
b^2+b^{-2}
\end{array}
\right)\overline{w}^3+
\left(
\begin{array}{cc}
a(b+b^{3})+3a^3 (b+3b^{3})\\
3a^3(b^2+b^{-2})
\end{array}
\right)\overline{w}^5+
\left(
\begin{array}{cc}
\breve{\kappa}_{11}-\breve{\kappa}_{21}\\
\breve{\kappa}_{12}-\breve{\kappa}_{22}
\end{array}
\right)\overline{w}
\\&+
\left(
\begin{array}{cc}
\breve{\gamma}_{11}-\breve{\gamma}_{21}\\
\breve{\gamma}_{12}-\breve{\gamma}_{22}
\end{array}
\right)\overline{w}^7+
\left(
\begin{array}{cc}
\breve{\rho}_{11}-\breve{\rho}_{21}\\
\breve{\rho}_{12}-\breve{\rho}_{22}
\end{array}
\right)\overline{w}^9+
\left(
\begin{array}{cc}
\breve{\sigma}_{11}-\breve{\sigma}_{21}\\
\breve{\sigma}_{12}-\breve{\sigma}_{22}
\end{array}
\right)\overline{w}^{11}+
\left(
\begin{array}{cc}
\breve{\xi}_{11}-\breve{\xi}_{21}\\
\breve{\xi}_{12}-\breve{\xi}_{22}
\end{array}
\right)\overline{w}^{13}.
\end{align*}
Inserting the above results into the formula of $\frac{\dif^3}{\dif t^3}G(\lambda_{4},tx_a)_{|t=0}$, we know the $e_2$ and $e_{4}$ terms in the first term of the formula of $\frac{\dif^3}{\dif t^3}G(\lambda_{4},tx_a)_{|t=0}$ are
\begin{eqnarray*}
6\left(
\begin{array}{cc}
3a^2(b^3-b)\\
-3a^2(b^{3}+b^{-1})
\end{array}
\right)e_2\quad
\text{ and }\quad
6\left(
\begin{array}{cc}
a(b+b^{3})+3a^3(b+3b^3)\\
3a^3(b^{3}+b^{-1})
\end{array}
\right)e_{4}.
\end{eqnarray*}
Combining the above results with \eqref{secondterm}, we have,
\begin{align}\label{dG3c}
\frac{1}{3}Q\frac{\dif^3}{\dif t^3}G(\lambda_{4},tx_a )_{|t=0}%\nonumber\\
&=
2Q\left[
\left(
\begin{array}{cc}
b-b^3\\
b^{-1}-b
\end{array}
\right)e_2
+\left(
\begin{array}{cc}
4a(b-b^3)+12a^3\\
4a(b-b^{-1})-6a^3(b^{-1}+b^{3})
\end{array}
\right)e_{4}
\right]\nonumber\\
&= -\sqrt{2}b^{-1}(b^2-1)^2\left[y_1-2a(2+3a^2)y_2\right].
\end{align}

To compute $Q\partial_t\partial_s G(\lambda_{4},t x_a+s\bar{x})_{|t=0,s=0}$, the second term of the formula of $\partial_{tt}F_2(\lambda_{4},0;a)$, we denote
\begin{eqnarray*}
\phi_j(t,s,w)=b_jw+t\alpha_j\overline{w}+t\theta_j\overline{w}^{3}+s\nu_j\overline{w}+s\eta_j\overline{w}^{7},
\end{eqnarray*}
where $b_j, \alpha_j,\theta_j$ are in \eqref{notation1} and $\nu_j,\eta_j$ are in \eqref{notation2}. Now we can write
\begin{eqnarray*}
G_j(\lambda_{4},t x_a+s\overline{x})
=\mathrm{Im}\left\{ \left[(1-\lambda_{4})\overline{\phi_j(t,s,w)}+I(\phi_j(t,s,w))\right]w\phi_j'(t,s,w)\right\}.
\end{eqnarray*} By direct computations, we obtain
\begin{align}\label{dGst2}
&\frac{\partial^2}{\partial t\partial s}
G_j(\lambda_{4},tx_a+s\overline{x})_{|t,s=0}\nonumber\\
=&\mathrm{Im}\left\{b_j w \frac{\partial^2}{\partial t\partial s}I(\phi_j)_{|t,s=0} \right\}
-\mathrm{Im}\left\{(\nu_j \overline{w}+7\eta_j\overline{w}^{7}) \frac{\partial}{\partial t}I(\phi_j)_{|t,s=0} \right\}
+6(\lambda_{4}-1)\alpha_j\eta_je_{6}
\nonumber\\
&-\mathrm{Im}\left\{(\alpha_j \overline{w}+3\theta_j\overline{w}^{3}) \frac{\partial}{\partial s}I(\phi_j)_{|t,s=0} \right\}
+2(\lambda_{4}-1)\theta_j\nu_je_{2}+4(\lambda_{4}-1)\theta_j\eta_je_{4}.
\end{align}
Similarly as before, we will focus on the terms of $e_2$ and $e_{4}$. By the definition of $I_i(z)$, we rewritten that
\begin{eqnarray*}
I_i(\phi_j(t,s,w))=\int\frac{\overline{A}+t(B+C)+s(M+N)}{A+t(\overline{B}+\overline{C})+s(\overline{M}+\overline{N})}
[b_i-t(D+E)-s(O+P)]\dif \tau,
\end{eqnarray*}
with
\begin{align*}
&A=b_jw-b_i\tau,\, B=\alpha_j w-\alpha_i \tau,\, C=\theta_jw^{3}-\theta_i\tau^{3},\, D=\alpha_i\overline{\tau}^2,\,E=3\theta_i \overline{\tau}^{4},\\
&M=\nu_j w-\nu_i \tau,\, N=\eta_j w^{7}-\eta_i \tau^{7},\, O=\nu_i \overline{\tau}^{2},\,
P=7\eta_i\overline{\tau}^{8}.
\end{align*}
Then it is easy to check that
\begin{eqnarray*}
\frac{\partial}{\partial t}I_i(\phi_j(t,s,w))_{|s,t=0}=
\int \frac{b_i(B+C)-\overline{A}(D+E)}{A}\dif \tau
-\int \frac{b_i\overline{A}(\overline{B}+\overline{C})}{A^2}\dif \tau,
\end{eqnarray*}
\begin{eqnarray*}
\frac{\partial}{\partial s}I_i(\phi_j(t,s,w))_{|s,t=0}=
\int \frac{b_i(M+N)-\overline{A}(O+P)}{A}\dif \tau
-\int \frac{b_i\overline{A}(\overline{M}+\overline{N})}{A^2}\dif \tau,
\end{eqnarray*}
and
\begin{footnotesize}
\begin{eqnarray*}
\begin{split}
\frac{\partial^2}{\partial t\partial s}I_i(\phi_j(t,s,w))_{|s,t=0}
=&
\!-\!\!\int\!\! \frac{(B+C)(O+P)+(M+N)(D+E)}{A}\dif \tau
\!+\!2\!\!\int\!\! \frac{b_i\overline{A}(\overline{B}+\overline{C})(\overline{M}+\overline{N})}{A^3}\dif \tau
\\
&-\!\!\int\!\! \frac{b_i(M+N)-\overline{A}(O+P)}{A^2}(\overline{B}+\overline{C})\dif \tau
-\!\!\int\!\! \frac{b_i(B+C)-\overline{A}(D+E)}{A^2}(\overline{M}+\overline{N})\dif \tau
.
\end{split}
\end{eqnarray*}
\end{footnotesize}
Observe that the value of $
\frac{\partial}{\partial t}I(\phi_j(t,s,w))_{|s,t=0}$ is the same as the one in \eqref{dI1c}, then we consider $\frac{\partial}{\partial s}I_i(\phi_j(t,s,w))_{|s,t=0}$ and $\frac{\partial^2}{\partial t\partial s}I_i(\phi_j(t,s,w))_{|s,t=0}$. By the change of variable, we may write
%\begin{eqnarray*}
%\frac{\partial}{\partial t}I_i(\phi_j(t,s,w))_{|s,t=0}
%=\mu_{ij}w+\eta_{ij}w^{4-1}+\gamma_{ij}\overline{w}^3+\kappa_{ij}\overline{w}^{4+1}.
%\end{eqnarray*}
\begin{eqnarray*}
\frac{\partial}{\partial s}I_i(\phi_j(t,s,w))_{|s,t=0}
=\dot{\mu}_{ij}w+\dot{\eta}_{ij}w^{7}+\dot{\gamma}_{ij}\overline{w}^{3}+\dot{\kappa}_{ij}\overline{w}^{9},
\end{eqnarray*}
and
\begin{align*}
\frac{\partial^2}{\partial t\partial s}I_i(\phi_j(t,s,w))_{|s,t=0}
=&\acute{\mu}_{ij}w+\acute{\eta}_{ij}w^{3}+\acute{\kappa}_{ij}\overline{w}^{3}
+\acute{\rho}_{ij}\overline{w}^{5}\\
&+\acute{\xi}_{ij}w^{5}+\acute{\gamma}_{ij}\overline{w}+ \acute{\chi}_{ij}\overline{w}^{7}+ \acute{\tau}_{ij}\overline{w}^{11}+ \acute{\iota}_{ij}\overline{w}^{13},
\end{align*}
where $\dot{\mu}_{ij},\dot{\eta}_{ij},\dot{\gamma}_{ij},\dot{\kappa}_{ij},
\acute{\mu}_{ij},\acute{\eta}_{ij},\acute{\kappa}_{ij},\acute{\rho}_{ij},\acute{\xi}_{ij},\acute{\gamma}_{ij},
\acute{\chi}_{ij},\acute{\tau}_{ij},\acute{\iota}_{ij}$ are real constants. From \eqref{dGst2}, we could just write the formula of the constants that would result $e_2$ and $e_4$ terms, which are
\begin{eqnarray*}
&&\dot{\mu}_{ij}w=\int \frac{b_iM}{A}\dif \tau,\quad \dot{\eta}_{ij}w^{7}=\int \frac{b_iN}{A}\dif \tau,\quad
\dot{\gamma}_{ij}\overline{w}^{3}=-\int \frac{\overline{A}O}{A}\dif \tau-\int \frac{b_i\overline{A}\overline{M}}{A^2}\dif \tau,\nonumber\\
&&\acute{\mu}_{ij}w=-\int \frac{CO}{A}\dif \tau-\int \frac{b_iC\overline{M}}{A^2}\dif \tau\dif \tau,
\quad\acute{\eta}_{ij}w^{3}=-\int \frac{NE}{A}\dif \tau-\int \frac{b_iN\overline{C}}{A^2},\nonumber\\
%&&\acute{\gamma}_{ij}w^{4-5}=-\int \frac{MD}{A}\dif \tau-\int \frac{b_iM\overline{B}}{A^2}\dif \tau,
&&\acute{\kappa}_{ij}\overline{w}^{3}=-\int \frac{ME}{A}\dif \tau-\int \frac{b_iM\overline{C}}{A^2}\dif \tau,\nonumber\\
%&&\acute{\nu}_{ij}\overline{w}^{4-3}=-\int \frac{BO}{A}\dif \tau-\int \frac{b_iB\overline{M}}{A^2}\dif \tau,\nonumber\\
&&\acute{\rho}_{ij}\overline{w}^{5}=-\int \frac{CP}{A}\dif \tau
+\int \frac{\overline{A}O\overline{B}-b_iC\overline{N}+\overline{A}D\overline{M}}{A^2}\dif \tau
+\int \frac{2b_i\overline{ABM}}{A^3}\dif \tau.
\end{eqnarray*}
By a series computations with the same methods in the proof of Proposition \ref{JF2c}, where we use the
residue theorem at $\infty$ or some useful equation, \ref{E:0}, \ref{E:01}-\ref{E:3}, then we get
\begin{align*}
&\dot{\mu}_{11}=\dot{\mu}_{22}=\dot{\mu}_{21}=0,\,\,\dot{\mu}_{12}=b\nu_1-\nu_2,\\
&\dot{\eta}_{11}=\dot{\eta}_{22}=\dot{\eta}_{21}=0,\, \,\dot{\eta}_{12}=b^{7}\eta_1-\eta_2,\\
&\dot{\gamma}_{11}=\dot{\gamma}_{22}=\dot{\gamma}_{12}=0,\,\,\dot{\gamma}_{21}=b^2\nu_1-b^{3}\nu_2,\\
&\acute{\mu}_{11}=\acute{\mu}_{22}=\acute{\mu}_{21}=0,\,\,\acute{\mu}_{12}=3a(\nu_2b^{3}-\nu_1b^2),\\
&\acute{\eta}_{11}=\acute{\eta}_{22}=\acute{\eta}_{21}=0,\,\,\acute{\eta}_{12}=-7a\eta_1(b^{4}+b^{6}),\\
%&&\text{For }p\!\ne\!2,\,\acute{\gamma}_{11}=\acute{\gamma}_{22}=\acute{\gamma}_{12}=\acute{\gamma}_{21}=0;
%\text{ For }p\!=\!2,\,
%\acute{\gamma}_{11}\!=b\nu_1,\,\acute{\gamma}_{22}\!=b^{-1}\nu_2,\,\acute{\gamma}_{12}\!=\nu_1,\,\acute{\gamma}_{21}\!=\nu_2,\\
&\acute{\kappa}_{11}=ab\nu_1,\,\,\acute{\kappa}_{22}=-ab^{-1}\nu_2,\,\,\acute{\kappa}_{12}=-a\nu_1,\,\,\acute{\kappa}_{21}=-ab^2\nu_2,\\
%&&\acute{\nu}_{11}=b\nu_1,\,\acute{\nu}_{22}=b^{-1}\nu_2,\,\acute{\nu}_{12}=b\nu_2,\,\acute{\nu}_{21}=b^{4-4}\nu_2,\\
&\acute{\rho}_{11}=3a\eta_1b,\,\,\acute{\rho}_{22}=-3a\eta_2b^{-1},\,\,
\acute{\rho}_{12}=3a\eta_2b^{3},\,\,
\acute{\rho}_{21}=-3a\eta_2b^{4}+\nu_1b^3-\nu_2b^{4}.
\end{align*}
With the expression of $\frac{\partial}{\partial s}I_i(\phi_1(t,s,w))_{|s,t=0}$, $\frac{\partial^2}{\partial t\partial s}I_i(\phi_j(t,s,w))_{|s,t=0}$, the definition of $I(z)$ in \eqref{E:Cauchy-m}  and the value of above constants, we get
\begin{align*}
\left(
\begin{array}{cc}
\frac{\partial}{\partial s}I(\phi_1(t,s,w))_{|s,t=0}\\
\frac{\partial}{\partial s}I(\phi_2(t,s,w))_{|s,t=0}
\end{array}
\right)
=&
\left(
\begin{array}{cc}
0\\
b\nu_1-\nu_2
\end{array}
\right)w+
\left(
\begin{array}{cc}
0\\
b^7\eta_1-\eta_2
\end{array}
\right)w^7+
\left(
\begin{array}{cc}
b^3\nu_2- b^2\nu_1\\
0
\end{array}
\right)\overline{w}^3
\\&+
\left(
\begin{array}{cc}
\dot{\kappa}_{11}-\dot{\kappa}_{21}\\
\dot{\kappa}_{12}-\dot{\kappa}_{22}
\end{array}
\right)\overline{w}^9
\end{align*}
and
\begin{align*}
&\left(
\begin{array}{cc}
\frac{\partial^2}{\partial t\partial s}I(\phi_1(t,s,w))_{|s,t=0}\\
\frac{\partial^2}{\partial t\partial s}I(\phi_2(t,s,w))_{|s,t=0}
\end{array}
\right)
\\=&
\left(
\begin{array}{cc}
0\\
3a(\nu_2b^3-\nu_1 b^2)
\end{array}
\right)w-
\left(
\begin{array}{cc}
0\\
7a\eta_1(b^4+b^6)
\end{array}
\right)w^3+
\left(
\begin{array}{cc}
a(b\nu_1+b^2\nu_2)\\
a(b^{-1}\nu_2-\nu_1)
\end{array}
\right)\overline{w}^3
\\&+
\left(
\begin{array}{cc}
3a(\eta_1 b+\eta_2b^4)-\nu_1b^3+\nu_2b^4\\
3a\eta_2(b^3+b^{-1})
\end{array}
\right)\overline{w}^5
+
\left(
\begin{array}{cc}
\acute{\xi}_{11}-\acute{\xi}_{21}\\
\acute{\xi}_{12}-\acute{\xi}_{22}
\end{array}
\right)w^5
+
\left(
\begin{array}{cc}
\acute{\gamma}_{11}-\acute{\gamma}_{21}\\
\acute{\gamma}_{12}-\acute{\gamma}_{22}
\end{array}
\right)\overline{w}
\\&+
\left(
\begin{array}{cc}
\acute{\chi}_{11}-\acute{\chi}_{21}\\
\acute{\chi}_{12}-\acute{\chi}_{22}
\end{array}
\right)\overline{w}^7
+
\left(
\begin{array}{cc}
\acute{\tau}_{11}-\acute{\tau}_{21}\\
\acute{\tau}_{12}-\acute{\tau}_{22}
\end{array}
\right)\overline{w}^{11}
+
\left(
\begin{array}{cc}
\acute{\iota}_{11}-\acute{\iota}_{21}\\
\acute{\iota}_{12}-\acute{\iota}_{22}
\end{array}
\right)\overline{w}^{13}.
\end{align*}
Inserting these above results and \eqref{dI1c} into \eqref{dGst2}, we have
\begin{align*}
&Q\frac{\partial^2}{\partial t\partial s}
G(\lambda_{4},tx_a+s\overline{x})_{|t,s=0}\nonumber\\
=&Q\left [ a
\left(
\begin{array}{cc}
b^3\nu_1+b^2\nu_2\\
(3b^3+2b)\nu_1-(b^2+2b^{4})\nu_2
\end{array}
\right)
e_{2}
+a
\left(
\begin{array}{cc}
b\eta_1+2b^3\eta_1+3b^{4}\eta_2\\
(7b^{5}+4b^{7})\eta_1+(1-2b^2-4b^{4})\eta_2
\end{array}
\right)e_{4}
\right],
\end{align*}
which implies that
\begin{align}\label{dffGa}
Q\partial_{ff}G(\lambda_{4},0)[x_a,\bar{x}]
=&a
\left\langle
\left(
\begin{array}{cc}
b^3\nu_1+b^2\nu_2\\
(3b^3+2b)\nu_1-(b^2+2b^{4})\nu_2
\end{array}
\right),
\left(
\begin{array}{cc}
\frac{1}{\sqrt{2}}\\
-\frac{1}{\sqrt{2}}
\end{array}
\right)
\right\rangle y_1\nonumber\\
&-a
\left\langle
\left(
\begin{array}{cc}
b\eta_1+2b^3\eta_1+3b^{4}\eta_2\\
(7b^{5}+4b^{7})\eta_1+(1-2b^2-4b^{4})\eta_2
\end{array}
\right),
\left(
\begin{array}{cc}
\frac{1}{\sqrt{2}}\\
\frac{1}{\sqrt{2}}
\end{array}
\right)
\right\rangle y_2
\nonumber\\
\triangleq &\sqrt{2}c_1 y_1+\sqrt{2}c_2y_2,
\end{align}
by the definition of $Q$ in Section \ref{S:Pre} and
\begin{align*}
\sqrt{2}c_1
=\frac{2a}{\sqrt{2}}[(b^2+b^{4})\nu_2-(b^3+b)\nu_1]
= \sqrt{2}a[(1-b^2)\nu_2-(b^3+b)\nu_1]
\end{align*}
and
\begin{align*}
\sqrt{2}c_2&=-\frac{a}{\sqrt{2}}[ \left(b + 2 b^3 + 4 b^{ 7}  + 7b^{5} \right)\eta_1+(1-2 b^2-b^{4})\eta_2]\\
&=
-\frac{a}{\sqrt{2}b}[ \left(b^2 + 2 b^4 + 4 (b^{4})^2  + 7b^{4} b^2 \right)\eta_1\\
&=-\frac{a}{\sqrt{2}b}[ \left(b^2 + 2 b^4 + 4 (1-2b^2)^2  + 7(1-2b^2) b^2 \right)\eta_1\\
&=-2\sqrt{2}a(b^2-1)^2 b^{-1} \eta_1
\end{align*}
by the fact $b^{4}=1-2b^2$. So with the value of $\nu_i$ and $\eta_i$, $i=1,2$ in \eqref{notation2}, we could directly compute that
\[
c_1=a[(1-b^2)\nu_2-(b^3+b)\nu_1]=-2a^2(1+b^2)(b+b^3)=-2a^2b(1+b^2)^2
\]
and
\[
c_2=-2a(b^2-1)^2 b^{-1} \eta_1=2a^3 \frac{(b^2+b^4)^2(4b^2-3)}{b(2b^4-1)}= \frac{2a^3(1-b^2)^2(4b^2-3)}{b(2b^4-1)}.
\].

Combing \eqref{dG3c} and \eqref{dffGa}, we have the result of $\partial_{tt}F_2(\lambda_4,0;a)$ in Proposition \ref{d2F2} (3).

For the results of $\partial_{tt}F_2(\lambda_{2p},0;a)$, $p\ge3$, we will list the value of some key components and omit all the computational details because they are cumbersome without new methods used in the proof of (1)-(3) of Proposition \ref{d2F2} for $p=2$ above. Here we list
\begin{align}\label{E:philga}
\tilde{x}=2b^{-1}
\left(
\begin{array}{cc}
1\\
0
\end{array}
\right)\overline{w}
-2a p b^{1-2p}
\left(
\begin{array}{cc}
1\\
0
\end{array}
\right)\overline{w}^{2p-1},
\end{align}
\begin{align}\label{E:phigga}
\bar{x}=
\left(
\begin{array}{cc}
\nu_1\\
\nu_2
\end{array}
\right)\overline{w}^{2p-3}+
\left(
\begin{array}{cc}
\eta_1\\
\eta_2
\end{array}
\right)\overline{w}^{4p-1},
\end{align}
where
\begin{eqnarray} \label{E:notation-3}
&\nu_1=2ab^2,\quad\quad\nu_{2}=-2a b^{-1},\nonumber\\
&\eta_{1}=-\frac{4a^2(b^2+b^{2p})^2[2pb^2+(1-2p)]}{(b^2-1)^2p [2+4p(b^2-1)+p^2(b^2-1)^2 ]}.
\end{eqnarray}
For $c_1$ and $c_2$, they have the following representations
\begin{eqnarray*}
c_1
=a[((p-2)b+(1-p)b^3-b^{2p-3})\nu_1+(1-b^2)\nu_2],
\end{eqnarray*}
and
\begin{eqnarray*}
c_2= -a(p-1)^2p(b^2-1)^2 b^{-1} \eta_1,
\end{eqnarray*}
which is equivalent to the ones in Proposition \ref{d2F2} (3).
\end{proof}

\section{admissible parameters for non-isolated intersections} \label{S:Degenerate}

As introduced in Section \ref{S:Scheme}, our aim is to find admissible parameter $a \in \R$ such that the quadratic real curves $H_2(\lambda,t;a)$ and $J_2(\lambda,t;a)$ are not intersected with each other transversallity. Due to the computations on the second-order derivatives of $F_2$ given in Proposition \ref{d2F2}, we have
\begin{lemma} \label{L:non-deg}
	For any $a \neq 0$, the bifurcation equation \eqref{E:F2} has only trivial solutions $(\lambda,t)=(\lambda_{2p},0)$ near $(\lambda_{2p},0)$.
\end{lemma}

\begin{proof}

Due to Proposition \ref{d2F2}, the second-order derivatives of $F_2(\lambda,t;a)$ are given as follows
\[
\begin{split}
	& \p_{\lambda \lambda} F_2(\lambda_{2p},0;a) = 4 \sqrt{2} \begin{pmatrix}
		b^{-1} \\ap^2 b^{1-2p}
	\end{pmatrix}, \; \quad \p_{\lambda t} F_2(\lambda_{2p},0;a) = 0, \\
	& \p_{tt} F_2(\lambda_{2p},0;a) = \sqrt{2} \begin{pmatrix}
		-b^{-1}(b^2-1)^2 + c_1 \\
		b^{-1}(b^2-1)^2 p (2+(2p-1)a )a +c_2
	\end{pmatrix},
\end{split}
\]
where for $p=2$,
\begin{eqnarray}\label{c12p=2}
c_1=-2a^2b(b^2+1)^2 ,\quad
c_2= \frac{2 a^3 (b^2-1)^2(4b^2-3)}{  b(2b^4-1)},
\end{eqnarray}
and for $p\ge3$
\begin{eqnarray}\label{c12p}
\quad\quad\quad c_1=-2a^2(b^2-1)^2b^{-1}[p+(p-1)b^2],\quad c_2= \frac{4a^3(p-1)^4  (b^2-1)^2 [2pb^3+(1-2p)b]}{ b^2[2+4p(b^2-1)+p^2(b^2-1)^2]}.
\end{eqnarray}
Consider the intersection of real quadratic curves
\beq\label{E:deg-1}
\begin{cases}
	4b^{-1} \lambda^2 +    \left[-b^{-1}(b^2-1)^2 +c_1 \right]t^2 = 0, \\
	4ap^2 b^{1-2p} \lambda^2 + \left[ 	b^{-1}(b^2-1)^2 p a(2+(2p-1)a^2 ) +c_2 \right]t^2 =0.
\end{cases}
\eeq
Clearly the intersection is not-isolated when $a=0$. On the other hand, suppose $ a\neq 0$,  system \eqref{E:deg-1} is equivalent to
\beq \label{E:deg-2}
\begin{cases}
	4b^{-1} \lambda^2 +    \left[-b^{-1}(b^2-1)^2 +c_1 \right]t^2 = 0, \\
	4p^2 b^{1-2p} \lambda^2 + \left[ 	b^{-1}(b^2-1)^2 p (2+(2p-1)a^2 ) +\frac{c_2}{a} \right]t^2 =0.
\end{cases}
\eeq

If the system \eqref{E:deg-2} has non-trivial solutions, the following equation  must be hold
\[
\begin{split}
	b^{-1}(b^2-1)^2 p[2+(2p-1)a^2] +\frac{c_2}{a}
=	 p^2 b^{2-2p}[-b^{-1}(b^2-1)^2+c_1],
\end{split}
\]
which is equivalent to
\begin{align}\label{E:degenerate}
	b^{-1}(b^2-1)^2 p(2+p b^{2-2p})
=	-b^{-1}(b^2-1)^2 p(2p-1) a^2-\frac{c_2}{a}+p^2 b^{2-2p}c_1.
\end{align}
Due to the fact $0<b<1$ and $p\ge2$, we have that
\[
b^{-1}(b^2-1)^2 p(2+p b^{2-2p})>0, \;\; -b^{-1}(b^2-1)^2 p(2p-1) a^2<0
\]
for $0 \neq a\in \mathbb{R}$. We claim that $\frac{c_2}{a}>0$ and $c_1<0$. Then \eqref{E:degenerate} becomes
\begin{align*}
	0<b^{-1}(b^2-1)^2 p(2+p b^{2-2p})
=	-b^{-1}(b^2-1)^2 p(2p-1) a^2-\frac{c_2}{a}+p^2 b^{2-2p}c_1<0,
\end{align*}
which is impossible. So the system \eqref{E:deg-2} has only trivial solution as well as system \eqref{E:deg-1} when $a\ne 0$ due to the discussion given in Section \ref{S:Scheme}, which implies the result in Lemma \ref{L:non-deg}.

Now it suffices to prove $\frac{c_2}{a}>0$ and $c_1<0$. Recall the formulation of $c_1$ given in \eqref{c12p=2} and \eqref{c12p}. It is obvious that $c_1<0$. When $p=2$, due to \eqref{c12p=2}, one has
\[
\frac{c_2}{a}=\frac{2 a^2 (b^2-1)^2(4b^2-3)}{  b(2b^4-1)}.
\]
Recall when $p=2$, $b^4=1-2b^2$, then we know $b=\sqrt{\sqrt{2}-1}$ and
\[
\frac{4b^2-3}{  2b^4-1}=\frac{3+8\sqrt{2}}{ 7}>0,
\]
which implies that $\frac{c_2}{a}>0$.
When $p\ge3$, from \eqref{c12p}, we have
\[
\frac{c_2}{a}= \frac{4a^2(p-1)^4  (b^2-1)^2 [2pb^3+(1-2p)b]}{ b^2[2+4p(b^2-1)+p^2(b^2-1)^2]}.
\]
To prove $\frac{c_2}{a}>0$, which is equivalent to prove
\[
[2pb^2+(1-2p)][2+4p(b^2-1)+p^2(b^2-1)^2 ]>0.
\]
From the fact $b^{2p}=p-1-pb^{2}$, we obtain $2pb^2+(1-2p)=-1-2 b^{2p}<0$. Denote $b^2-1=z$, consider the equation $p^2z^2+4p z+2=0$, which solutions are $\frac{-2\pm\sqrt{2}}{p}$. If $b^2-1\ge\frac{\sqrt{2}-2}{p} $, then $b^2\ge1+\frac{\sqrt{2}-2}{p} $,
\[
b^{2p}=p-1-p b^2\leq p-1-p (1+\frac{\sqrt{2}-2}{p})=1-\sqrt{2}<0;
\]
if $b^2-1\leq\frac{-\sqrt{2}-2}{p} $, then $b^2\leq1-\frac{\sqrt{2}+2}{p} $,
\[
b^{2p}=p-1-p b^2\ge p-1-p (1-\frac{\sqrt{2}+2}{p})=1+\sqrt{2}>1.
\]
By the fact $0<b<1$, $\forall p\ge3$,  we know $-\frac{\sqrt{2}+2}{p}\leq b^2-1\leq\frac{\sqrt{2}-2}{p} $, and hence $2+4p(b^2-1)+p^2(b^2-1)^2 <0$. The claim holds true.

%So we prove our claim that $\frac{c_1}{a}>0$.

\end{proof}

\section{Higher-order derivatives of $F_2(\lambda,t;0)$} \label{S:Higher-D}

In this section, we shall calculate the explicit formulae for the derivatives of $F_2(\lambda,t;0)$ up to the $p+1$-order for $p=2,3,4$ The main result of this section roughly reads as follow:
\begin{lemma} \label{L:Key-L}
For $p=2,3,4$, only $Q_2 \p_{\lambda \lambda}^2 \p_t^{p-1} F_{2}(\lambda_{2p},0;0) $ is non-zero.
\end{lemma}

\subsection{The third-order derivatives for $a=0$.} \label{S:Third}
In this section, we shall compute the third order derivatives concern $F_2$ given in \eqref{E:F2} with $a=0$. For simple, in the following we just use $F_2(\lambda, t) $ to represent $F_2(\lambda,t;0)$.

\subsubsection{Preparations}
%We give some friendly using lemmas for calculations here, which proofs are written in the Appendix.
\begin{lemma}\label{lem1}
With $G$ defined in \eqref{E:Non}, for $f\!=\!\sum_{n=1}^{\infty} \!
\left(\!
	\begin{array}{cc}
		a_n^1\\
		a_n^2
	\end{array}
	\!\right) \overline{w}^{2n-1}, \,
 g_k\!=\!\sum_{n=1}^{\infty}\!
\left(\!
	\begin{array}{cc}
		\beta_n^{k,1}\\
		\beta_n^{k,2}
	\end{array}
	\!\right) \overline{w}^{2n-1}$, we have
\begin{eqnarray*}
\partial_{\lambda}\partial_fG(\lambda_{2p},0)f\!=\!
\sum_{n=1}^{\infty}
\left(
	\begin{array}{cc}
		2n a_n^1\\
		2n b a_n^2
	\end{array}
	\right) e_{2n},\,
\partial^l_{\lambda}\partial^k_{f}G(\lambda_{2p},0)[g_1,...,g_k]\!=\!0,\,l\ge2, k\ge 0\text{ or } l\!=\!1, k\ge3
\end{eqnarray*}
and
\begin{small}
\begin{eqnarray*}
\partial_{\lambda}\partial^2_{f}G_j(\lambda_{2p},0)[g_1,g_2]\!=\!\mathrm{Im}\!
\left[
\sum_{n=1}^\infty (\beta_n^{1,j} w^{2n})\sum_{n=1}^\infty (2n\!-\!1)(\beta_n^{2,j} \overline{w}^{2n})
\!+\!\sum_{n=1}^\infty (\beta_n^{2,j} w^{2n})\sum_{n=1}^\infty (2n\!-\!1)(\beta_n^{1,j} \overline{w}^{2n})
\right].
\end{eqnarray*}
\end{small}
\end{lemma}

\begin{proof}%[Proof of Lemma \ref{lem1}]
	Denote $\phi_j(w)=b_j w+t\sum_{n=1}^\infty a_n^j \overline{w}^{2n-1}+\sum_{l=1}^ks_l \left( \sum_{n=1}^{\infty}\beta_{n}^{l,j}\overline{w}^{2n-1}\right)$. Then by the definition of $G_j(\lambda,\phi_1,\phi_2)(w)$ in \eqref{E:Non}, it is easy to check that
	\begin{eqnarray*}
		\partial_{\lambda}G_j(\lambda_{2p},0)=-{\rm Im}\{\overline{\phi_j(w)}w\phi'_j(w)\}.
	\end{eqnarray*}
	So we can directly compute our results as following
	\begin{small}
		\begin{align*}
			\partial_{\lambda}\partial_fG_j(\lambda_{2p},0)f&=\partial_{\lambda}\partial_t G_j(\lambda_{2p},\phi_1,\phi_2)|_{t,s_1,...,s_k=0}=
			\sum_{n=1}^{\infty}
			2n b_j a_n^j e_{2n},
			\\
			\partial_{\lambda}\partial^2_{f}G_j(\lambda_{2p},0)[g_1,g_2]\!
			&=\partial_{\lambda}\partial_{s_1}\partial_{s_2} G_j(\lambda_{2p},\phi_1,\phi_2)|_{t,s_1,...,s_k=0}
			\\&=
			\mathrm{Im}\!
			\left[
			\sum_{n=1}^\infty (\beta_n^{1,j} w^{2n})\sum_{n=1}^\infty (2n\!-\!1)(\beta_n^{2,j} \overline{w}^{2n})
			\!+\!\sum_{n=1}^\infty (\beta_n^{2,j} w^{2n})\sum_{n=1}^\infty (2n\!-\!1)(\beta_n^{1,j} \overline{w}^{2n})
			\right]
		\end{align*}
	\end{small}
	and
	$
	\partial^l_{\lambda}\partial^k_{f}G(\lambda_{2p},0)[g_1,...,g_k]
	=0,\,l\ge2, k\ge 0\text{ or } l\!=\!1, k\ge3
	$
\end{proof}

\begin{lemma}\label{lem2}
With $G$ defined in \eqref{E:Non}, for
$f_1=
\left(
	\begin{array}{cc}
		a_1\\
		a_2
	\end{array}
	\right) \overline{w}, \,
f_2=
\left(
	\begin{array}{cc}
		\beta_1\\
		\beta_2
	\end{array}
	\right) \overline{w},\,
 f_3=
\left(
	\begin{array}{cc}
		\gamma_1\\
		\gamma_2
	\end{array}
	\right) \overline{w}
$, $f_4=
\left(
	\begin{array}{cc}
		\sigma_1\\
		\sigma_2
	\end{array}
	\right) \overline{w}
$, we have
\begin{align*}
&\partial_{ff}G(\lambda_{2p},0)[f_1,f_2]=4b^2(\alpha_1-b\alpha_2)(\beta_1-b\beta_2)
\left(
	\begin{array}{cc}
		1\\
		0
	\end{array}
	\right) e_{4},\\
&\partial_{fff}G(\lambda_{2p},0)[f_1,f_2,f_3]=
\left(
	\begin{array}{cc}
		c_3\\
		c_4
	\end{array}
	\right) e_{2}+
c_5
\left(
	\begin{array}{cc}
		1\\
		0
	\end{array}
	\right) e_{6},\\
&\partial^4_{f}G(\lambda_{2p},0)[f_1,f_2,f_3,f_4]=c_6
\left(
	\begin{array}{cc}
		1\\
		0
	\end{array}
	\right) e_{4}+
c_7
\left(
	\begin{array}{cc}
		1\\
		0
	\end{array}
	\right) e_{8},
\end{align*}
where
\begin{small}
\begin{align*}
c_3&\!=\!-6\alpha_1\beta_1\gamma_1+4(\alpha_2\beta_2\gamma_1+\alpha_2\beta_1\gamma_2+\alpha_1\beta_2\gamma_2)-6b\alpha_2\beta_2\gamma_2,\\ c_4&\!=\!-2(\alpha_2\beta_2\gamma_1+\alpha_2\beta_1\gamma_2+\alpha_1\beta_2\gamma_2)+6b^{-1}\alpha_2\beta_2\gamma_2, \\ c_5&\!=\!-12b^2\alpha_1\beta_1\gamma_1\!+\!18b^3(\alpha_2\beta_1\gamma_1\!+\!\alpha_1\beta_2\gamma_1\!+\!\alpha_1\beta_1\gamma_2)
\!-\!24b^4(\alpha_2\beta_2\gamma_1\!+\!\alpha_2\beta_1\gamma_2\!+\!\alpha_1\beta_2\gamma_2)\!+\!30b^5\alpha_2\beta_2\gamma_2,\\
c_6&=-8(\alpha_2\beta_2\gamma_1\sigma_1+\alpha_2\beta_1\gamma_2\sigma_1+\alpha_1\beta_2\gamma_2\sigma_1+\alpha_2\beta_1\gamma_1\sigma_2
+\alpha_1\beta_2\gamma_1\sigma_2+\alpha_1\beta_1\gamma_2\sigma_2)\\
&+24b(\alpha_2\beta_2\gamma_2\sigma_1+\alpha_2\beta_2\gamma_1\sigma_2+\alpha_2\beta_1\gamma_2\sigma_2+\alpha_1\beta_2\gamma_2\sigma_2)
-48b^2\alpha_2\beta_2\gamma_2\sigma_2,\\
c_7&=
48b^2 \alpha_1\beta_1\gamma_1\sigma_1-96b^3(\alpha_2\beta_1\gamma_1\sigma_1+\alpha_1\beta_2\gamma_1\sigma_1+\alpha_1\beta_1\gamma_2\sigma_1+\alpha_1\beta_1\gamma_1\sigma_2)
\\&+160b^4(\alpha_2\beta_2\gamma_1\sigma_1+\alpha_2\beta_1\gamma_2\sigma_1+\alpha_1\beta_2\gamma_2\sigma_1+\alpha_2\beta_1\gamma_1\sigma_2
+\alpha_1\beta_2\gamma_1\sigma_2+\alpha_1\beta_1\gamma_2\sigma_2)
\\&-240b^5(\alpha_2\beta_2\gamma_2\sigma_1+\alpha_2\beta_2\gamma_1\sigma_2+\alpha_2\beta_1\gamma_2\sigma_2+\alpha_1\beta_2\gamma_2\sigma_2)
+336b^6\alpha_2\beta_2\gamma_2\sigma_2.
\end{align*}
\end{small}
\end{lemma}

\begin{proof}%[Proof of Lemma \ref{lem2}]
	Denote $\phi_j(w)=b_j w+t\alpha_j\overline{w}+s \beta_j\overline{w}+r \gamma_j \overline{w}+h \sigma_j\overline{w}$. Then by the definition of $G_j(\lambda,\phi_1,\phi_2)(w)$ in \eqref{E:Non}, we know
	\begin{align}\label{E:Gff}
		&\partial^2_{f}G_j(\lambda_{2p},0)[f_1,f_2]=\partial_t\partial_s G_j(\lambda_{2p},\phi_1,\phi_2)|_{t,s,r,h=0}\\
		=&
		{\rm Im}\{b_jw\frac{\partial^2 }{\partial_t\partial_s}I(\phi_j)|_{t,s,r,h=0}\}
		-{\rm Im}\{\alpha_j\overline{w}\frac{\partial }{\partial_s}I(\phi_j)|_{t,s,r,h=0}\}
		-{\rm Im}\{\beta_j\overline{w}\frac{\partial }{\partial_t}I(\phi_j)|_{t,s,r,h=0}\},\nonumber
	\end{align}
	\begin{align*}
		&\partial^3_{f}G_j(\lambda_{2p},0)[f_1,f_2,f_3]=\partial_t\partial_s\partial_r G_j(\lambda_{2p},\phi_1,\phi_2)|_{t,s,r,h=0}\\
		=&
		{\rm Im}\{b_jw\frac{\partial^3 }{\partial_t\partial_s\partial_r}I(\phi_j)|_{t,s,r,h=0}\}
		-{\rm Im}\{\alpha_j\overline{w}\frac{\partial^2 }{\partial_s\partial_r}I(\phi_j)|_{t,s,r,h=0}\}\\
		-&{\rm Im}\{\beta_j\overline{w}\frac{\partial^2 }{\partial_t\partial_r}I(\phi_j)|_{t,s,r,h=0}\}
		\!-\!{\rm Im}\{\gamma_j\overline{w}\frac{\partial^2 }{\partial_s\partial_t}I(\phi_j)|_{t,s,r,h=0}\},
	\end{align*}
	\begin{align*}
		&\partial^4_{f}G_j(\lambda_{2p},0)[f_1,f_2,f_3,f_4]=\partial_t\partial_s\partial_r\partial_h G_j(\lambda_{2p},\phi_1,\phi_2)|_{t,s,r,h=0}\\
		=&
		{\rm Im}\{b_jw\frac{\partial^4 }{\partial_t\partial_s\partial_r\partial_h}I(\phi_j)|_{t,s,r,h=0}\}
		-{\rm Im}\{\alpha_j\overline{w}\frac{\partial^3 }{\partial_s\partial_r\partial_h}I(\phi_j)|_{t,s,r,h=0}\}\\
		-&{\rm Im}\{\beta_j\overline{w}\frac{\partial^3 }{\partial_t\partial_r\partial_h}I(\phi_j)|_{t,s,r,h=0}\}
		\!-\!{\rm Im}\{\gamma_j\overline{w}\frac{\partial^3 }{\partial_s\partial_t\partial_h}I(\phi_j)|_{t,s,r,h=0}\}
		\!-\!{\rm Im}\{\sigma_j\overline{w}\frac{\partial^3 }{\partial_s\partial_r\partial_t}I(\phi_j)|_{t,s,r,h=0}\}.
	\end{align*}
	where $I(\phi_j)=I_1(\phi_j)-I_2(\phi_j)$ and
	\begin{align*}
		I_i(\phi_j)=\int \frac{\overline{A}+tB+sC+rD+hJ}{A+t\overline{B}+s\overline{C}+r\overline{D}+h\overline{J}}
		(b_i-tE-sF-rH-hK)\dif \tau
	\end{align*}
	with
	\begin{align*}
		&A=b_jw-b_i\tau,\, B=\alpha_j w-\alpha_i \tau,\, C=\beta_j w-\beta_i \tau,\, D=\gamma_j w-\gamma_i \tau, \\
		&J=\sigma_j w-\sigma_i \tau,\,
		E=\alpha_i \overline{\tau}^2,\,
		F=\beta_i \overline{\tau}^2,\,
		H=\gamma_i \overline{\tau}^2,\,
		K=\sigma_i \overline{\tau}^2.
	\end{align*}
	Since $f_1$, $f_2$, $f_3$, $f_4$ are symmetric, we only need to compute the value of $\frac{\partial }{\partial_t}I(\phi_j)|_{t,s,r,h=0}$, $\frac{\partial^2 }{\partial_t\partial_s}I(\phi_j)|_{t,s,r,h=0}$, $\frac{\partial^3 }{\partial_t\partial_s\partial_r}I(\phi_j)|_{t,s,r,h=0}$ and $\frac{\partial^4 }{\partial_t\partial_s\partial_r\partial_h}I(\phi_j)|_{t,s,r,h=0}$.
	
	Observe that
	\begin{align*}
		\frac{\partial}{\partial_t}I_i(\phi_j)|_{t,s,r,h=0}=\int \frac{b_i B-\overline{A}D}{A}\dif \tau-\int \frac{b_i \overline{A}\overline{B}}{A^2}\dif \tau.
	\end{align*}
	By a series calculations like the proof of Proposition \ref{JF2c}, we have
	\begin{align*}
		&\frac{\partial}{\partial_t}I_1(\phi_1)|_{t,s,r,h=0}=\frac{\partial}{\partial_t}I_2(\phi_2)|_{t,s,r,h=0}=0,\\
		&\frac{\partial}{\partial_t}I_1(\phi_2)|_{t,s,r,h=0}=(b \alpha_1-\alpha_2)w,
		\frac{\partial}{\partial_t}I_2(\phi_1)|_{t,s,r,h=0}=(b^2\alpha_1-b^3\alpha_2)\overline{w}^3,\\
		&\frac{\partial}{\partial_t}I(\phi_1)|_{t,s,r,h=0}=(b^3\alpha_2 -b^2\alpha_1)\overline{w}^3
		\text{ and } \frac{\partial}{\partial_t}I(\phi_2)|_{t,s,r,h=0}=(b \alpha_1-\alpha_2)w,
	\end{align*}
	which implies that
	\begin{align*}
		\frac{\partial}{\partial_s}I(\phi_1)|_{t,s,r,h=0}=(b^3\beta_2 -b^2\beta_1)\overline{w}^3
		\text{ and } \frac{\partial}{\partial_t}I(\phi_2)|_{t,s,r,h=0}=(b \beta_1-\beta_2)w.
	\end{align*}
	Similarly, by a series calculations like the proof of Proposition \ref{JF2c}, we have
	\begin{align*}
		&\frac{\partial^2}{\partial_t\partial_s}I(\phi_1)|_{t,s,r,h=0}
		=(2\alpha_1\beta_1-2\alpha_2\beta_2)\overline{w}+[2b^2 \alpha_1\beta_1-3b^3(\alpha_1\beta_2+\alpha_2\beta_1)+4b^4\alpha_2\beta_2]\overline{w}^5,\\ &\frac{\partial^2}{\partial_t\partial_s}I(\phi_2)|_{t,s,r,h=0}=( \alpha_2\beta_1+\alpha_1\beta_2-2b^{-1}\alpha_2\beta_2)\overline{w}.
	\end{align*}
	Inserting the above equations into \eqref{E:Gff}, we obtain the value of $\partial_{ff}G_j(\lambda_{2p},0)[f_1,f_2]$, that is, $\partial_{ff}G(\lambda_{2p},0)[f_1,f_2]=4b^2(\alpha_1-b\alpha_2)(\beta_1-b\beta_2)
	\left(
	\begin{array}{cc}
		1\\
		0
	\end{array}
	\right) e_{4}$.
	
	Since the calculation process is tedious with the same methods in the proof of Proposition \ref{JF2c}, we omit this process and give some important calculation results instead, which are
	\begin{align*}
		\frac{\partial^3 }{\partial_t\partial_s\partial_r}I(\phi_1)|_{t,s,r,h=0}
		=&[2(\alpha_2\beta_2\gamma_1+\alpha_2\beta_1\gamma_2+\alpha_1\beta_2\gamma_2)-6b\alpha_2\beta_2\gamma_2]\overline{w}^3\\
		-&[6b^2 \alpha_1\beta_1\gamma_1-12b^3(\alpha_2\beta_1\gamma_1+\alpha_1\beta_2\gamma_1+\alpha_1\beta_1\gamma_2)\\
		&+20b^4(\alpha_1\beta_2\gamma_2+\alpha_2\beta_1\gamma_2+\alpha_2\beta_2\gamma_1)-30b^5\alpha_2\beta_2\gamma_2]\overline{w}^5,\\
		\frac{\partial^3 }{\partial_t\partial_s\partial_r}I(\phi_2)|_{t,s,r,h=0}=&0.
	\end{align*}
	and
	\begin{align*}
		\frac{\partial^4 }{\partial_t\partial_s\partial_r\partial_h}I(\phi_1)|_{t,s,r,h=0}
		=\tilde{c}_1\overline{w}^5+\tilde{c}_2 \overline{w}^9,\quad \frac{\partial^4 }{\partial_t\partial_s\partial_r\partial_h}I(\phi_2)|_{t,s,r,h=0}=0,
	\end{align*}
	with
	\begin{align*}
		\tilde{c}_1&=-4(
		\alpha_2\beta_2\gamma_1\sigma_1+\alpha_2\beta_1\gamma_2\sigma_1+\alpha_1\beta_2\gamma_2\sigma_1+\alpha_2\beta_1\gamma_1\sigma_2+\alpha_1\beta_2\gamma_1\sigma_2+\alpha_1\beta_1\gamma_2\sigma_2)
		\\
		&+18b(\alpha_2\beta_2\gamma_2\sigma_1+\alpha_2\beta_2\gamma_1\sigma_2+\alpha_2\beta_1\gamma_2\sigma_2+\alpha_1\beta_2\gamma_2\sigma_2)+48b^2\alpha_2\beta_2\gamma_2\sigma_2\\
		\tilde{c}_2&=6[4b^2\alpha_1\beta_1\gamma_1\sigma_1
		-10b^3(\alpha_2\beta_1\gamma_1\sigma_1+\alpha_1\beta_2\gamma_1\sigma_1+\alpha_1\beta_1\gamma_2\sigma_1+\alpha_1\beta_1\gamma_1\sigma_2)
		\\&+20b^4(\alpha_2\beta_2\gamma_1\sigma_1+\alpha_2\beta_1\gamma_2\sigma_1+\alpha_1\beta_2\gamma_2\sigma_1+\alpha_2\beta_1\gamma_1\sigma_2+\alpha_1\beta_2\gamma_1\sigma_2+\alpha_1\beta_1\gamma_2\sigma_2)
		\\
		&-35b^5(\alpha_2\beta_2\gamma_2\sigma_1+\alpha_2\beta_2\gamma_1\sigma_2+\alpha_2\beta_1\gamma_2\sigma_2+\alpha_1\beta_2\gamma_2\sigma_2)+56b^6\alpha_2\beta_2\gamma_2\sigma_2.
	\end{align*}
	
\end{proof}

\begin{lemma}\label{lem3}
With $G$ defined in \eqref{E:Non}, for
$f_1=
\left(
	\begin{array}{cc}
		a_1\\
		a_2
	\end{array}
	\right) \overline{w}, \,
f_2=
\left(
	\begin{array}{cc}
		\beta_1\\
		\beta_2
	\end{array}
	\right) \overline{w}^3\,
$, we have
\begin{align*}
\partial_{ff}G(\lambda_{2p},0)[f_1,f_2]&=
\left(
	\begin{array}{cc}
		b^2 \alpha_1\beta_1-b^2\alpha_2\beta_2\\
		3b^2\alpha_1\beta_1-2b^3\alpha_2\beta1-2b \alpha_1\beta_2+b^2\alpha_2\beta_2
	\end{array}
	\right) e_{2}
\\&+
6b^2(\alpha_1-b \alpha_2 )(\beta_1- b^3\beta_2 )
\left(
	\begin{array}{cc}
		1\\
	    0
	\end{array}
	\right) e_{6}.
\end{align*}
\end{lemma}

%\begin{lemma}\label{lem4}
%With $G$ defined in \eqref{E:Non}, for
%$f_1=
%\left(
%	\begin{array}{cc}
%		a_1\\
%		a_2
%	\end{array}
%	\right) \overline{w}, \,
%f_2=
%\left(
%	\begin{array}{cc}
%		\beta_1\\
%		\beta_2
%	\end{array}
%	\right) \overline{w}\,
%f_3=
%\left(
%	\begin{array}{cc}
%		\gamma_1\\
%		\gamma_2
%	\end{array}
%	\right) \overline{w}^3
%$, we have
%\begin{align*}
%\partial_{fff}G(\lambda_{2p},0)[f_1,f_2,f_3]&=
%\left(
%	\begin{array}{cc}
%		c_8\\
%		c_9
%	\end{array}
%	\right) e_{4}
%+
%c_{10}
%\left(
%	\begin{array}{cc}
%		1\\
%	    0
%	\end{array}
%	\right) e_{8},
%\end{align*}
%where
%\begin{small}
%\begin{align*}
%c_8&\!=\!8(\alpha_2\beta_2\gamma_1\!-\!\alpha_1\beta_1\gamma_1)\!+\!4b^2(\alpha_2\beta_1\gamma_2\!+\!\alpha_1\beta_2\gamma_2)\!-\!8b^3\alpha_2\beta_2\gamma_2,
%\,c_9\!=\!-4(\alpha_2\beta_1\gamma_2\!+\!\alpha_1\beta_2\gamma_2)\!+\!8b^{-1}\alpha_2\beta_2\gamma_2, \\ c_{10}&\!=\!-16b^2\alpha_1\beta_1\gamma_1\!+\!24b^3(\alpha_2\beta_1\gamma_1\!+\!\alpha_1\beta_2\gamma_1)\!-\!32b^4\alpha_2\beta_2\gamma_1
%\!+\!40b^5\alpha_1\beta_1\gamma_2\!-\!48b^6(\alpha_1\beta_2\gamma_2\!+\!\alpha_2\beta_1\gamma_2)\!\\&+\!56b^7\alpha_2\beta_2\gamma_2.
%\end{align*}
%\end{small}
%\end{lemma}

\begin{proof}%[Proof of Lemma \ref{lem3}]
	Denote $\phi_j(w)=b_j w+t\alpha_j\overline{w}+s \beta_j\overline{w}^3$. Then by the definition of $G_j(\lambda,\phi_1,\phi_2)(w)$ in \eqref{E:Non}, we know
	\begin{align}\label{E:Gff3}
		&\partial^2_{f}G_j(\lambda_{2p},0)[f_1,f_2]=\partial_t\partial_s G_j(\lambda_{2p},\phi_1,\phi_2)|_{t,s=0}\nonumber\\
		=&
		{\rm Im}\{b_jw\frac{\partial^2 }{\partial_t\partial_s}I(\phi_j)|_{t,s=0}\}
		-{\rm Im}\{\alpha_j\overline{w}\frac{\partial }{\partial_s}I(\phi_j)|_{t,s=0}\}\nonumber\\
		-&3{\rm Im}\{\beta_j\overline{w}^3\frac{\partial }{\partial_t}I(\phi_j)|_{t,s=0}\}+2(\lambda_{2p}-1)
		\begin{pmatrix}
			\alpha_1\beta_1 \\ \alpha_2\beta_2
		\end{pmatrix}
		,
	\end{align}
	where $2(\lambda_{2p}-1)=b^2-1$, $I(\phi_j)=I_1(\phi_j)-I_2(\phi_j)$ and
	\begin{align*}
		I_i(\phi_j)=\int \frac{\overline{A}+tB+sC}{A+t\overline{B}+s\overline{C}}
		(b_i-tD-sE)\dif \tau
	\end{align*}
	with
	\begin{align*}
		&A=b_jw-b_i\tau,\, B=\alpha_j w-\alpha_i \tau,\, C=\beta_j w-\beta_i \tau,\, D=\alpha_i \overline{\tau}^2,\,
		E=3\beta_i \overline{\tau}^4.
	\end{align*}
	From the proof of Lemma \ref{lem2}, we know
	\begin{align*}
		\frac{\partial}{\partial_t}I(\phi_1)|_{t,s=0}=(b^3\alpha_2 -b^2\alpha_1)\overline{w}^3
		\text{ and } \frac{\partial}{\partial_t}I(\phi_2)|_{t,s=0}=(b \alpha_1-\alpha_2)w,
	\end{align*}
	Observe that
	\begin{small}
		\begin{align*}
			&\frac{\partial}{\partial_s}I_i(\phi_j)|_{t,s=0}=\int \frac{b_i C-\overline{A}E}{A}\dif \tau-\int \frac{b_i \overline{A}\overline{C}}{A^2}\dif \tau,\\
			&\frac{\partial^2}{\partial_t\partial_s}I_i(\phi_j)|_{t,s=0}=-\int \frac{D C+BE}{A}\dif \tau-\int \frac{b_i C-\overline{A}E}{A^2}\overline{B}\dif \tau
			-\int \frac{b_i B-\overline{A}D}{A^2}\overline{C}\dif \tau+2\int \frac{b_i \overline{A}\overline{B}\overline{C}}{A^3}\dif \tau.
		\end{align*}
	\end{small}
	By a series calculations like the proof of Proposition \ref{JF2c}, we have
	\begin{align*}
		&\frac{\partial}{\partial_s}I(\phi_1)|_{t,s=0}=(b^5\beta_2 -b^2\beta_1)\overline{w}^5,\quad
		\frac{\partial}{\partial_s}I(\phi_2)|_{t,s=0}=(b^3 \beta_1-\beta_2)w^3,
	\end{align*}
	and
	\begin{align*}
		&\frac{\partial^2}{\partial_t\partial_s}I(\phi_1)|_{t,s=0}
		=(\alpha_1\beta_1-b^2\alpha_2\beta_2)\overline{w}^3+(2b^2 \alpha_1\beta_1-3b^3\alpha_2\beta_1-5b^5\alpha_1\beta_2+6b^6\alpha_2\beta_2)\overline{w}^7,\\ &\frac{\partial^2}{\partial_t\partial_s}I(\phi_2)|_{t,s=0}=3\beta_1(-b\alpha_1+b^2\alpha_2)+( \alpha_1\beta_2-b^{-1}\alpha_2\beta_2)\overline{w}^3.
	\end{align*}
	Inserting the above equations into \eqref{E:Gff3}, we obtain the value of $\partial_{ff}G_j(\lambda_{2p},0)[f_1,f_2]$ in the Lemma.
\end{proof}

\subsubsection{The third-order derivatives of $\varphi$ }
With $a=0$, we immediately get that $x_0=x^1=		\left(
		\begin{array}{cc}
			b\\
			1
		\end{array}
		\right)\overline{w}$, $\bar{x}=\partial_{gg}\varphi(\lambda_{2p},0)[x_0,x_0]=0$ and $\tilde{x}=\partial_\lambda\partial_{g}\varphi(\lambda_{2p},0)x_0=2b^{-1}\left(
	\begin{array}{cc}
		1\\
		0
	\end{array}
	\right)\overline{w} $ by \eqref{E:philga} and \eqref{E:phigga}.
\begin{lemma}\label{phi2a=0}
With $G$ defined in \eqref{E:Non}, we have
\begin{align*}
&\partial_{\lambda\lambda\lambda}\varphi(\lambda_{2p},0)=0,\quad\partial_{\lambda\lambda}\partial_g\varphi(\lambda_{2p},0)x_0=-2[\partial_f G(\lambda_{2p},0)]^{-1}({\rm {Id}}-Q)\partial_\lambda\partial_f G(\lambda_{2p},0)\tilde{x},\\
&\partial_{\lambda}\partial_{gg}\varphi(\lambda_{2p},0)[x_0,x_0]=-2[\partial_f G(\lambda_{2p},0)]^{-1}({\rm{Id}}-Q)
\partial_{ff}G(\lambda_{2p},0)[\tilde{x},x_0]\text{ and }\\
&\partial_{ggg}\varphi(\lambda_{2p},0)[x_0,x_0,x_0]=-[\partial_f G(\lambda_{2p},0)]^{-1}({\rm{Id}}-Q)
\partial_{fff} G(\lambda_{2p},0)[x_0,x_0,x_0],
\end{align*}
where $\tilde{x}$ is defined in Proposition \ref{acformal} with $a=0$.
\end{lemma}

\begin{proof}
	By the definition of $\varphi$, we know that for any $\lambda$ close to $\lambda_{2p}$,
	\begin{eqnarray}\label{F1phia=0}
		({\rm {Id}}-Q)G(\lambda,tx_0+\varphi(\lambda,tx_0))=0.
	\end{eqnarray}
	By calculating $\partial_{\lambda\lambda\lambda}$, $\partial_{\lambda\lambda}\partial_{t}$, $\partial_{\lambda}\partial_{tt}$ and $\partial_{ttt}$  on the both side of \eqref{F1phia=0} respectively at the point $(\lambda_{2p},0)$, we know
\begin{align}
&({\rm {Id}}-Q)\big\{\partial_{\lambda\lambda\lambda}G(\lambda_{2p},0)
+3\partial_{\lambda\lambda}\partial_fG(\lambda_{2p},0)\partial_{\lambda}\varphi(\lambda_{2p},0)
+3\partial_{\lambda}\partial_fG(\lambda_{2p},0)\partial_{\lambda\lambda}\varphi(\lambda_{2p},0)\label{lll}\\
+&3\partial_{\lambda}\partial_{ff}G(\lambda_{2p},\!0)[\partial_{\lambda}\varphi(\lambda_{2p},\!0),\partial_{\lambda}\varphi(\lambda_{2p},0)]		\!+\!\partial_{fff}G(\lambda_{2p},0)[\partial_{\lambda}\varphi(\lambda_{2p},0),\partial_{\lambda}\varphi(\lambda_{2p},0),\partial_{\lambda}\varphi(\lambda_{2p},0)]
\nonumber
\\+&3\partial_{ff}G(\lambda_{2p},0)[\partial_{\lambda\lambda}\varphi(\lambda_{2p},0),\partial_{\lambda}\varphi(\lambda_{2p},0)]
+\partial_fG(\lambda_{2p},0)\partial_{\lambda\lambda\lambda}\varphi(\lambda_{2p},0)\big\}=0,\nonumber
\end{align}
\begin{align}
&({\rm {Id}}-Q)\big\{\partial_{\lambda\lambda}\partial_fG(\lambda_{2p},0)(x_0+\partial_g\varphi(\lambda_{2p},0)x_0)
\!+\!\partial_{ff}G(\lambda_{2p},0)[\partial_{\lambda\lambda}\varphi(\lambda_{2p},0),x_0\!+\!\partial_{g}\varphi(\lambda_{2p},0)x_0]\label{llt}
\\+&2\partial_{\lambda}\partial_{f}G(\lambda_{2p},0)\partial_{\lambda}\partial_g\varphi(\lambda_{2p},0)x_0 +\partial_{fff}G(\lambda_{2p},0)[\partial_{\lambda}\varphi(\lambda_{2p},0),\partial_{\lambda}\varphi(\lambda_{2p},0),x_0\!+\!\partial_{g}\varphi(\lambda_{2p},0)x_0]\nonumber
\\+&2\partial_{ff}G(\lambda_{2p},0)[\partial_{\lambda}\varphi(\lambda_{2p},0),\partial_{\lambda}\partial_g\varphi(\lambda_{2p},0)x_0]
\!+\!2\partial_{\lambda}\partial_{ff}G(\lambda_{2p},0)[\partial_{\lambda}\varphi(\lambda_{2p},0),x_0\!+\!\partial_{g}\varphi(\lambda_{2p},0)x_0]
\nonumber
\\+&\partial_fG(\lambda_{2p},0)\partial_{\lambda\lambda}\partial_g\varphi(\lambda_{2p},0)x_0\big\}=0,\nonumber
\end{align}
\begin{align}
&({\rm {Id}}-Q)\big\{\partial_{\lambda}\partial_{ff}G(\lambda_{2p},0)[x_0+\partial_{g}\varphi(\lambda_{2p},0)x_0,x_0+\partial_{g}\varphi(\lambda_{2p},0)x_0]
+\partial_{\lambda}\partial_{f}G(\lambda_{2p},0)\bar{x}\label{ltt} \\+&\partial_{fff}G(\lambda_{2p},0)[\partial_{\lambda}\varphi(\lambda_{2p},0),x_0+\partial_{g}\varphi(\lambda_{2p},0)x_0,x_0+\partial_{g}\varphi(\lambda_{2p},0)x_0]
\nonumber
\\+&2\partial_{ff}G(\lambda_{2p},0)[\partial_{\lambda}\partial_{g}\varphi(\lambda_{2p},0)x_0,x_0+\partial_{g}\varphi(\lambda_{2p},0)x_0]
+\partial_{ff}G(\lambda_{2p},0)[\partial_{\lambda}\varphi(\lambda_{2p},0),\bar{x}]\nonumber\\
+&\partial_fG(\lambda_{2p},0)\partial_{\lambda}\partial_{gg}\varphi(\lambda_{2p},0)[x_0,x_0]\big\}=0\nonumber
\end{align}
\begin{align}
&({\rm {Id}}-Q)\big\{\partial_{fff}G(\lambda_{2p},0)[x_0+\partial_{g}\varphi(\lambda_{2p},0)x_0,x_0+\partial_{g}\varphi(\lambda_{2p},0)x_0,x_0+\partial_{g}\varphi(\lambda_{2p},0)x_0]
\label{ttt}
\\+&3\partial_{ff}G(\lambda_{2p},0)[\bar{x},x_0+\partial_{g}\varphi(\lambda_{2p},0)x_0]
+\partial_{f}G(\lambda_{2p},0)\partial_{ggg}\varphi(\lambda_{2p},0)[x_0,x_0,x_0]\big\}=0,\nonumber
\end{align}
	
With the facts that $\partial_\lambda\varphi(\lambda_{2p},0)=\partial_{\lambda\lambda}\varphi(\lambda_{2p},0)=\partial_g\varphi(\lambda_{2p},0)x_0=0$, $\bar{x}=0$, $\partial_{\lambda\lambda\lambda}G(\lambda_{2p},0)=\partial_{\lambda\lambda}\partial_fG(\lambda_{2p},0)=0$, and $Q\partial_fG(\lambda_{2p},0)=0$, using Lemma \ref{lem1}, the equations \eqref{lll}, \eqref{llt}, \eqref{ltt} and \eqref{ttt} turns to be
\begin{eqnarray*}
\partial_fG(\lambda_{2p},0)\partial_{\lambda\lambda\lambda}\varphi(\lambda_{2p},0)=0,
\end{eqnarray*}
\begin{eqnarray*}
2({\rm {Id}}-Q)\partial_\lambda\partial_{f}G(\lambda_{2p},0)\partial_\lambda\partial_g\varphi(\lambda_{2p},0)x_0
+({\rm {Id}}-Q)\partial_{f}G(\lambda_{2p},0)\partial_{\lambda\lambda}\partial_{g}\varphi(\lambda_{2p},0)x_0=0,
\end{eqnarray*}
\begin{align*}
&2({\rm {Id}}-Q)\partial_{ff}G(\lambda_{2p},0)[\partial_\lambda\partial_g\varphi(\lambda_{2p},0)x_0,x_0]
+\partial_{f}G(\lambda_{2p},0)\partial_{\lambda}\partial_{gg}\varphi(\lambda_{2p},0)[x_0,x_0]=0,
\end{align*}
and
\begin{align*}
({\rm {Id}}-Q)\partial_{fff}G(\lambda_{2p},0)[x_0,x_0,x_0]
+&\partial_{f}G(\lambda_{2p},0)\partial_{ggg}\varphi(\lambda_{2p},0)[x_0,x_0,x_0]=0,
\end{align*}
which just implies the results in Lemma \ref{phi2a=0}.
\end{proof}

Next, we will compute the third derivatives of $\varphi$ at $(\lambda_{2p},0)$, that is,
\begin{proposition}\label{d3phia=0}
	We give some results of third derivatives of $\varphi$ at $(\lambda_{2p},0)$:
	
	(1) $	\partial_{\lambda\lambda\lambda} \varphi(\lambda_{2p},0)=0.$

	(2) $		\partial_{\lambda\lambda}\partial_g \varphi(\lambda_{2p},0)x_0
		=4b^{-3}
		\left(
		\begin{array}{cc}
			1\\
			0
		\end{array}
		\right)\overline{w}.$
	
	(3) $	\partial_{\lambda}\partial_{gg} \varphi(\lambda_{2p},0)[x_0,x_0]=0.$

	(4) $		\partial_{ggg} \varphi(\lambda_{2p},0)[x_0,x_0,x_0]
		=
		3(b^{-3}-b)
		\left(
		\begin{array}{cc}
			1\\
			0
		\end{array}
		\right)
		\overline{w}.$

\end{proposition}

\begin{proof}
Recall that when $a=0$,
$\tilde{x}=
\partial_\lambda\partial_g \varphi(\lambda_{2p},0)x_0
=2b^{-1}
\left(
\begin{array}{cc}
1\\
0
\end{array}
\right)\overline{w}.
$
%\begin{eqnarray*}
%\partial_{gg} \varphi(\lambda_{2p},0)[x_0,x_0]=
%-\left[\partial_{f} G(\lambda_{2p},0)\right]^{-1}\frac{\dif^2}{\dif t^2}({\rm {Id}-Q})G(\lambda_{2p}, tx_0)_{|t=0}=
%\left(
%\begin{array}{cc}
%\nu_1\\
%\nu_2
%\end{array}
%\right)\overline{w}^{2p-3}
%+
%\left(
%\begin{array}{cc}
%\eta_1\\
%\eta_2
%\end{array}
%\right)\overline{w}^{4p-1}=0
%\end{eqnarray*}
%from the proof of (2) of proposition \ref{d2F2} with \eqref{notation3} and \eqref{notation4}.

	(1) We know $\partial_{\lambda\lambda\lambda}\varphi(\lambda_{2p},0)=0$ directly from Lemma \ref{phi2a=0}.
	
	(2) Now we compute $\partial_{\lambda\lambda}\partial_g\varphi(\lambda_{2p},0)x_0$ first. Recall that
	\begin{eqnarray*}
		\partial_{\lambda\lambda}\partial_g\varphi(\lambda_{2p},0)x_0=-2[\partial_f G(\lambda_{2p},0)]^{-1}({\rm {Id}}-Q)\partial_\lambda\partial_f G(\lambda_{2p},0)\tilde{x}.
	\end{eqnarray*}
	By Lemma \ref{lem1}, we have
	\begin{eqnarray*}
		({\rm {Id}}-Q)\partial_\lambda\partial_f G(\lambda_{2p},0)\tilde{x}
		=-2({\rm {Id}}-Q)
		\left(
		\begin{array}{cc}
			4b^{-1}\\
			0
		\end{array}
		\right)
		e_2
		=-4b^{-1}
		\left(
		\begin{array}{cc}
			1\\
			1
		\end{array}
		\right)e_2.
	\end{eqnarray*}
	Using the inverse relationship of $\partial_f G$ in \eqref{iG}, we have
	\begin{eqnarray*}
		\partial_{\lambda\lambda}\partial_g\varphi(\lambda_{2p},0)x_0
		=(-4b^{-1})(-b^{-2})
		\left(
		\begin{array}{cc}
			1\\
			1
		\end{array}
		\right)
		\overline{w}
		=4b^{-3}
		\left(
		\begin{array}{cc}
			1\\
			1
		\end{array}
		\right)
		\overline{w}.
	\end{eqnarray*}
	
(3) Then we compute $\partial_{\lambda}\partial_{gg}\varphi(\lambda_{2p},0)[x_0,x_0]$. Recall that
\begin{eqnarray*}
\begin{split}
\partial_{\lambda}\partial_{gg}\varphi(\lambda_{2p},0)[x_0,x_0]\!=&-2[\partial_f G(\lambda_{2p},0)]^{-1}({\rm{Id}}-Q)\partial_{ff}G(\lambda_{2p},0)[\partial_\lambda\partial_g\varphi(\lambda_{2p},0)\tilde{x},x_0].
\end{split}
\end{eqnarray*}
From Lemma \ref{lem2} and
$x_0=		
\left(
\begin{array}{cc}
	b\\
	1
\end{array}
\right)
\overline{w}$, we obtain $\partial_{ff}G(\lambda_{2p},0)[\partial_\lambda\partial_g\varphi(\lambda_{2p},0)\tilde{x},x_0]=0$ with $a_1=b$, $a_2=1$, and hence
 $	\partial_{\lambda}\partial_{gg} \varphi(\lambda_{2p},0)[x_0,x_0]=0.$
	
(4) Now we compute $\partial_{ggg}\varphi(\lambda_{2p},0)[x_0,x_0,x_0]$ to finish the proof. By Lemma \ref{phi2a=0}, we know
\begin{align*}
\partial_{ggg}\varphi(\lambda_{2p},0)[x_0,x_0,x_0]
=-[\partial_f G(\lambda_{2p},0)]^{-1}({\rm{Id}}-Q)\partial_{fff} G(\lambda_{2p},0)[x_0,x_0,x_0].
\end{align*}
Using Lemma \ref{lem2} with
$x_0=		
\left(
\begin{array}{cc}
	b\\
	1
\end{array}
\right)
\overline{w}$, we have
$\partial_{fff} G(\lambda_{2p},0)[x_0,x_0,x_0]=-6
\left(
		\begin{array}{cc}
			b^3-b\\
			b-b^{-1}
		\end{array}
		\right)
		e_2
$, and hence
	\begin{eqnarray*}
		\partial_{ggg}\varphi(\lambda_{2p},0)[x_0,x_0,x_0]
		=6[\partial_f G(\lambda_{2p},0)]^{-1}({\rm{Id}}-Q)
		\left(
		\begin{array}{cc}
			b^3-b\\
			b-b^{-1}
		\end{array}
		\right)
		e_2
		=
		3(b^{-3}-b)
		\left(
		\begin{array}{cc}
			1\\
			0
		\end{array}
		\right)
		\overline{w}.
	\end{eqnarray*}
\end{proof}

\subsubsection{The third-order derivatives of $F_2$}
By directly calculations with \eqref{E:Reduced-p}, using Lemma \ref{lem1}, we have the following proposition.
\begin{proposition}\label{acformal2a=0}
	The following assertions hold true.
	
(1) Formula for $\partial_{\lambda\lambda\lambda} F_2(\lambda_{2p},0)$.
\begin{eqnarray*}
\partial_{\lambda\lambda\lambda} F_2(\lambda_{2p},0)=3Q\partial_{\lambda}\partial_fG(\lambda_{2p},0)\partial_{\lambda\lambda}\partial_g\varphi(\lambda_{2p},0)x_0.
\end{eqnarray*}
	
(2) Formula for $\partial_{\lambda\lambda} \partial_t F_2(\lambda_{2p},0)$.
\begin{small}
\begin{align*}
\partial_{\lambda\lambda} \partial_t F_2(\lambda_{2p},0)\!=\!
Q\{2\partial_\lambda\partial_{ff}G(\lambda_{2p},0)[x_0,\!\tilde{x}]
\!+\!\partial_{ff}G(\lambda_{2p},0)[\partial_{\lambda\lambda}\partial_g\varphi(\lambda_{2p},0)x_0,x_0]
\!+\!\partial_{ff}G(\lambda_{2p},0)[\tilde{x},\tilde{x}]\}.
\end{align*}
\end{small}
	
(3) Formula for $\partial_\lambda \partial_{tt} F_2(\lambda_{2p},0)$.
\begin{small}
\begin{align*}
\partial_\lambda \partial_{tt} F_2(\lambda_{2p},0)\!=\!
\frac{1}{3}Q\partial_\lambda\partial_f G(\lambda_{2p},0)\partial_{ggg}\varphi(\lambda_{2p},0)[x_0,x_0,x_0]
\!+\!
Q\partial_{fff}G(\lambda_{2p},0)\left[\tilde{x},x_0,x_0\right].
\end{align*}
\end{small}

(4) Formula for $\partial_{ttt} F_2(\lambda_{2p},0)$.
\begin{eqnarray*}
\partial_{ttt} F_2(\lambda_{2p},0)\!=\!\frac{1}{4}Q\partial^4_{f}G(\lambda_{2p},0)\left[x_0,x_0,x_0,x_0\right]
\!+\!Q\partial_{ff}G(\lambda_{2p},0)\left\{\partial_{ggg}\varphi(\lambda_{2p},0)[x_0,x_0,x_0],x_0\right\}.
\end{eqnarray*}
\end{proposition}
Now we are in the position to calculate  third-order derivatives of $F_2$.
\begin{proposition}\label{P:d3F2a=0}
	Recall $y_1=\mathbb{W}_2$ and $y_2 =-\mathbb{W}_{2p}$. The following assertions hold true.
	
	(1) For $\partial_{\lambda\lambda\lambda} F_2(\lambda_{2p},0)$:
	\begin{eqnarray*}
		\partial_{\lambda\lambda\lambda} F_2(\lambda_{2p},0)
		=
		Q
		\left[24b^{-3}
		\left(
		\begin{array}{cc}
			1\\
			0
		\end{array}
		\right)e_2
		\right]=
		12\sqrt{2}b^{-3} y_1.
	\end{eqnarray*}

	(2) For $\partial_{\lambda\lambda} \partial_t F_2(\lambda_{2p},0)$, if $p\ne2$,
	\begin{eqnarray*}
		\partial_{\lambda\lambda} \partial_t F_2(\lambda_{2p},0)=0,
	\end{eqnarray*}
	and if $p=2$,
	\begin{eqnarray*}
		\partial_{\lambda\lambda} \partial_t F_2(\lambda_{2p},0)
		=Q\left[
		16
		\left(
		\begin{array}{cc}
			1\\
			0
		\end{array}
		\right) e_{4}\right]
		=
		-8\sqrt{2}y_2.
	\end{eqnarray*}

	(3) For $\partial_\lambda \partial_{tt} F_2(\lambda_{2p},0)$,
	\begin{eqnarray*}
		\partial_\lambda \partial_{tt} F_2(\lambda_{2p},0)=Q
		\left[
		\left(
		\begin{array}{cc}
			2b^{-3}+8b^{-1}-14b\\
			-4b^{-1}
		\end{array}
		\right)
		e_2\right]=
		\sqrt{2}(b^{-3}+6b^{-1}-7b) y_1.
	\end{eqnarray*}

	(4) For $\partial_{ttt} F_2(\lambda_{2p},0)$,
	\begin{eqnarray*}
		\partial_{ttt} F_2(\lambda_{2p},0)=0.
	\end{eqnarray*}
	
\end{proposition}

\begin{proof}
	(1) From Proposition \ref{d3phia=0}, we know
$
		\partial_{\lambda\lambda}\partial_g \varphi(\lambda_{2p},0)x_0
		=4b^{-3}
		\left(
		\begin{array}{cc}
			1\\
			0
		\end{array}
\right) \overline{w}.
$
So by Lemma \ref{lem1}, we have
	\begin{eqnarray*}
		\partial_{\lambda\lambda\lambda} F_2(\lambda_{2p},0)=3Q\partial_{\lambda}\partial_fG(\lambda_{2p},0)\partial_{\lambda\lambda}\partial_g\varphi(\lambda_{2p},0)x_0
		=
		Q\left[
        24b^{-3}
		\left(
		\begin{array}{cc}
			1\\
			0
		\end{array}
		\right)e_2\right]
=
		12\sqrt{2}b^{-3} y_1.
	\end{eqnarray*}

(2) From Proposition \ref{acformal2a=0}, we know
{\begin{small}
\begin{align*}
\partial_{\lambda\lambda} \partial_t F_2(\lambda_{2p},0)\!=\!
Q\{2\partial_\lambda\partial_{ff}G(\lambda_{2p},0)[x_0,\!\tilde{x}]
\!+\!\partial_{ff}G(\lambda_{2p},0)[\partial_{\lambda\lambda}\partial_g\varphi(\lambda_{2p},0)x_0,x_0]
\!+\!\partial_{ff}G(\lambda_{2p},0)[\tilde{x},\tilde{x}]\}.
\end{align*}
\end{small}}
Recall that
	$x_0=\left(
	\begin{array}{cc}
		b\\
		1
	\end{array}
	\right)
	\overline{w}$
	and
	$
	\tilde{x}
	=2b^{-1}
	\left(
	\begin{array}{cc}
		1\\
		0
	\end{array}
	\right)\overline{w}.
	$
So by Lemma \ref{lem1}, it is easy to see that $2\partial_\lambda\partial_{ff}G(\lambda_{2p},0)[x_0,\!\tilde{x}]=0$. From above, we know that
$
\partial_{\lambda\lambda}\partial_g \varphi(\lambda_{2p},0)x_0
=4b^{-3}
\left(
\begin{array}{cc}
1\\
0
\end{array}
\right)\overline{w},
$
then using Lemma \ref{lem2}, we can compute $\partial_{ff}G(\lambda_{2p},0)[\partial_{\lambda\lambda}\partial_g\varphi(\lambda_{2p},0)x_0,x_0]=0$ directly. Similarly, by Lemma \ref{lem2} we get
$\partial_{ff}G(\lambda_{2p},0)[\tilde{x},\tilde{x}]
		=16
		\left(
		\begin{array}{cc}
			1\\
			0
		\end{array}
		\right)
		e_4
$ and hence
\begin{align*}
\partial_{\lambda\lambda} \partial_t F_2(\lambda_{2p},0)
=
16Q\left[
\left(
\begin{array}{cc}
1\\
0
\end{array}
\right)
e_4
\right]
=\left\{
\begin{array}{cc}
0,&p\ne2,\\
- 8\sqrt{2} y_2,& p=2.
\end{array}
\right.
\end{align*}

(3) From Proposition \ref{acformal2a=0}, we know
\begin{small}
\begin{align*}
\partial_\lambda \partial_{tt} F_2(\lambda_{2p},0)\!=\!
\frac{1}{3}Q\partial_\lambda\partial_f G(\lambda_{2p},0)\partial_{ggg}\varphi(\lambda_{2p},0)[x_0,x_0,x_0]
\!+\!
Q\partial_{fff}G(\lambda_{2p},0)\left[\tilde{x},x_0,x_0\right].
\end{align*}
\end{small}
Recall that
	$
	\partial_{ggg} \varphi(\lambda_{2p},0)[x_0,x_0,x_0]
	=
	3(b^{-3}-b)
	\left(
	\begin{array}{cc}
		1\\
		0
	\end{array}
	\right)
	\overline{w}.
	$
	From Lemma \ref{lem1}, we have
	\begin{eqnarray}\label{ltt2}
		\frac{1}{3}Q\partial_\lambda\partial_f G(\lambda_{2p},0)\partial_{ggg}\varphi(\lambda_{2p},0)[x_0,x_0,x_0]
		=
		2(b^{-3}-b)Q
		\left[
		\left(
		\begin{array}{cc}
			1\\
			0
		\end{array}
		\right)e_2\right].
	\end{eqnarray}
Recall that
	$x_0=\left(
	\begin{array}{cc}
		b\\
		1
	\end{array}
	\right)
	\overline{w}$
	and
	$
	\tilde{x}
	=2b^{-1}
	\left(
	\begin{array}{cc}
		1\\
		0
	\end{array}
	\right)\overline{w}.
	$
Then by Lemma \ref{lem2},
\begin{eqnarray}\label{ltt3}
		Q\partial_{fff}G(\lambda_{2p},0)\left[\partial_\lambda\partial_{g}\varphi(\lambda_{2p},0)x_0,x_0,x_0\right]
		=Q
		\left[
		\left(
		\begin{array}{cc}
			8b^{-1}-12 b\\
			-4b^{-1}
		\end{array}
		\right)
		e_2\right].
	\end{eqnarray}
Combining \eqref{ltt2} and \eqref{ltt3}, we have
	\begin{align*}
		\partial_\lambda \partial_{tt} F_2(\lambda_{2p},0)=
		Q
		\left[
		\left(
		\begin{array}{cc}
			2b^{-3}+8b^{-1}-14b\\
			-4b^{-1}
		\end{array}
		\right)
		e_2\right]
		=
		\sqrt{2}(b^{-3}+6b^{-1}-7b) y_1.
	\end{align*}

	(4) From Proposition \ref{acformal2a=0}, we get
\begin{eqnarray*}
\partial_{ttt} F_2(\lambda_{2p},0)\!=\!\frac{1}{4}Q\partial^4_{f}G(\lambda_{2p},0)\left[x_0,x_0,x_0,x_0\right]
\!+\!Q\partial_{ff}G(\lambda_{2p},0)\left\{\partial_{ggg}\varphi(\lambda_{2p},0)[x_0,x_0,x_0],x_0\right\}\!.
\end{eqnarray*}
Since $x_0\!=\!\left(
	\begin{array}{cc}
		b\\
		1
	\end{array}
	\right)
	\overline{w}$ and
and
$
		\partial_{ggg} \varphi(\lambda_{2p},0)[x_0,x_0,x_0]
		\!=\!
		3(b^{-3}-b)
		\left(
		\begin{array}{cc}
			1\\
			0
		\end{array}
		\right)
		\overline{w},
$, then using Lemma \ref{lem2}, we have
$\frac{1}{4}\partial^4_{f}G(\lambda_{2p},0)\left[x_0,x_0,x_0,x_0\right]\!=\!0$ and $\partial_{ff}G(\lambda_{2p},0)\!\left\{\partial_{ggg}\varphi(\lambda_{2p},0)[x_0,x_0,x_0],x_0\right\}\!=\!0$.
So we have $\partial_{ttt} F_2(\lambda_{2p},0)=0$ and finish the proof.
\end{proof}

\subsection{The forth-order derivatives} \label{S:Fourth}
In this section, we only consider $p\ge 3$.

\subsubsection{The forth-order derivatives of $\varphi$.}\label{subS:forth}
By the similar methods to prove Lemma \ref{phi2a=0}, we have

\begin{lemma}\label{phi3a=0}
With $G$ defined in \eqref{E:Non}, we have
\begin{eqnarray*}
\partial_{\lambda\lambda\lambda\lambda}\varphi(\lambda_{2p},0)=0,
\end{eqnarray*}
\begin{align*}
\partial_{\lambda\lambda\lambda}\partial_g\varphi(\lambda_{2p},0)x_0
&=-3[\partial_f G(\lambda_{2p},0)]^{-1}({\rm{Id}}-Q)
\partial_{\lambda}\partial_fG(\lambda_{2p},0)\partial_{\lambda\lambda}\partial_g\varphi(\lambda_{2p},0)x_0\\
&=-[\partial_f G(\lambda_{2p},0)]^{-1}({\rm {Id}}-Q)Q^{-1}\partial_{\lambda\lambda\lambda} F_2(\lambda_{2p},0),
\end{align*}
\begin{align*}
\partial_{\lambda\lambda}\partial_{gg}\varphi(\lambda_{2p},0)[x_0,x_0]&=
-[\partial_f G(\lambda_{2p},0)]^{-1}({\rm {Id}}-Q)\big\{
4\partial_\lambda\partial_{ff}G(\lambda_{2p},0)[x_0,\tilde{x}]
\\
&+\!2\partial_{ff}G(\lambda_{2p},0)[\partial_{\lambda\lambda}\partial_g\varphi(\lambda_{2p},0)x_0,x_0]
+2\partial_{ff}G(\lambda_{2p},0)[\tilde{x},\tilde{x}]\big\}\\
&=-2[\partial_f G(\lambda_{2p},0)]^{-1}({\rm {Id}}-Q)Q^{-1}\partial_{\lambda\lambda} \partial_t F_2(\lambda_{2p},0),
\end{align*}
\begin{align*}
&\partial_{\lambda}\partial_{ggg}\varphi(\lambda_{2p},0)[x_0,x_0]\\
=&\!-\![\partial_f G(\lambda_{2p},0)]^{-1}({\rm {Id}}-Q)\big\{\partial_\lambda\partial_f G(\lambda_{2p},0)\partial_{ggg}\varphi(\lambda_{2p},0)[x_0,x_0,x_0]\!+\!
3\partial_{fff}G(\lambda_{2p},0)\left[\tilde{x},x_0,x_0\right]\big\}\\
=&-3[\partial_f G(\lambda_{2p},0)]^{-1}({\rm {Id}}-Q)Q^{-1}\partial_\lambda \partial_{tt} F_2(\lambda_{2p},0)
\end{align*}
and
\begin{align*}
\partial_{gggg}\varphi(\lambda_{2p},0)[x_0,x_0,x_0]=&
\!-[\partial_f G(\lambda_{2p},0)]^{-1}({\rm{Id}}-Q)
\big\{\partial_{ffff}G(\lambda_{2p},0)\left[x_0,x_0,x_0,x_0\right]\\
&+\!4\partial_{ff}G(\lambda_{2p},0)\left\{\partial_{ggg}\varphi(\lambda_{2p},0)[x_0,x_0,x_0],x_0\right\}\!\!\big\}\!\\
=&\!-4[\partial_f G(\lambda_{2p},0)]^{-1}({\rm{Id}}-Q)Q^{-1}\partial_{ttt}F_2(\lambda_{2p},0).
\end{align*}
\end{lemma}
So with the proof of Proposition \ref{P:d3F2a=0}, which is the calculations of the third-order derivatives of $F_2$, it is easy to see that
\begin{align*}
\partial_{\lambda\lambda\lambda}\partial_g\varphi(\lambda_{2p},0)x_0
&=-[\partial_f G(\lambda_{2p},0)]^{-1}({\rm {Id}}-Q)
\left[24b^{-3}
		\left(
		\begin{array}{cc}
			1\\
			0
		\end{array}
		\right)e_2
		\right]
\\
&=
-12b^{-3}[\partial_f G(\lambda_{2p},0)]^{-1}
\left[
		\left(
		\begin{array}{cc}
			1\\
			1
		\end{array}
		\right)e_2
		\right]
=12b^{-5}
\left(
		\begin{array}{cc}
			1\\
			0
		\end{array}
		\right)\overline{w}.
\end{align*}
\begin{align*}
\partial_{\lambda\lambda}\partial_{gg}\varphi(\lambda_{2p},0)[x_0,x_0]
=&-[\partial_f G(\lambda_{2p},0)]^{-1}({\rm {Id}}-Q)
\left[
		32
		\left(
		\begin{array}{cc}
			1\\
			0
		\end{array}
		\right) e_{4}\right]\\
%=&-[\partial_f G(\lambda_{2p},0)]^{-1}
%\left[
%		32
%		\left(
%		\begin{array}{cc}
%			1\\
%			0
%		\end{array}
%		\right) e_{4}\right]\\
=&-32M_4^{-1}(\lambda_{2p})
		\left(
		\begin{array}{cc}
			1\\
			0
		\end{array}
		\right) \overline{w}^3
%=&-32
%\left(
%\begin{array}{cc}
%1-2b^2&b^5\\
%-b^4& 2b^3-b
%\end{array}
%\right)^{-1}
%\left(
%\begin{array}{cc}
%1\\
%			0
%\end{array}
%\right) \overline{w}^3\\
%=&-\frac{32}{b(b^2-1)^2(b^4+2b^2-1)}
%\left(
%\begin{array}{cc}
%2b^3-b\\
%b^4
%\end{array}
%\right) \overline{w}^3\\
\triangleq
\left(
\begin{array}{cc}
\hat{\alpha}_1\\
\hat{\alpha}_2
\end{array}
\right) \overline{w}^3
,
\end{align*}
where
\[
\hat{\alpha}_1=-\frac{32(2b^3-b)}{b(b^2-1)^2(b^4+2b^2-1)},\,
\hat{\alpha}_2=-\frac{32b^4}{b(b^2-1)^2(b^4+2b^2-1)}.
\]

\begin{align*}
\partial_{\lambda}\partial_{ggg}\varphi(\lambda_{2p},0)[x_0,x_0]
=&-3[\partial_f G(\lambda_{2p},0)]^{-1}({\rm {Id}}-Q)
\left[
		\left(
		\begin{array}{cc}
			2b^{-3}+8b^{-1}-14b\\
			-4b^{-1}
		\end{array}
		\right)
		e_2\right]
\\=&
-3[\partial_f G(\lambda_{2p},0)]^{-1}
\left[
(b^{-3}+2b^{-1}-7b)
		\left(
		\begin{array}{cc}
			1\\
			1
		\end{array}
		\right)
		e_2\right]\\
=&
3(b^{-5}+2b^{-3}-7b^{-1})
		\left(
		\begin{array}{cc}
			1\\
			0
		\end{array}
		\right)
		\overline{w},
\end{align*}
and
\begin{align*}
\partial_{gggg}\varphi(\lambda_{2p},0)[x_0,x_0,x_0]=0.
\end{align*}

\subsubsection{The fourth-order derivatives of $F_2$.}
By directly calculations with \eqref{E:Reduced-p}, using Lemma \ref{lem1}, we have the following proposition.
\begin{proposition}\label{acformal4a=0}
Denote $\tilde{x}=\partial_{\lambda}\partial_{g}\varphi(\lambda_{2p},0)x_0$ and $\check{x}=\partial_{ggg}\varphi(\lambda_{2p},0)[x_0,x_0,x_0]$. The following assertions hold true.
	
(1) Formula for $\partial_{\lambda\lambda\lambda\lambda} F_2(\lambda_{2p},0)$.
\begin{eqnarray*}
\partial_{\lambda\lambda\lambda\lambda}F_2(\lambda_{2p},0)
=4Q\partial_{\lambda}\partial_fG(\lambda_{2p},0)\partial_{\lambda\lambda\lambda}\partial_g\varphi(\lambda_{2p},0)x_0.
\end{eqnarray*}

(2) Formula for $\partial_{\lambda\lambda\lambda}\partial_t F_2(\lambda_{2p},0)$.
\begin{align*}
\partial_{\lambda\lambda\lambda} \partial_{t} F_2(\lambda_{2p},0)=&
3Q\partial_\lambda\partial_{ff}G(\lambda_{2p},0)[\partial_{\lambda\lambda}\partial_{g}\varphi(\lambda_{2p},0)x_0,x_0]
\!+\!3Q\partial_\lambda\partial_{ff}G(\lambda_{2p},0)[\tilde{x},\tilde{x}]\\
+&\frac{3}{2}Q\partial_\lambda\partial_{f}G(\lambda_{2p},0)\partial_{\lambda\lambda}\partial_{gg}\varphi(\lambda_{2p},0)[x_0,x_0]
\!+\!3\partial_{ff}G(\lambda_{2p},0)[\partial_{\lambda\lambda}\partial_{g}\varphi(\lambda_{2p},0)x_0,\tilde{x}]\\
+&\partial_{ff}G(\lambda_{2p},0)[\partial_{\lambda\lambda\lambda}\partial_{g}\varphi(\lambda_{2p},0)x_0,x_0]
\end{align*}

(3) Formula for $\partial_{\lambda\lambda} \partial_{tt} F_2(\lambda_{2p},0)$.
\begin{align*}
\partial_{\lambda\lambda} \partial_{tt} F_2(\lambda_{2p},0)=&
\frac{2}{3}Q\partial_\lambda\partial_f G(\lambda_{2p},0)\partial_{\lambda}\partial_{ggg}\varphi(\lambda_{2p},0)[x_0,x_0,x_0]\\
+&
2Q\partial_{fff}G(\lambda_{2p},0)\left[\tilde{x},\tilde{x},x_0\right]
+Q\partial_{fff}G(\lambda_{2p},0)\left[\partial_{\lambda\lambda}\partial_{g}\varphi(\lambda_{2p},0)x_0,x_0,x_0\right]\\
+&Q\partial_{ff}G(\lambda_{2p},0)\left\{\partial_{\lambda\lambda}\partial_{gg}\varphi(\lambda_{2p},0)[x_0,x_0],x_0\right\}.
\end{align*}

(4) Formula for $\partial_{\lambda} \partial_{ttt} F_2(\lambda_{2p},0)$.
\begin{align*}
\partial_{\lambda} \partial_{ttt} F_2(\lambda_{2p},0)=&
Q\partial_\lambda\partial_{ff} G(\lambda_{2p},0)[\check{x},x_0]
+
Q\partial_{ffff}G(\lambda_{2p},0)\left[\tilde{x},x_0,x_0,x_0\right]\\
+&Q\partial_{ff}G(\lambda_{2p},0)[\check{x},\tilde{x}]
+Q\partial_{ff}G(\lambda_{2p},0)\left\{\partial_{\lambda}\partial_{ggg}\varphi(\lambda_{2p},0)[x_0,x_0,x_0],x_0\right\}.
\end{align*}

(5) Formula for $ \partial_{tttt} F_2(\lambda_{2p},0)$.
\begin{align*}
 \partial_{tttt} F_2(\lambda_{2p},0)=&
\frac{1}{5}Q\partial_{fffff}G(\lambda_{2p},0)[x_0,x_0,x_0,x_0,x_0]
+2Q\partial_{fff} G(\lambda_{2p},0)[\check{x},x_0,x_0].
\end{align*}
\end{proposition}
	
We will only compute the $e_{2p,p\ge3}$ terms of the forth-derivatives of $F_2$ to verify our idea and use $*$ to represents the constants that we need not to compute in the explicit form.
\begin{proposition}\label{P:d4F2a=0}	
(1) For $\partial_{\lambda\lambda\lambda}\partial_t F_2(\lambda_{2p},0)$,
\begin{align*}
\partial_{\lambda\lambda\lambda\lambda}F_2(\lambda_{2p},0)
=96b^{-5}
Q\left[
\left(
		\begin{array}{cc}
			1\\
			0
		\end{array}
		\right)e_2\right].
\end{align*}

(2) For $\partial_{\lambda\lambda\lambda}\partial_t F_2(\lambda_{2p},0)$,
\begin{align*}
\partial_{\lambda\lambda\lambda} \partial_{t} F_2(\lambda_{2p},0)=Q[*e_4].
\end{align*}

(3) For $\partial_{\lambda\lambda} \partial_{tt} F_2(\lambda_{2p},0)$,
\[
\partial_{\lambda\lambda} \partial_{tt} F_2(\lambda_{2p},0)
=Q\left[*e_2
+\left(
\begin{array}{cc}
48b\\
0
\end{array}
\right)e_6\
\right].
\]

(4) For $\partial_{\lambda} \partial_{ttt} F_2(\lambda_{2p},0)$,
\begin{align*}
\partial_{\lambda} \partial_{ttt} F_2(\lambda_{2p},0)=Q[*e_4].
\end{align*}

(5) For $ \partial_{tttt} F_2(\lambda_{2p},0)$. ,
\begin{align*}
 \partial_{tttt} F_2(\lambda_{2p},0)=Q[*e_2].
\end{align*}

Moreover, when $p=3$, we have
\begin{align*}
&\partial_{\lambda\lambda\lambda\lambda}F_2(\lambda_{2p},0)=*y_1,\,\, \partial_{tttt} F_2(\lambda_{2p},0)=*y_1,
\\
&\partial_{\lambda\lambda\lambda}\partial_t F_2(\lambda_{2p},0)=\partial_{\lambda} \partial_{ttt} F_2(\lambda_{2p},0)=0
\text{ and }\partial_{\lambda\lambda} \partial_{tt} F_2(\lambda_{2p},0)=-24\sqrt{2}b y_2.
\end{align*}
\end{proposition}

\begin{proof}
(1) Since
$
\partial_{\lambda\lambda\lambda}\partial_g\varphi(\lambda_{2p},0)x_0
=12b^{-5}
\left(
		\begin{array}{cc}
			1\\
			0
		\end{array}
		\right)\overline{w},
$
by Lemma \ref{lem1}, we have
\begin{align*}
\partial_{\lambda\lambda\lambda\lambda}F_2(\lambda_{2p},0)
=4Q\left[
\left(
		\begin{array}{cc}
			24b^{-5}\\
			0
		\end{array}
		\right)e_2\right]
=96b^{-5}
Q\left[
\left(
		\begin{array}{cc}
			1\\
			0
		\end{array}
		\right)e_2\right].
\end{align*}

(2) Notice
$
\partial_{\lambda\lambda}\partial_g \varphi(\lambda_{2p},0)x_0
=4b^{-3}
\left(
\begin{array}{cc}
1\\
0
\end{array}
\right)\overline{w},
$
$
\tilde{x}
=2b^{-1}
\left(
\begin{array}{cc}
1\\
0
\end{array}
\right)\overline{w},
$
$
\partial_{\lambda\lambda}\partial_{gg}\varphi(\lambda_{2p},0)x_0
=\left(
\begin{array}{cc}
\hat{\alpha}_1\\
\hat{\alpha}_2
\end{array}
\right)\overline{w}^3
$
and
$
\partial_{\lambda\lambda\lambda}\partial_g\varphi(\lambda_{2p},0)x_0
=12b^{-5}
\left(
		\begin{array}{cc}
			1\\
			0
		\end{array}
		\right)\overline{w}.
$
By Lemma \ref{lem1} and Lemma \ref{lem2}, it is easy to see that
\begin{align*}
&3Q\partial_\lambda\partial_{ff}G(\lambda_{2p},0)[\partial_{\lambda\lambda}\partial_{g}\varphi(\lambda_{2p},0)x_0,x_0]
=3Q\partial_\lambda\partial_{ff}G(\lambda_{2p},0)[\tilde{x},\tilde{x}]=0,\\
&\partial_{ff}G(\lambda_{2p},0)[\partial_{\lambda\lambda\lambda}\partial_{g}\varphi(\lambda_{2p},0)x_0,x_0]=0,
\end{align*}
and
\begin{align*}
\frac{3}{2}Q\partial_\lambda\partial_{f}G(\lambda_{2p},0)\partial_{\lambda\lambda}\partial_{gg}\varphi(\lambda_{2p},0)[x_0,x_0]=Q[*e_4],
3\partial_{ff}G(\lambda_{2p},0)[\partial_{\lambda\lambda}\partial_{g}\varphi(\lambda_{2p},0)x_0,\tilde{x}]=Q[*e_4],
\end{align*}
where we use $*$ to represents some constants that we need not to compute in the explicit form. So we obtain
$
\partial_{\lambda\lambda\lambda} \partial_{t} F_2(\lambda_{2p},0)=Q[*e_4]
$
by the formula of $\partial_{\lambda\lambda\lambda} \partial_{t} F_2(\lambda_{2p},0)$ in Proposition \ref{acformal4a=0}.

(3) For $\partial_{\lambda\lambda} \partial_{tt} F_2(\lambda_{2p},0)$, we want to show the non-vanish of the $e_6$ term for $p=3$, so we use $*$ to omit the value of the $e_2$ term.

Recall $
\partial_{\lambda}\partial_{ggg}\varphi(\lambda_{2p},0)[x_0,x_0]
=
3(b^{-5}+2b^{-3}-7b^{-1})
		\left(
		\begin{array}{cc}
			1\\
			0
		\end{array}
		\right)
		\overline{w}.
$
By Lemma \ref{lem1}, it is easy to see that $\frac{2}{3}\partial_\lambda\partial_f G(\lambda_{2p},0)\partial_{\lambda}\partial_{ggg}\varphi(\lambda_{2p},0)[x_0,x_0,x_0]=[*e_2]$. Since
$x_0
=
\left(
\begin{array}{cc}
b\\
1
\end{array}
\right)\overline{w}$,
$
\tilde{x}
=2b^{-1}
\left(
\begin{array}{cc}
1\\
0
\end{array}
\right)\overline{w},
$
$
\partial_{\lambda\lambda}\partial_{gg}\varphi(\lambda_{2p},0)x_0
=\left(
\begin{array}{cc}
\hat{\alpha}_1\\
\hat{\alpha}_2
\end{array}
\right)\overline{w}^3,
$
using Lemma \ref{lem2} and \ref{lem3}, we have
\begin{small}
\begin{align*}
&2\partial_{fff}G(\lambda_{2p},0)\left[\tilde{x},\tilde{x},x_0\right]
=\left[*e_2+48b
\left(
\begin{array}{cc}
1\\
0
\end{array}
\right)e_6\right],\\
&\partial_{fff}G(\lambda_{2p},0)\!\left[\partial_{\lambda\lambda}\partial_{g}\varphi(\lambda_{2p},0)x_0,x_0,x_0\right]=\left[*e_2\right],
\partial_{ff}G(\lambda_{2p},0)\left\{\partial_{\lambda\lambda}\partial_{gg}\varphi(\lambda_{2p},0)[x_0,x_0],x_0\right\}
=[*e_2],
\end{align*}
\end{small}
and hence by the formula of $\partial_{\lambda\lambda} \partial_{tt} F_2(\lambda_{2p},0)$ in Proposition \ref{acformal4a=0} we obtain
\[
\partial_{\lambda\lambda} \partial_{tt} F_2(\lambda_{2p},0)
=Q\left[*e_2
+\left(
\begin{array}{cc}
48b\\
0
\end{array}
\right)e_6\
\right].
\]

(4) Using Lemma \ref{lem2} with
$x_0
=
\left(
\begin{array}{cc}
b\\
1
\end{array}
\right)\overline{w}$,
$
\tilde{x}
=2b^{-1}
\left(
\begin{array}{cc}
1\\
0
\end{array}
\right)\overline{w}
$,
$
\check{x}=
		3(b^{-3}-b)
		\left(
		\begin{array}{cc}
			1\\
			0
		\end{array}
		\right)
		\overline{w}$
and
$
\partial_{\lambda}\partial_{ggg}\varphi(\lambda_{2p},0)[x_0,x_0]
=
3(b^{-5}+2b^{-3}-7b^{-1})
		\left(
		\begin{array}{cc}
			1\\
			0
		\end{array}
		\right)
		\overline{w}
$, we have
\begin{align*}
\partial_\lambda\partial_{ff} G(\lambda_{2p},0)[\check{x},x_0]=0,\,\,&\partial_{ffff}G(\lambda_{2p},0)\left[\tilde{x},x_0,x_0,x_0\right]=[*e_4],\\
\partial_{ff}G(\lambda_{2p},0)[\check{x},\tilde{x}]=[*e_4],\,\,
&\partial_{ff}G(\lambda_{2p},0)\left\{\partial_{\lambda}\partial_{ggg}\varphi(\lambda_{2p},0)[x_0,x_0,x_0],x_0\right\}=[*e_4].
\end{align*}
So by the formula of $ \partial_{\lambda}\partial_{ttt} F_2(\lambda_{2p},0)$ in Proposition \ref{acformal4a=0} we obtain
$
 \partial_{\lambda}\partial_{ttt} F_2(\lambda_{2p},0)=Q[*e_4].
$

(4)Using Lemma \ref{lem2} with
$x_0
=
\left(
\begin{array}{cc}
b\\
1
\end{array}
\right)\overline{w}$ and
$
\check{x}=
		3(b^{-3}-b)
		\left(
		\begin{array}{cc}
			1\\
			0
		\end{array}
		\right)
		\overline{w}$, we have
\begin{eqnarray*}
\partial_{fffff}G(\lambda_{2p},0)\left[x_0,x_0,x_0,x_0,x_0\right]=0, \,
2\partial_{fff}G(\lambda_{2p},0)\left[x_0,x_0,\check{x}\right]=[*e_2].
\end{eqnarray*}
So by the formula of $ \partial_{tttt} F_2(\lambda_{2p},0)$ in Proposition \ref{acformal4a=0} we obtain
$
 \partial_{tttt} F_2(\lambda_{2p},0)=Q[*e_2]
$ to finish our proof.
\end{proof}

\subsection{The fifth-order derivatives} \label{S:Fifth}
In this section, we only consider $p\ge 4$. Using the same methods in subsection \ref{subS:forth} with Lemma \ref{lem1} and the proof of Proposition \ref{P:d4F2a=0}, we have
\begin{eqnarray*}
\partial^5_{\lambda}\varphi(\lambda_{2p},0)=0,
\end{eqnarray*}
\begin{align*}
\partial^4_{\lambda}\partial_g\varphi(\lambda_{2p},0)x_0
&=-[\partial_f G(\lambda_{2p},0)]^{-1}({\rm {Id}}-Q)Q^{-1}\partial^4_{\lambda} F_2(\lambda_{2p},0)\\
&=-[\partial_f G(\lambda_{2p},0)]^{-1}({\rm {Id}}-Q)\left[96b^{-5}
\left(
		\begin{array}{cc}
			1\\
			0
		\end{array}
		\right)
		e_2\right]
=48b^{-7}
\left(
		\begin{array}{cc}
			1\\
			0
		\end{array}
		\right)
		\overline{w},
\end{align*}
\begin{align*}
\partial^3_{\lambda}\partial^2_{g}\varphi(\lambda_{2p},0)[x_0,x_0]
=-2[\partial_f G(\lambda_{2p},0)]^{-1}({\rm {Id}}-Q)Q^{-1}\partial^3_{\lambda} \partial_t F_2(\lambda_{2p},0)=*\overline{w}^3,
\end{align*}
\begin{align*}
\partial^2_{\lambda}\partial^3_{g}\varphi(\lambda_{2p},0)[x_0,x_0,x_0]
=-3[\partial_f G(\lambda_{2p},0)]^{-1}({\rm {Id}}-Q)Q^{-1}\partial^2_\lambda \partial^2_{t} F_2(\lambda_{2p},0)=*\overline{w}+*\overline{w}^5,
\end{align*}
\begin{align*}
\partial_\lambda\partial^4_{g}\varphi(\lambda_{2p},0)[x_0,x_0,x_0,x_0]=-4[\partial_f G(\lambda_{2p},0)]^{-1}({\rm{Id}}-Q)Q^{-1}\partial_\lambda\partial^3_{t}F_2(\lambda_{2p},0)=*\overline{w}^3,
\end{align*}
and
\begin{align*}
\partial^5_{g}\varphi(\lambda_{2p},0)[x_0,x_0,x_0,x_0,x_0]
=-5[\partial_f G(\lambda_{2p},0)]^{-1}({\rm {Id}}-Q)Q^{-1}\partial^5_{t} F_2(\lambda_{2p},0)=*\overline{w},
\end{align*}
where we use $*$ to represents the constants that we need not to compute in the explicit form.

By directly calculations with \eqref{E:Reduced-p}, using Lemma \ref{lem1}, we have the following proposition.
\begin{proposition}\label{acformal4a=0}
Denote $\tilde{x}=\partial_{\lambda}\partial_{g}\varphi(\lambda_{2p},0)x_0$. The following assertions .
	
(1) Formula for $\partial^5_{\lambda} F_2(\lambda_{2p},0)$.
\begin{eqnarray*}
\partial^5_{\lambda}F_2(\lambda_{2p},0)
=5Q\partial_{\lambda}\partial_fG(\lambda_{2p},0)\partial^4_{\lambda}\partial_g\varphi(\lambda_{2p},0)x_0.
\end{eqnarray*}

(2) Formula for $\partial^4_{\lambda}\partial_t F_2(\lambda_{2p},0)$.
\begin{small}
\begin{align*}
\partial^4_{\lambda} \partial_{t} F_2(\lambda_{2p},0)\!=&
2Q\partial_\lambda\partial_{f}G(\lambda_{2p},0)\partial^3_{\lambda}\partial^2_{g}\varphi(\lambda_{2p},0)[x_0,x_0]
+4\partial^2_{f}G(\lambda_{2p},0)[\partial^3_{\lambda}\partial_{g}\varphi(\lambda_{2p},0)x_0,\tilde{x}]\\
+&3\partial^2_{f}G(\lambda_{2p},0)[\partial^2_{\lambda}\partial_{g}\varphi(\lambda_{2p},0)x_0,\partial^2_{\lambda}\partial_{g}\varphi(\lambda_{2p},0)x_0]
\!+\!\partial^2_{f}G(\lambda_{2p},0)[\partial^4_{\lambda}\partial_{g}\varphi(\lambda_{2p},0)x_0,x_0].
\end{align*}
\end{small}

(3) Formula for $\partial^3_{\lambda} \partial^2_{t} F_2(\lambda_{2p},0)$.
\begin{small}
\begin{align*}
\partial^3_{\lambda} \partial^2_{t} F_2(\lambda_{2p},0)=&
3Q\partial_\lambda\partial^2_f G(\lambda_{2p},0)[\partial^2_{\lambda}\partial^2_{g}\varphi(\lambda_{2p},0)[x_0,x_0],x_0]
+2Q\partial^3_{f}G(\lambda_{2p},0)\left[\tilde{x},\tilde{x},\tilde{x}\right]\\
+&
Q\partial_\lambda\partial_f G(\lambda_{2p},0)\partial^2_{\lambda}\partial^3_{g}\varphi(\lambda_{2p},0)[x_0,x_0,x_0]
+6Q\partial^3_{f}G(\lambda_{2p},0)\left[\partial^2_{\lambda}\partial_{g}\varphi(\lambda_{2p},0)x_0,\tilde{x},x_0\right]\\
+&Q\partial^3_{f}G(\lambda_{2p},0)\left[\partial^3_{\lambda}\partial_{g}\varphi(\lambda_{2p},0)x_0,x_0,x_0\right]
+3Q\partial^2_{f}G(\lambda_{2p},0)\left[\partial^2_{\lambda}\partial^2_{g}\varphi(\lambda_{2p},0)[x_0,x_0],\tilde{x}\right]\\
+&Q\partial^2_{f}G(\lambda_{2p},0)\left[\partial^3_{\lambda}\partial^2_{g}\varphi(\lambda_{2p},0)[x_0,x_0],x_0\right].
\end{align*}
\end{small}

(4) Formula for $\partial^2_{\lambda} \partial^3_{t} F_2(\lambda_{2p},0)$.
\begin{align*}
&\partial^2_{\lambda} \partial^3_{t} F_2(\lambda_{2p},0)\\=&
\frac{1}{2}Q\partial_\lambda\partial_f G(\lambda_{2p},0)\partial_{\lambda}\partial^4_{g}\varphi(\lambda_{2p},0)[x_0,x_0,x_0,x_0]+
3Q\partial^4_{f} G(\lambda_{2p},0)[x_0,x_0,\tilde{x},\tilde{x}]\\
+&
Q\partial^4_{f}G(\lambda_{2p},0)\left[x_0,x_0,x_0, \partial^2_{\lambda}\partial_{g}\varphi(\lambda_{2p},0)x_0\right]
+2 Q\partial^3_{f}G(\lambda_{2p},0)\left[x_0,x_0, \partial^2_{\lambda}\partial^2_{g}\varphi(\lambda_{2p},0)[x_0,x_0]\right]\\
+&2Q\partial_{ff}G(\lambda_{2p},0)[\partial_{\lambda}\partial^3_{g}\varphi(\lambda_{2p},0)[x_0,x_0,x_0],\tilde{x}]
+Q\partial^2_{f}G(\lambda_{2p},0)\left[x_0, \partial^2_{\lambda}\partial^3_{g}\varphi(\lambda_{2p},0)[x_0,x_0,x_0]\right]\\
+&Q\partial_{ff}G(\lambda_{2p},0)\left[\partial^3_{g}\varphi(\lambda_{2p},0)[x_0,x_0,x_0],\partial^2_{\lambda}\partial_{g}\varphi(\lambda_{2p},0)x_0\right].
\end{align*}

(5) Formula for $\partial_\lambda\partial^4_{t} F_2(\lambda_{2p},0)$.
\begin{align*}
&\partial_\lambda\partial^4_{t} F_2(\lambda_{2p},0)\\=&
\frac{1}{5}Q\partial_\lambda\partial_{f}G(\lambda_{2p},0)\partial^5_{g}\varphi(\lambda_{2p},0)[x_0,x_0,x_0,x_0,x_0]
+\partial^5_{f}G(\lambda_{2p},0)[x_0,x_0,x_0,x_0,\tilde{x}]\\
+&4Q\partial^3_{f} G(\lambda_{2p},0)[x_0,\tilde{x}, \partial^3_{g}\varphi(\lambda_{2p},0)[x_0,x_0,x_0]]
+2Q\partial^3_{f} G(\lambda_{2p},0)[x_0,x_0, \partial_\lambda\partial^3_{g}\varphi(\lambda_{2p},0)[x_0,x_0,x_0]]\\
+&Q\partial^2_{f} G(\lambda_{2p},0)[x_0, \partial_\lambda\partial^4_{g}\varphi(\lambda_{2p},0)[x_0,x_0,x_0,x_0]].
\end{align*}

(vi) Formula for $\partial^5_{t} F_2(\lambda_{2p},0)$.
\begin{align*}
\partial^5_{t} F_2(\lambda_{2p},0)=&
\frac{1}{6}\partial^6_{f}G(\lambda_{2p},0)[x_0,x_0,x_0,x_0,x_0,x_0]
+\frac{10}{3}Q\partial^4_{f} G(\lambda_{2p},0)[x_0,x_0, \partial^3_{g}\varphi(\lambda_{2p},0)[x_0,x_0,x_0]]\\
+&\frac{5}{3}Q\partial^2_{f} G(\lambda_{2p},0)[\partial^3_{g}\varphi(\lambda_{2p},0)[x_0,x_0,x_0],\partial^3_{g}\varphi(\lambda_{2p},0)[x_0,x_0,x_0]]\\
+&Q\partial^2_{f} G(\lambda_{2p},0)[x_0, \partial^5_{g}\varphi(\lambda_{2p},0)[x_0,x_0,x_0,x_0,x_0]].
\end{align*}
\end{proposition}

With the value of the third-order derivatives of $\varphi$ in Proposition \ref{d3phia=0}, the forth-order derivatives of $\phi$ in Section \ref{subS:forth} and the form of the forth-order derivatives of $\phi$ in the above, by Lemma \ref{lem1}, Lemma \ref{lem2} and Lemma \ref{lem3} and a series calculations, we can obtain
\begin{eqnarray*}
\partial^5_{\lambda}F_2(\lambda_{2p},0)=Q[*e_2],\,
\partial^4_{\lambda} \partial_{t} F_2(\lambda_{2p},0)=Q[*e_4]=0,\,
\partial^3_{\lambda} \partial^2_{t} F_2(\lambda_{2p},0)=Q[*e_2+*e_6]=Q[*e_2],
\end{eqnarray*}
\begin{align*}
\partial^2_{\lambda} \partial^3_{t} F_2(\lambda_{2p},0)=Q\left[*e_4+192b^2
\left(
		\begin{array}{cc}
			1\\
			0
		\end{array}
		\right)e_8\right]
=Q\left[192b^2
\left(
		\begin{array}{cc}
			1\\
			0
		\end{array}
		\right)e_8\right],
\end{align*}
\begin{align*}
\partial_\lambda\partial^4_{t}F_2(\lambda_{2p},0)=Q[*e_2+*e_6+*e_{10}]=Q[*e_2+*e_{10}],\,\partial^5_{t} F_2(\lambda_{2p},0)=Q[*e_4]=0,
\end{align*}
where we omit the calculation process of the results that
\begin{align*}
&\partial^3_{f}G(\lambda_{2p},0)\left[x_0,x_0, \partial^2_{\lambda}\partial^2_{g}\varphi(\lambda_{2p},0)[x_0,x_0]\right]=*e_4, \\ &\partial^5_{f}G(\lambda_{2p},0)[x_0,x_0,x_0,x_0,\tilde{x}]=*e_6+*e_{10},\quad \partial^6_{f}G(\lambda_{2p},0)[x_0,x_0,x_0,x_0,x_0,x_0]=0,
\end{align*}
which can be obtained using the same method of Section \ref{S:Deriv-12}.

So when $p=4$, we have
\begin{align*}
&\partial^5_{\lambda}F_2(\lambda_{2p},0)=Q[*e_2],\,\,
\partial^3_{\lambda} \partial^2_{t} F_2(\lambda_{2p},0)=Q[*e_2],\,\, \partial_\lambda\partial^4_{t}F_2(\lambda_{2p},0)=Q[*e_2],\\
&
\partial^4_{\lambda} \partial_{t} F_2(\lambda_{2p},0)=\partial^5_{t} F_2(\lambda_{2p},0)=0,\,\,
\partial^2_{\lambda} \partial^3_{t} F_2(\lambda_{2p},0)
=Q\left[192b^2
\left(
		\begin{array}{cc}
			1\\
			0
		\end{array}
		\right)e_8\right]=-96 \sqrt{2}b^2 y_2.
\end{align*}

{\bf Acknowledgement:} The first author would like to thank Prof. Chongchun Zeng for helpful discussion and suggestion. Maolin Zhou is supported by the National Key Research and Development Program of China (2021YFA1002400), Nankai Zhide Foundation and National Science Foundation of China (No. 12271437, 11971498). Yuchen Wang is partially supported by
the National Science Foundation of China No. 11831009 and the funding of innovating activities in Science and Technology of Hubei Province.

\bibliographystyle{plain}
  \setlength{\bibsep}{0.0ex}
\bibliography{VP}

\begin{thebibliography}{10}

\bibitem{Ber1977}
Melvin~S. Berger.
\newblock {\em Nonlinearity and functional analysis}.
\newblock Pure and Applied Mathematics. Academic Press [Harcourt Brace
  Jovanovich, Publishers], New York-London, 1977.
\newblock Lectures on nonlinear problems in mathematical analysis.

\bibitem{Berti2023}
Massimiliano Berti, Zineb Hassainia, and Nader Masmoudi.
\newblock Time quasi-periodic vortex patches of {E}uler equation in the plane.
\newblock {\em Invent. Math.}, 233(3):1279--1391, 2023.

\bibitem{BC94}
A.~L. Bertozzi and P.~Constantin.
\newblock Global regularity for vortex patches.
\newblock {\em Comm. Math. Phys.}, 152(1):19--28, 1993.

\bibitem{Burbea1982}
Jacob Burbea.
\newblock Motions of vortex patches.
\newblock {\em Letters in Mathematical Physics. A Journal for the Rapid
  Dissemination of Short Contributions in the Field of Mathematical Physics},
  6(1):1--16, 1982.

\bibitem{Castro2016}
Angel Castro, Diego Córdoba, and Javier Gómez-Serrano.
\newblock Uniformly rotating analytic global patch solutions for active
  scalars.
\newblock {\em Ann. PDE}, 2(1), 2016.

\bibitem{Che1993}
Jean-Yves Chemin.
\newblock Persistance de structures géométriques dans les fluides
  incompressibles bidimensionnels.(french) [[persistence of geometric
  structures in two-dimensional incompressible fluids]].
\newblock {\em Ann. Sci. École Norm. Sup}, (4):517--542, 1993.

\bibitem{CR1971}
Michael Crandall, Rabinowitz, and H~Paul.
\newblock Bifurcation from simple eigenvalues.
\newblock {\em J. Functional Analysis}, 8:321--340, 1971.

\bibitem{Hmidi2016c}
Francisco de~la Hoz, Zineb Hassainia, and Taoufik Hmidi.
\newblock Doubly connected {V}-states for the generalized surface
  quasi-geostrophic equations.
\newblock {\em Arch. Ration. Mech. Anal.}, 220(3):1209--1281, 2016.

\bibitem{Hoz2016a}
Francisco de~la Hoz, Taoufik Hmidi, Joan Mateu, and Joan Verdera.
\newblock Doubly connected {$V$}-states for the planar {E}uler equations.
\newblock {\em SIAM Journal on Mathematical Analysis}, 48(3):1892--1928, 2016.

\bibitem{HMW2020}
Zineb Hassainia, Nader Masmoudi, and Miles~H. Wheeler.
\newblock Global bifurcation of rotating vortex patches.
\newblock {\em Communications on Pure and Applied Mathematics},
  73(9):1933--1980, 2020.

\bibitem{Hmidi2016}
Taoufik Hmidi and Joan Mateu.
\newblock Bifurcation of rotating patches from {K}irchhoff vortices.
\newblock {\em Discrete and Continuous Dynamical Systems. Series A},
  36(10):5401--5422, 2016.

\bibitem{Hmidi2016a}
Taoufik Hmidi and Joan Mateu.
\newblock Degenerate bifurcation of the rotating patches.
\newblock {\em Advances in Mathematics}, 302:799--850, 2016.

\bibitem{Hmidi2013}
Taoufik Hmidi, Joan Mateu, and Joan Verdera.
\newblock Boundary regularity of rotating vortex patches.
\newblock {\em Archive for Rational Mechanics and Analysis}, 209(1):171--208,
  2013.

\bibitem{Hmidi2017}
Taoufik Hmidi and Coralie Renault.
\newblock Existence of small loops in a bifurcation diagram near degenerate
  eigenvalues.
\newblock {\em Nonlinearity}, 30(10):3821--3852, 2017.

\bibitem{Kie2012}
Hansj\"{o}rg Kielh\"{o}fer.
\newblock {\em Bifurcation theory}, volume 156 of {\em Applied Mathematical
  Sciences}.
\newblock Springer, New York, second edition, 2012.
\newblock An introduction with applications to partial differential equations.

\bibitem{KisLuo2023}
Alexander Kiselev and Xiaoyutao Luo.
\newblock Illposedness of {$C^2$} vortex patches.
\newblock {\em Arch. Ration. Mech. Anal.}, 247(3):Paper No. 57, 49, 2023.

\bibitem{LP99}
Massimo Lanza~de Cristoforis and Luca Preciso.
\newblock On the analyticity of the cauchy integral in schauder spaces.
\newblock {\em J. Integral Equations Appl.}, 1999.

\bibitem{Cris2001}
Massimo Lanza~de Cristoforis and Sergei~V. Rogosin.
\newblock Analyticity of a nonlinear operator associated to the conformal
  representation of a doubly connected domain in {S}chauder spaces.
\newblock {\em Complex Variables Theory Appl.}, 44(3):193--223, 2001.

\bibitem{Liu2013}
Ping Liu, Junping Shi, and Yuwen Wang.
\newblock Bifurcation from a degenerate simple eigenvalue.
\newblock {\em J. Funct. Anal.}, 264(10):2269--2299, 2013.

\bibitem{LWZ2019}
Yiming Long, Yuchen Wang, and Chongchun Zeng.
\newblock Concentrated steady vorticities of the {E}uler equation on 2-d
  domains and their linear stability.
\newblock {\em J. Differential Equations}, 266(10):6661--6701, 2019.

\bibitem{MP1994}
Carlo Marchioro and Mario Pulvirenti.
\newblock {\em Mathematical theory of incompressible nonviscous fluids},
  volume~96 of {\em Applied Mathematical Sciences}.
\newblock Springer-Verlag, New York, 1994.

\bibitem{Roulley2023}
Emeric Roulley.
\newblock Vortex rigid motion in quasi-geostrophic shallow-water equations.
\newblock {\em Asymptot. Anal.}, 133(3):397--446, 2023.

\bibitem{Serfati1994}
Philippe Serfati.
\newblock Une preuve directe d'existence globale des vortex patches {$2$}{D}.
\newblock {\em C. R. Acad. Sci. Paris S\'{e}r. I Math.}, 318(6):515--518, 1994.

\bibitem{Yud1963}
V~Yudovich.
\newblock Non-stationary flows of an ideal incompressible fluid.
\newblock {\em Z. Vycisl. Mat. i. Mat. Fiz (Russian)}, 3:1032--1066, 1963.

\end{thebibliography}

\end{document}